\documentclass[10pt]{article}
\usepackage[leqno]{amsmath}
\usepackage{amssymb}
\usepackage{amsthm}
\usepackage{amsfonts}
\usepackage{graphicx}
\usepackage{enumerate}
\usepackage{mathrsfs}
\usepackage{bbm}
\usepackage{hyperref}
\usepackage{appendix}
\hypersetup{colorlinks=true, linkcolor=blue, urlcolor=blue,
linktocpage=true, implicit=false, pdfstartview=FitH,
pdfpagemode=None}

\newtheoremstyle{nonum}{}{}{\itshape}{}{\bfseries}{.}{ }{\thmnote{#3}}

\newtheorem{thm}{Theorem}[section]

\newtheorem{cor}[thm]{Corollary}
\newtheorem{lem}[thm]{Lemma}
\newtheorem{prop}[thm]{Proposition}

\theoremstyle{definition}
\newtheorem{exm}[thm]{Example}
\newtheorem{defn}[thm]{Definition}

\theoremstyle{nonum}

\theoremstyle{definition}
\newtheorem{rem}[thm]{Remark}

\newcommand{\iprod}[2]{\langle #1,#2 \rangle} 

\newcommand{\kn}{\mathcal{K}^n}   

\newcommand{\R}{\mathbb R}
\newcommand{\C}{\mathbb C}

\newcommand{\T}{\mathbb T}

\def\K{{\cal K}}
\def\A{{\cal A}}
\def\L{{\cal L}}
\def\I{{\cal I}}

\def\J{{\cal J}}
\def\T{{\cal T}}
\def\S{{\cal S}}

\def\cE{{\cal E}}

\begin{document}
\title{Order isomorphisms on convex functions in windows}
\author{S. Artstein-Avidan, D. Florentin and V. Milman
\thanks{The research was supported in part by Israel Science Foundation: first
and second named authors by grant No. 865/07, third named author by
grant No. 491/04.  The authors were also supported in part by BSF
grant No. 2006079.}}
\maketitle

\begin{abstract}
In this paper we give a characterization of all order isomorphisms
on some classes of convex functions. We deal with the class $Cvx(K)$
consisting of lower-semi-continuous convex functions defined on a
convex set $K$, and its subclass $Cvx_0(K)$ of non negative
functions attaining the value zero at the origin. We show that any
order isomorphism on these classes must be induced by a point map on
the epi-graphs of the functions, and determine the exact form of
this map. To this end we study convexity preserving maps on subsets
of $\R^n$, and also in this area we have some new interpretations,
and proofs. \end{abstract}
\section{Introduction}
In recent years, a big research project initiated by the first and
third named authors has been carried out, in which a
characterization of various transforms by their most simple and
basic properties has been found. Among these are the Fourier
transform (see \cite{AAM} and \cite{AAFM}), the Legendre transform
(see \cite{AM}), polarity for convex sets (see \cite{BS},
\cite{AM2}, \cite{Slo}), the derivative (see \cite{AKM}) and various
other transforms. This paper is part of this effort.

One example for such characterization is the understanding of
bijective order isomorphisms for certain partially ordered sets (see
Section \ref{Sect_Background} for definitions and details). In this
category one includes the Legendre transform, which is, up to linear
terms, the unique bijective order reversing map on the set of all
convex (lower-semi-continuous) functions on $\R^n$, with respect to
the pointwise order. In this paper we discuss analogous results
where the convex functions are defined on a ``window'', that is on a
convex set $K\subseteq\R^n$, and several other variants as well.

One main tool in the proof of such results is the fundamental
theorem of affine geometry, which states that an injective mapping
on $\R^n$ which sends lines to lines, and whose image is not
contained in a line, must be affine linear. When working with
``windows'', one immediately encounters a need for a similar theorem
for maps defined on a subset of $\R^n$. Such  theorems exist in the
literature, and a mapping which maps intervals to intervals on a
subset of $\R^n$ must be of a very specific form, which we call here
``fractional linear'', discussed in Section \ref{Sect_I.P.M}. A
remark about this name is in need: the mappings are of the form:
$\frac{Ax+b}{\langle c,x\rangle +d}$ for $A\in L_n(\R), b,c\in
\R^n$ and $d\in \R$, with some extra restriction. In the literature, the name ``fractional linear
maps'' sometimes refers to M\"{o}bius transformations, which is not
the case here (note that a M\"{o}bius transformation on a subset of
${\mathbb C} = \R^2$ does not preserve intervals, but these notions
are indeed connected - see Example \ref{example-mobious}). One could
name them ``permissable projective transformation'' but we prefer to
think about them exclusively in $\R^n$ and not on the projective
space. Another option was to call them ``convexity preserving
maps'', which describes their action rather than their functional
form, but this hides the fact that they are of a very simple form.

The classification of interval preserving transforms of a convex
subset of $\R^n$ is known (see \cite{Shiffman}). However, since this
is an essential part for our study of the order isomorphisms on
functions in windows, we dedicate the whole of Section
\ref{Sect_I.P.M} to this topic. We also provide some new insights
and results, and give a seemingly new geometric proof of the main
fact which is that such maps are fractional linear. Some parts of
this section are elementary and may be known to the reader, but we
include them as they too serve as intuition for the way these maps
behave.

In Sections \ref{Sect_Background}-\ref{Sect_Cvx0} we turn to the
main topic of this paper, namely characterization of order
preserving (and reversing) isomorphisms on classes of convex
functions defined on windows. Let $T\subseteq K$ be two closed
convex sets. The class of all lower-semi-continuous convex functions
$\{f:K \to\R\cup\{\infty\}\}$ is denoted $Cvx(K)$, and its subclass
of non negative functions satisfying $f(T)=0$ is denoted $Cvx_T(K)$.
In Section \ref{Sect_Background} we give background on general order
isomorphisms. In Sections \ref{Sect_Cvx}, \ref{Sect_Cvx0} we deal
with characterization of such transforms on $Cvx(K)$ and $Cvx_0(K)$
respectively. In both cases, the proof is based on finding a subset
of convex functions which are extremal, in some sense. The extremal
elements are relatively simple functions, which can be described by
a point in $\R^{n+1}$. We show that an order isomorphism is
determined by its action on the extremal family, and that its
restriction to this family must be a bijection. Therefore the
transform induces a bijective point map on a subset of $\R^{n+1}$.
By applying our uniqueness theorem, we show that this point map must
be fractional linear. We then discuss some generalizations of these
theorems to other classes of non-negative convex functions.

\section{Interval Preserving Maps}\label{Sect_I.P.M}

We start this section with a simple but curious fact which  is
stated again and proved as Theorem \ref{Thm_Polara-By-FL} below.
This fact demonstrates the idea that fractional linear maps should
be a key ingredient in convexity theory. Consider a convex body $K$
(actually any closed convex set will do) which includes the origin
and is included in the half-space $H_1 = \{x_1<1\}\subseteq\R^n$.
One may take its polar, defined by
\[ K^{\circ} = \{ y: \sup_{x\in K}\langle x, y \rangle \le 1\}. \]
Since polarity reverses the partial order of inclusion on closed
convex sets including the origin, $K^{\circ}$ includes the set
$[0,e_1]$ which is the polar of $H_1$. Translate it by $-e_1$ (so
now it includes $[-e_1, 0]$, and in particular includes the origin)
and then take its polar again. In other words, we constructed a map
mapping certain convex sets (which include $0$ and are included in
the half-space $H_1$) to convex sets, given by
\[F(K) =
(K^\circ - e_1)^\circ.\] Clearly this mapping is order preserving.

While polarity is a ``global'' operation, it turns out that this
mapping is actually induced by a {\rm point map} on $H_1$,
$\tilde{F}: H_1 \to \R^n$, which preserves intervals, and can be
explicitly written as $\tilde{F}(x)=\frac{x}{1-x_1}$. (This is a
simple calculation, and for completeness we provide it in the proof
of Theorem \ref{Thm_Polara-By-FL} below.)

This map is a special case of the so called ``fractional linear
maps'' which are the main topic of this section.

\subsection{Definition and simple observations}
\begin{defn} Let $D \subseteq \R^n$. A function $f:D\to\R^n$ is called
an interval preserving  map, if $f$ maps every interval
$[x,y]\subseteq D$ to an interval $[z,w]$.
\end{defn}

\begin{lem}\label{Lem_Edge-To-Edge} Let $D \subseteq \R^n$, and let
$f: D \to \R^n$ be an injective interval preserving map. Then for
every $x,y\in D$ with $[x,y]\subseteq D$ we have that $f([x,y]) =
[f(x), f(y)]$.
\end{lem}
\noindent{\bf Proof.} Indeed, assume that $f([x,y]) = [w,z]$, and
that, say, $f(y) \in (w,z)$. Pick a point $b\in (w, f(y))$, and a
point $b'\in (f(y), z)$. Then for some $a,a' \in (x,y)$, $b = f(a)$
and $b' = f(a')$. Consider $f([a,a'])$. It is an interval that
includes the points $b$ and $b'$, and therefore it includes $f(y)$,
whereas $y\not\in [a,a']$ - in contradiction to the injectivity of
$f$. \qed

\begin{lem}  Let $D \subseteq \R^n$, and let
$f: D \to \R^n$ be an injective interval preserving map. Then the
inverse $f^{-1}: f(D) \to D$ is interval preserving.
\end{lem}
\noindent{\bf Proof.} Let $I = [f(a),f(b)]$ be an interval in the
image, and $f(c)\in I$. From Lemma \ref{Lem_Edge-To-Edge}
$f([a,b])=I$, so by injectivity, $c\in [a, b]$. \qed

\begin{rem} Clearly, an interval preserving map $f$ must be
convexity preserving, i.e. $f$ must map every convex set $K$ to a
convex set $f(K)$. We will actually need the opposite direction,
given in the following lemma.
\end{rem}

\begin{lem}\label{Lem_Inv-is-CPM} Let $D\subseteq\R^n$ be a convex
set, and let $f:D\to \R^n$ be an injective interval preserving map.
Then the inverse image of a convex set in $f(D)$ is convex.
\end{lem}
\noindent{\bf Proof.} Let $K\subseteq f(D)$ be a convex set, and let
$x,y\in f^{-1}(K)$. We wish to show that $[x,y]\subseteq f^{-1}(K)$.
The interval $[x,y]$ is contained in the convex domain $D$, and by
Lemma 2.2 we know that \[ f([x,y]) = [f(x), f(y)]\subseteq K,\]
where the last inclusion is due to convexity of $K$. This implies
$[x,y]\subseteq f^{-1}(K)$. \qed

\begin{lem}\label{Lem_Int-Pres-Continuous} Let $D\subseteq\R^n$ be an
open domain and $f:D\to\R^n$ an injective interval preserving map,
then $f$ is continuous.
\end{lem}
\noindent{\bf Proof.} Let us prove that $f$ is continuous at a point
$x\in D$. We may assume that $D$ is convex (restrict $f$ to an open
convex neighborhood of $x$). Let $y=f(x)\in f(D)$ and $B_y$ an open
ball containing $y$. If $x\in int(f^{-1}(B_y))$ we are finished (we
have a neighborhood of $x$ that is mapped into $B_y$). We claim this
must be the case. Indeed, assume otherwise, then $x$ is on the
boundary of the set $f^{-1}(B_y)$, which by Lemma
\ref{Lem_Inv-is-CPM}, is convex. Let $x_{out} \in D$ such that $[x,
x_{out}]\cap f^{-1}(B_y)=\{x\}$. Then $f([x, x_{out}])$ is an
interval $I\subseteq f(D)$ such that $I\cap B_y=\{y\}$, but no such
interval exists, since $B_y$ is open. \qed

\subsection{Fractional linear maps}
Clearly, linear maps are interval preserving. It turns out that when
the domain of the map is contained in a half-space of $\R^n$, there
is a larger family of (injective) interval preserving maps. Indeed,
fix a scalar product $\iprod{\cdot}{\cdot}$ on $\R^n$, let $A\in
L_n(\R)$ be a linear map, $b,c \in \R^n$ two vectors and $d \in \R$
some constant, then the map
 \[ v \mapsto \frac{1}{\iprod{c}{v} + d} (Av + b)\]
is defined on the open half space $\iprod{c}{v} < - d$ and is
interval preserving. One can check interval preservation directly,
or deduce it from the projective description in Section
\ref{Sect_Proj-Desc}, as well as an injectivity argument. A
necessary and sufficient condition for this map to be injective is
that the associated matrix $\hat{A}$ (defined below) is invertible.
The matrix $A$ itself need not be invertible, for an example see
Remark \ref{Rem_example-not-inv}. We call these maps ``fractional
linear maps'', see the introduction for a remark about this name.

\subsubsection{Projective description}\label{Sect_Proj-Desc}
Consider the projective space $\R P^n = P(\R^{n+1})$, the set of
$1$-dimensional subspaces of $\R^{n+1}$. It is easily seen that
\[ \R P^n = \R^n \cup \R P^{n-1},\]
where one can geometrically think of $\R^n$ as $\R^n\times
\{1\}\subseteq \R^{n+1}$, so that each line which is not on the
hyperplane $e_{n+1}^{\perp}(\approx \R^n)$, intersects the shifted
copy of $\R^{n}$ at exactly one point. The lines which lie on
$e_{n+1}^{\perp}$ are thus lines in an $n$-dimensional subspace, and
are identified with $\R P^{n-1}$.

Any linear transformation on $\R^{n+1}$ induces a transformation on
$\R P^n$, mapping lines to lines. Thus it induces in particular a
map $F: \R^n \to \R^n \cup \R P^{n-1}$. It is easily checked that
the part mapped to $\R P^{n-1}$ is either empty - in which case the
induced transformation on $\R^n$ is linear, or an affine hyperplane
$H$ - in which case the induced transformation $F: \R^n \setminus
H\to \R^n$ is fractional linear.

Indeed, if the matrix associated with the original transformation in
${\cal L}(\R^{n+1}, \R^{n+1})$ is $\hat{A} \in GL_{n+1}$, then the
hyperplane in $\R^n$ mapped to $e_{n+1}^{\perp}$ is simply $\{ x\in
\R^n: (\hat{A}(x,1))_{n+1}= 0\}$. If $\hat{A}$ is given by
\[ \hat{A} = \left(
         \begin{array}{cc}
           A & b \\
           c^T & d \\
         \end{array}\right)\]
with $A\in M_{n\times n}$, $b,c\in\R^n$ and $d\in\R$, then the set
$\{x:\,\iprod{c}{x}+d \neq0\}\subset\R^n$ is exactly the pre image
of $\R P^{n-1}$ under $F$. It is an $n-1$ dimensional subspace if
$c\neq0$, and empty if $c=0$ ($\hat{A} \in GL_{n+1}$ implies
$d\neq0$ in that case).

Pick any vector $x\in \R^n$ which is not in this hyperplane, then it
is mapped to  $y = \hat{A}(x,1) \in \R^{n+1}$ which has a
non-vanishing  $(n+1)^{th}$ coordinate, $y_{n+1} = \iprod{c}{x}+d$.
Normalize $y \to y/y_{n+1}$, so that the last coordinate is $1$, and
consider only the first $n$ coordinates of this vector (we denote
the projection to the first $n$ coordinates by $P_n$). Thus, under
the above map, $x$ is mapped to \begin{equation}\label{eq:F}
 F(x) =
P_{n} \left(\hat{A} (x,1)/(\hat{A}(x,1))_{n+1}\right) = \frac{Ax +
b}{ \iprod{c}{x} + d }.\end{equation}

 Denote the domain of the map
by $D\subseteq\R^n$. Clearly, if $D=\R^n$, the condition $\phi(x) =
(\hat{A}(x,1))_{n+1}\neq 0$ for all $x\in D$ implies that this
(affine linear) function $\phi(x)$ is constant, which means that our
induced map is affine linear. However, when $D$ is contained in a
half space (for example, if $D$ is a convex set strictly contained
in $\R^n$), there are many choices of $\hat{A}$ which satisfy this
condition. Indeed, $(\hat{A}(x,1))_{n+1}= \iprod{c}{x} + d$ for some
$c\in \R^n$ and $d \in \R$ ($(c,d)$ is the $(n+1)^{th}$ row of
$\hat{A}$), and the condition is that \[\forall x \in D \qquad
\iprod{c}{x} \neq -d,\] which can be satisfied for appropriate
chosen $c$ and $d$; for every direction $c$ in which $D$ is bounded,
there exists a critical $d$ such that from this value onwards the
condition is satisfied (the critical $d$ may or may not be chosen,
depending on the boundary of $D$). Other ways of describing these
maps will be given in Section \ref{Sect_More-Reps}.

\noindent{\bf Notation.} For future reference we denote the map $F$
associated with a matrix $\hat{A}$ by $F_A$ and the matrix $\hat{A}$
associated with a map $F$ by $A_F$. Note that $\hat{A}$ is defined
uniquely up to a multiplicative constant. We say that $A_F$ induces
$F$, and may also write $F\sim A_F$.

\begin{exm}\label{example-mobious} We are, in fact, very much
familiar with one class of projective transformations: M\"{o}bius
transformations of the extended complex plane. These are just
projective transformations of the complex projective line $P^1(\C)$
to itself. We describe points in $P^1(\C)$ by homogeneous
coordinates $[z_0, z_1]$, and then a projective transformation
$\tau$ is given by $\tau ([z_0, z_1]) = [az_0 + bz_1, cz_0 + dz_1]$
where $ad - bc \not= 0$. This corresponds to the invertible linear
transformation
\[ T = \left(
         \begin{array}{cc}
           a & b \\
           c & d \\
         \end{array}
       \right).\]

It is convenient to write $P^1(\C) = \C\cup\{+\infty\}$ where the
point $\{+\infty\}$ is now the 1-dimensional space $z_1 = 0$. Then
if $z_1 \not= 0$, $[z_0, z_1] = [z, 1]$ and $\tau ([z, 1]) = [az +
b, cz + d]$ and if $cz + d \not= 0$ we can write $\tau ([z, 1]) = [
(az + b)/(cz + d) , 1]$ which is the usual form of a M\"{o}bius
transformation, i.e. \[ z \mapsto \frac{ az + b}{ cz + d}.\]

The advantage of projective geometry is that the point $1 = [1, 0]$
plays no special role. If $cz + d = 0$ we can still write $\tau ([z,
1]) = [az + b, cz + d] = [az + b, 0] = [1, 0]$ and if $z = 1$ (i.e.
$[z_0, z_1] = [1, 0]$) then we have $\tau([1, 0]) = [a, c]$. In this
note, however, we work over $\R$ and these transformations when
considered as acting over $\R^2$ do not preserve intervals.
\end{exm}

\begin{exm} If we view the real projective plane $P^2(\R)$ in the
same way, we get some less familiar transformations. Write $P^2(\R)
= \R^2 \cup P^1(\R)$ where the projective line at infinity is $z =
0$. A linear transformation $T:\R^3 \to \R^3$ can then be written as
the matrix \[ T =\left(\begin{array}{ccc}
                        a_{11} & a_{12} & b_{1}\\
                        a_{21} & a_{22} & b_{2}\\
                        c_{1}  & c_{2}  & d    \\
                        \end{array}\right),\]
and its action on $[x, y, 1]$ can be expressed, with $v = (x, y) \in
\R^2$, as \[ v \to \frac{ 1}{\iprod{c}{v} + d } (Av + b)\] where $A$
is the matrix $\{a_{ij}\}$ and $b, c$ are the vectors $(b_1, b_2)$,
$(c_1, c_2)$.  Each such transformation can be considered as a
composition of an invertible linear transformation, a translation
and an inversion $v \to  v/(\iprod{c}{v} + d)$. Clearly it is easier
here to consider projective transformations defined by $3\times 3$
matrices.
\end{exm}

\begin{rem}\label{Rem_example-not-inv} Consider the matrix
\[\hat{A} =\left(\begin{array}{ccc}
                  1 & 0 & 0\\
                  0 & 0 & 1\\
                  0 & 1 & 0\\
                 \end{array}\right).\] It gives rise to the
transformation \[(x,y)\mapsto
\left(\frac{x}{y},\frac{1}{y}\right),\] which is injective (where it
is defined). The upper $2\times 2$ block (or $n\times n$, in the
general case) is not invertible, though. We will get back to this
transformation in later sections.
\end{rem}

\subsubsection{Basic properties}\label{Sect_Basic-Prop}
\begin{enumerate}[]
\item 1. {\em Preservation of intervals.} It is very easy to check
that the map $F$ defined above in (\ref{eq:F}) preserves intervals.
Indeed, an interval in $\R^n$ is a subset of a line, which
corresponds to a two dimensional plane in $\R^{n+1}$. The latter is
mapped by  $A_F$ to a two dimensional plane, and after the radial
projection to the level $x_{n+1}=1$ we again get a line.

\item 2. {\em Maximal domain.} A non affine fractional linear
map $F$ can be extended to a half space. The only restriction is
that for $x \in D$ one has $\iprod{c}{x} \neq -d$, that is, $D$
cannot intersect some given affine hyperplane $H$. Since we are
interested in a convex domain, we must choose one side, which means
the domain can be extended to a half space. It is not immediately
clear why it cannot be extended further. To see why it cannot be
extended further {\em while preserving intervals}, consider a point
$x_0 \in H$. We shall see that there is no way to define $F$ on
$x_0$. Indeed, take two rays emanating from $x_0$ into the domain of
$F$, say $H^{+}$ (and not on $H$); $\{x_0 + \lambda y: \lambda>0\},
\{x_0+\lambda y': \lambda>0\}$. The fact that the rays are not on
$H$ means that $\iprod{c}{y}\neq0$, likewise for $y'$. Moreover,
$\iprod{c}{y}$ and $\iprod{c}{y'}$ have the same sign (say,
positive, if we are in $H^+$). Assume $F(x) =
\frac{Ax+b}{\iprod{c}{x} + d}$. Remember $\iprod{c}{x_0} = -d$. Then
$F(x_0+\lambda y)= \frac{A(x_0+\lambda y)+b}{\iprod{c}{x_0 + \lambda
y} + d}=\frac {Ay}{\iprod{c}{y}} + \frac{1}
{\lambda}\cdot\left(\frac{Ax_0+b} {\iprod{c}{y}}\right)$, and
similarly $F(x_0+\lambda y')= \frac {Ay'}{\iprod{c}{y'}} +
\frac{1}{\lambda}\cdot\left(\frac{Ax_0+b} {\iprod{c}{y'}}\right)$.
We see the two rays are mapped to two parallel half lines, which by
injectivity of $F$ are not identical, and therefore $F(x_0)$ cannot
be chosen so that it lies on both these lines. This means $F$ cannot
be extended to a domain which intersects $H$, and still preserve
intervals.

\item 3. {\em The image}. The image of a (non-affine) fractional
linear map $F$, whose domain is maximal (meaning it is an open half
space) is an open half space. Indeed, let $\hat{A}=A_F$ be the
associated matrix, and let $\hat{A}(\{ (x,1): x\in \R^n\}) = E
\subseteq \R^{n+1}$. Then the image of $F$ is the radial projection
into $\{ (x,1): x\in \R^n\}$ of the part of $E$ with positive
$(n+1)^{th}$ coordinate. It is easily checked that this is a half
space in $\{ (x,1): x\in \R^n\}\sim\R^n$, whose boundary is the
hyperplane \[\partial (Im(F)) = \{Ax:\iprod{c}{x}=1\}, \quad
\mbox{where}\quad \hat{A} = \left(\begin{array}{cc}
                                   A & b \\
                                   c^T & d \\
                                  \end{array}\right).\] Note that it
does {\em not} depend on $b$ and $d$.

\item 4. {\em Composition.} It is easily checked that $A_{F\circ G}
= A_F \cdot A_G$. In particular, the composition of two fractional
linear maps is again a fractional linear map. As for the domains:
The maximal domain of each of the maps is a half space, and so is
the image, thus the map is formally defined only on
$G^{-1}(Im(G)\cap dom(F))$, and by the previous remarks it can be
extended to be defined on some half space.

\item 5. {\em The inverse map.} It is easily checked that
$A_{F^{-1}} = A_{F}^{-1}$ and in particular, every fractional linear
map has an inverse, which is also fractional linear. The domain of
$F^{-1}$ is the image of $F$, which by previous remarks is exactly
the radial projection into $\{ (x,1): x\in \R^n \}$ of the part of
$E = A(\{ (x,1): x\in \R^n \})$ with, say, positive $(n+1)^{th}$
coordinate.
\end{enumerate}

\subsubsection{More properties}
To continue we first need two properties of fractional linear maps,
given in Lemma \ref{Lem_Transitivity-n+2} and Lemma
\ref{Lem_Mizdahim-Interior}. The first is a transitivity result.

\begin{lem}\label{Lem_Transitivity-n+2} Fix a point $p$ in the
interior of the simplex $\Delta = \{ z=\sum z_ie_i: 0\le z_i,
\sum_{i=1}^nz_i\le 1\}$, where $\{e_i\}_{i=1}^n$ is the standard
basis of $\R^n$. Given $n+2$ points, $x_0, x_1, \ldots, x_{n}, y$ in
$\R^n$ such that $y$ is in the interior of $conv(x_i)_{i=0}^{n}$,
there exists an open convex domain $D$ which contains the points and
a fractional linear map $F: D \to \R^n$ such that for $1\le i\le n$,
$F(x_i) = e_i$, $F(x_{0}) = 0$ and $F(y) = p$.
\end{lem}

\begin{rem} By invertibility (see Item 5.~above) an equivalent
formulation is as follows: there exists a fractional linear map $F:
\Delta \to \R^n$ such that for $1\le i\le n$, $F(e_i) = x_i$, $F(0)
=x_{0}$ and $F(p) = y$.
\end{rem}

\begin{rem} From Lemma \ref{Lem_Mizdahim-n+2}, it will follow that
the map in Lemma \ref{Lem_Transitivity-n+2} is unique.
\end{rem}

\begin{rem}\label{Rem_Transitivity-n+2-Classic} Let us compare Lemma
\ref{Lem_Transitivity-n+2} with the more standard transitivity
result of projective geometry, which can be found for example in
\cite{Prasolov-Tichomirov}
(Theorem 2, page 59):\\
Let $A_1, \ldots, A_{n+2}$ and $B_1, \ldots, B_{n+2}$ be two sets of
points in general position in $\R P^n$. Then there exists a unique
projective transformation $f: \R P^n \to \R P^n$ such that $f(A_i) =
B_i$ for $i = 1, \ldots, n+2$. Indeed, they have the same flavor,
however we demand more (in both sets, one point is in the convex
hull of all the others) and get more; the whole convex hull is in
the domain of the fractional linear map (i.e. it is mapped, within
$\R P^n=\R^n\cup\R P^{n-1}$, to the part not in $\R P^{n-1}$).
\end{rem}

\noindent{\bf Proof of Lemma \ref{Lem_Transitivity-n+2}.} First let
us build an affine linear map which maps $x_i$ to $e_i$ for $i = 1,
\ldots, n$ and $x_{0}$ to $0$. This is clearly possible by linear
algebra. So we are left with the following task: given $z$ in the
interior of the simplex, build a fractional linear map $F$ whose
domain contains the simplex, such that $F(e_i) = e_i$, $F(0) = 0$
and $F(z) = p$.

To describe this map $F$, consider its associated matrix $\hat{A}$ in
$GL_{n+1}$. Let us give the matrix elements which produce the
desired map. Let the matrix be given by \[ \hat{A} =
\left(\begin{array}{cccc}
&   &  &   0   \\
& A &  & \vdots\\
&   &  &   0   \\
& c^T &  &   d   \\
\end{array} \right),\] where $A$ is an $n\times n$ matrix, $c\in
\R^n$, and $d\in \R^+$. Let $A$ be the diagonal matrix with diagonal
entries $A_{i,i} = \frac{p_i}{z_i}$, let $d = \frac{1- \sum
p_i}{1-\sum z_i}$, and let the vector $c$ be given by $c_i =
\frac{p_i}{z_i}- d$. The matrix induces a fractional linear map on
the domain $\{ x: \iprod{c}{x} > -d \}$. We must verify that the
points $0, \{e_i\}_{i=1}^n$ are in this domain. Indeed, $d>0$ since
the points are in the simplex, and also $c_i > -d$. Finally, it is
easily checked that the associated fractional linear map satisfies
the desired conditions. \qed

Once we know that the map from Lemma \ref{Lem_Transitivity-n+2}
exists, it follows that it is unique. Indeed, by the Theorem  quoted
in Remark \ref{Rem_Transitivity-n+2-Classic}, there exists only one
fractional linear map which maps $n+2$ given points to another $n+2$
given points. We formulate it below, and for completeness provide
the proof.

\begin{lem}\label{Lem_Mizdahim-n+2} Let $F_1: D_1 \to \R^n$ and
$F_2: D_2 \to \R^n$ be two fractional linear maps, where
$D_i\subseteq \R^n$. Let $\{x_i\}_{i=0}^{n+1}$ be $(n+2)$ points in
$D_1\cap D_2$ such that one is in the interior of the convex hull of
the others. If $F_1(x_i) = F_2(x_i)$ for every $0\le i\le n+1$, then
the two maps coincide on all of $D_1 \cap D_2$ and moreover, are
induced by the same matrix in $GL_{n+1}$ (up to multiplication by a
non-zero scalar).
\end{lem}

\noindent{\bf Proof of Lemma \ref{Lem_Mizdahim-n+2}.} Without loss
of generality, by Lemma \ref{Lem_Transitivity-n+2}, we can assume
that $x_{0} = 0$, $x_i= e_i$ for $i= 1, \ldots, n$, and that
$x_{n+1} = p$ is any point we desire in the interior of the convex
hull of $\{x_i\}_{i=1}^{n}$. Furthermore, by the same lemma, we may
assume that $F_1(x_i) = F_2(x_i) = x_i$ for all $i$, and therefore
we may simply compare, say, $F_1$ to $Id$ (and then also $F_2$).
Consider the matrix which induces $F_1$, given by some
 \[ \left(\begin{array}{cccc}
                   &   &  &   \\
                   & A &  & b \\
                   &   &  &   \\
                   & c^T &  & d \\
                  \end{array} \right),\] where $A$ is an $n\times n$
matrix, $b$ and $c$ are vectors in $\R^n$ and $d \in \R$. In fact,
$d\neq0$ since $0$ is in the domain of $F_1$, which is
$\{x:\iprod{c}{x} > -d\}.$ Without loss of generality we let $d =
1$. From the condition $F_1(0) = 0$ we see that $b=0$. From
$F_1(e_i) = e_i$ we see that $A$ is diagonal, let us denote $A_{i,i}
= a_i$, so that $c_i = a_i - 1$. Finally, for $F_1(p) = p$ we see
that
\[ p_i = \frac{a_ip_i}{1 + \sum (a_j-1) p_j}.\]
This implies that for all $i$, $a_i = 1 + \sum (a_j-1) p_j$, and in
particular $a_1=\ldots=a_n$. This means that $(a_1-1)(1-\sum p_j) =
0$, and since $p$ is not on the hyperplane passing through
$\{e_i\}_{i=1}^n$, it implies $a_i = 1$ for all $i$, that is, $F_1$
is the identity mapping. The same holds for $F_2$. \qed

As a consequence we get the following useful fact:
\begin{cor}\label{Lem_Mizdahim-Interior} Let $F_1: D_1 \to \R^n$ and
$F_2: D_2 \to \R^n$ be two fractional linear maps, where
$D_i\subseteq \R^n$. Let $D\subseteq D_1 \cap D_2$ be some open
domain in $\R^n$ such that $F_1|_D = F_2|_D$. Then the two maps
coincide, i.e. they are induced by the same matrix, and their
maximal extension is the same function, with the same maximal
domain.
\end{cor}

\subsection{Uniqueness}
When the domain of an interval preserving map is assumed to be all
of $\R^n$, it is a well known classical theorem that the map must be affine
linear, as stated in the fundamental theorem of affine geometry,
quoted below as Theorem \ref{Thm_FTAG-classical}. As a reference
see, for example, \cite{Prasolov-Tichomirov}, or \cite{Artin} for
the projective counterpart. More generally, interval preservation
can be replaced by ``collineation''. More far reaching
generalizations also exist, and we refer the reader to the
forthcoming \cite{AS} where an elaborate account of these is given.

\begin{thm}\label{Thm_FTAG-classical}[The Fundamental Theorem of
Affine Geometry] Let $m\ge 2$ and $f:\R^n\to\R^m$ be a bijective
interval preserving map. Then $f$ must be an affine transformation.
\end{thm}

In this section we discuss the fact that when the domain is a convex
set (or, more generally, a connected open domain), the only interval preserving maps
are fractional linear. This result was obtained by Shiffman in
\cite{Shiffman}. His method of proof is different from ours, and
works in the projective setting, where he shows that any such map
can be extended (using Desargues' theorem) to a mapping of the whole
projective space, and then, from the fundamental theorem of
projective geometry he concludes that it must be projective linear.
We work in a more elementary way, never leaving $\R^n$. However,
Shiffman's result is in a more general setting where not {\em all}
intervals are assumed to be mapped to intervals, but only a
subfamily which is large enough. This is important in some
applications, in particular in the proof of Theorem
\ref{Thm_CvxTK-Preserving}.

The main theorem discussed in this section is the following.

\begin{thm}\label{Thm_FL-Unique-nD} Let $n\ge2$ and let $K\subseteq\R^n$
be a convex set with non empty interior. If $F:K\to\R^n$ is an
injective interval preserving map, then $F$ is a fractional linear
map.
\end{thm}

The proof of the theorem relies on the following lemma:

\begin{lem}\label{Lem_Id-On-Simplex} Assume $n\ge2$. Let
$\Delta\subseteq U\subseteq\R^n$, where $U$ is an open set, and
$\Delta$ is a non-degenerate simplex with vertices $x_0,\dots,x_n$.
Let $p$ belong to the interior of $\Delta$. If $F: U \to\R^n$ is an
injective interval preserving map that fixes all $(n+1)$ vertices of
$\Delta$ and the interior point, that is $F(x_i)=x_i$ for every
$0\le i\le n$, and $F(p) = p$, then $F|_\Delta =Id|_\Delta$.
\end{lem}
%

\noindent{\bf Proof of Lemma \ref{Lem_Id-On-Simplex}.} The proof
goes by induction on the dimension $n$. Begin with $n=2$. Consider a
two dimensional simplex $\Delta$, that is, a triangle in $\R^2$,
with vertices $a,b,c,$ and a point $p\in int(\Delta)$. Since $F$ is
injective and interval preserving, by Lemma
\ref{Lem_Int-Pres-Continuous} it is continuous, which implies that
the set $D=\{x\in\Delta : F(x)=x\}$ is closed.

Let us check that all the edges are contained in $D$. Assume the
contrary, namely that there is a point $e\in [a,b]$, $e\not\in D$.
Since $D$ is closed, there exists an interval $[a',b']\subseteq
[a,b]$, such that $a',b'\in D$, but $(a',b')\cap D=\emptyset$. Now
we will find a point $e'\in (a',b')\cap D$ in contradiction, thus
concluding that no such $e$ exists. Let us find two points $a''\in
[a',c]$ and $b''\in [b',c]$, such that $a'',b''\in D$. To this end,
consider the intervals $[a',c]$, $[b',c]$. They are both mapped to
themselves by $F$, and both intersect the line $L$ containing $a$
and $p$, for which we have $F(L)\subseteq L$. Let $a''\in [a',c]$
and $b''\in [b',c]$ be the points of intersection with $L$. Then
$F(a'')=a''$ since this is the only point in $[a',c]$ and in $L$,
and similarly $F(b'')=b''$.
\begin{figure}[h]
\begin{center}
\scalebox{0.25}{\includegraphics{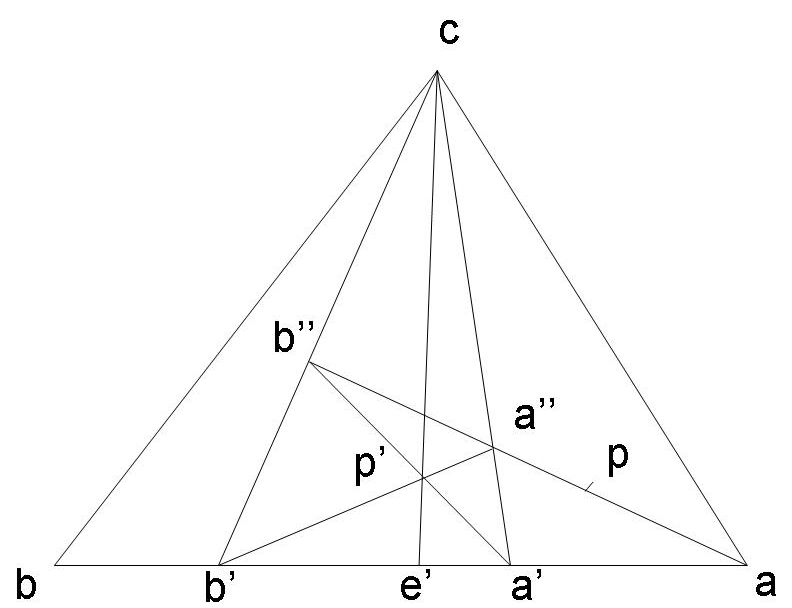}}
\end{center}
\end{figure}

Now we look at the intersection of $[a'',b']$ with $[b'',a']$. This
is a point $p'$ in the interior of the triangle $a'b'c$. The line
between $c$ and $p'$ intersects with $[a',b']$ at some point $e'\in
(a',b')$, and by the same argument as before, $e'\in D$. We get a
contradiction which proves that $F$ is the identity map on the edges
of $\Delta$.

Next, for every point $y$ in the interior, we draw two intervals
containing $y$ - each connecting a vertex with an edge, and get that
the two intervals must be mapped to themselves (since the end points
are on the edges and are thus mapped to themselves). This implies,
as before, $F(y)=y$, which completes the proof for $n=2$.

For the inductive step, we assume that the proposition is true for
dimension $n-1$, and prove it for dimension $n$. Let $\Delta$ be an
$n$ dimensional simplex. Denote by
$\Delta_i:=Conv\{x_0,\dots,x_{i-1}, x_{i+1},\dots,x_n\}$ the face of
$\Delta$ opposite to $x_i$. First we claim that $F(\Delta_i) =
\Delta_i$. Indeed, this is due to interval preservation, together
with the fact that the vertices are mapped to themselves. Denote by
$y\in relint(\Delta_i)$ the unique point in the intersection of
$\Delta_i$, with the line connecting $x_i$ and $p\in int(\Delta)$.
Interval preservation implies that $F(y)$ remains on this line, and
since it must remain on the face, we get $F(y)=y$. By applying the
claim to the ($n-1$) dimensional simplex $\Delta_i$, we conclude
that $F|_{\Delta_i}=Id|_{\Delta_i}$. The fact that the restriction
of $F$ to each of the faces is the identity, combined with interval
preservation, implies that $F|_\Delta=Id|_\Delta$ simply by
representing a point in the interior as the intersection of two
intervals with endpoints on faces. \qed

By the transitivity result from Lemma \ref{Lem_Transitivity-n+2}, we
may state a corollary of the above lemma for general maps on the
simplex.

\begin{cor}\label{Cor_genFL-On-Simplex}
Assume $n\ge2$. Let $\Delta\subseteq U\subseteq\R^n$, where $U$ is
an open set, and $\Delta$ is a non-degenerate simplex with vertices
$x_0,\dots,x_n$. If $F:U \to\R^n$ is an injective interval
preserving map 
then there exists a fractional linear map $F_A$ such that $F|_\Delta
={F_A}|_\Delta$.
\end{cor}

\noindent{\bf Proof of Corollary \ref{Cor_genFL-On-Simplex}.} Let
$p$ belong to the interior of $\Delta$. The main step is to show
that the mapping $F$ maps the point $p$ to a point in the interior
of $conv\{F(x_i)\}_{i=0}^n$, so that we may invoke transitivity and
Lemma \ref{Lem_Id-On-Simplex}. To this end we shall use induction
and prove the following claim: an injective interval preserving map
must map simplices of dimension $k$, for any $k\ge 1$, to simplices
of the same dimension, whose vertices are the images of the original
vertices. Once this is done, an interior point must be mapped to an
interior point by injectivity of $F$. The case $k=1$ is almost by
definition (see Lemma \ref{Lem_Edge-To-Edge}). Assume this is the
case for simplices of dimension $\le k$ and let $y_0, \ldots
y_{k+1}$, the vertices of some $(k+1)$ dimensional simplex, in
general position, be given. By induction, the relative boundary of
the convex hull is mapped to the relative boundary of the simplex
$\{F(y_j)\}_{j=0}^{k+1}$. Since a point in the interior can be
written as the intersection of two intervals with endpoints on the
boundary, we get that the interior of the simplex
$conv\{  y_j \}_{j=0}^{k+1}$ is mapped to the interior of
$conv\{F(y_j)\}_{j=0}^{k+1}$, as needed. Applying this, we have that
the points $\{x_i\}_{i=0}^n$ are mapped to points
$\{F(x_i)\}_{i=0}^n$ which are the vertices of a non degenerate
simplex, $F(\Delta) = conv \{F(x_i)\}_{i=0}^n=:\Delta'$, and for any
point $p\in int(\Delta)$ we have that $F(p) \in int(\Delta')$. To
prove the corollary, chose any $p\in \Delta$, and compose $F$ with
some fractional linear $G$ so that $(G\circ F)(x_i) = x_i$ for $i=0,
\ldots n$ and $(G\circ F)(p) = p$. Using Lemma
\ref{Lem_Id-On-Simplex} we have that $G\circ F = Id$ on $\Delta$,
and therefore $F|_{\Delta} = G^{-1}|_{\Delta}$, which is fractional
linear, as claimed. \qed

\noindent{\bf Proof of Theorem \ref{Thm_FL-Unique-nD}.} First we
prove the theorem under the assumption that $K$ is open and convex,
and at the end of the proof we remark on the extension to general
convex $K$ (with non empty interior).

First we note that for every simplex $\Delta$ inside $K$ the
statement holds: consider $n+2$ points $x_0,\dots,x_n,p\in\R^n$,
arranged as a simplex $\Delta$ and a point in its interior, as in
Corollary \ref{Cor_genFL-On-Simplex}. Since $F$ is injective and
interval preserving, by Corollary \ref{Cor_genFL-On-Simplex}
$F|_{\Delta}$ is fractional linear.

Next, consider the union of two simplices $\Delta_1$ and $\Delta_2$
such that the intersection has a non empty interior. $F|_{\Delta_i}$
is fractional linear on each simplex $\Delta_1$ and $\Delta_2$, and
these mappings coincide on the intersection, so they must be induced
by the same matrix, by Corollary \ref{Lem_Mizdahim-Interior}.

Finally, by covering the domain $K$ with simplices so that each two
are connected by a chain of simplices $\{\Delta_i\}_{i=0}^N$, with
the property that the intersection of $\Delta_i$ and $\Delta_{i+1}$
has a non empty interior, we get that there is one map which induces
all of the maps $F|_{\Delta}$ for all these simplices, meaning that
$F$ itself is a fractional linear map. Such a covering exists, for
example an infinite family $\{\Delta_{x,y}:x,y\in K \}$, where
$\Delta_{x,y}$ is some simplex  which contains $x$ and $y$ in the
interior will do (such a simplex exists for every $x$ and $y$). This
completes the proof in the case where $K$ is open.

For a general convex $K$ with non empty interior we must deal with
the boundary of $K$. We know there exists a fractional linear map
$G:U\to\R^n$ s.t. $F|_{int (K)}=G|_{int(K)}$, where $U$ is the
maximal domain of $G$ (an open half space), and of course
$int(K)\subseteq U$. We wish to show that $K\subseteq U$, and that
$F=G$ also on $K \cap \partial K$. Take $x\in K \cap\partial K$. We
first claim that $x\in U$, for which we need only show that
$x\not\in H=\partial U$. However, we have shown in item 2 of Section
\ref{Sect_Basic-Prop} that $G$ cannot be extended to be defined on
any point of $H$ so that it is still interval preserving, from which
we conclude $K \subseteq U$. Indeed, this was shown by considering
two points $a,b$ in the interior of $K$, to which correspond
intervals $[a,x)$ and $[b,x)$ which are mapped to intervals, by $G$.
Were $x$ on the boundary, these intervals would have been parallel,
and no way to define $F(x)$ would have existed. When $x\not\in H$,
the intervals $[G(a), G(x)]$ and $[G(b), G(x)]$ have a unique point
of intersection $G(x)$, and we conclude that $F(x) = G(x)$. \qed

\begin{rem}\label{Rem_FL-Uniq-Open-Conn} Theorem
\ref{Thm_FL-Unique-nD} can be proved for a general open connected
set $K$; we only used convexity of $K$ when arguing that $K$ can be
covered by simplices to get the wanted chains. This argument holds
also whenever $K$ is open and connected. Indeed, to get this
covering we took between every two points $x,y\in K$ a simplex
$\Delta_{x,y}$. This simplex is now replaced by a chain of simplices
connecting $x$ and $y$, constructed using an $\epsilon$ neighborhood
of the path between $x$ and $y$.
\end{rem}

To complete the picture let us also attend to the case $n=1$,
although this will not be used in the sequel. Obviously, a similar
theorem cannot be proved in $\R$, since, for example, all continuous
functions are interval preserving. The next theorem, Theorem
\ref{Thm_FL-Unique-1D}, gives a characterization of one dimensional
fractional linear maps. The theorem is a local version of the more
well-known fact from projective geometry, stating that maps
preserving cross ratio are linear when the domain and range are
lines, and projective when the domain and range are extended lines.

We recall that the {\em cross ratio} of four numbers (thought of as
coordinates of points on a line) is defined to be \[ [a,b,c,d] :=
\left(\frac{c-a}{c-b}\right)/\left(\frac{d-a}{d-b}\right).\] For
details and discussion see, for example, \cite{Prasolov-Tichomirov}.

\begin{rem} \label{Rem_Preserve-Cross-Ratio} Note that
$[a,b,c,x] = [a',b',c',x']$ implies $x' = \frac{\alpha x +
\beta}{\gamma x + \delta}$, where $\alpha,\beta,\gamma,\delta$ are
some function of $a,b,c,a',b',c'$. Conversely, every fractional
linear map on $\R$ preserves the cross ratio of any four points in
its domain.
\end{rem}

\begin{rem}\label{Rem_Cross-Ratio-Permut} Regarding permutations
of $a,b,c,d$, we have the following:
\[ [A,B,c,d] =     [B,A,c,d]^{-1},\]
\[ [a,b,C,D] =     [a,b,D,C]^{-1},\]
\[ [a,B,C,d] = 1 - [a,C,B,d]     ,\] and using the rule for these
three transpositions, the cross ratio of any permutation of
$a,b,c,d$ can be derived from $[a,b,c,d]$. Moreover, as a
consequence, we see that if we have $[a,b,c,d] = [x,y,z,w]$, then
for every permutation $\sigma$ we also have that
$[\sigma(a),\sigma(b), \sigma(c), \sigma(d)]=[\sigma(x),\sigma(y),
\sigma(z), \sigma(w)]$.
\end{rem}

A basic notion when dealing with one dimensional fractional linear
maps is the projection of one line to another line, through a so
called ``focus point'' situated outside the two lines. See
\cite{Prasolov-Tichomirov} for more details on the relation between
fractional linear maps, preservation of cross ratio, and projection.
\begin{thm}\label{Thm_FL-Unique-1D} Let $I\subseteq\R$ be a convex
set, either bounded or not, and $f:I\to\R$. Assume further that $f$
preserves cross ratio on $I$, so for every four distinct points
$a<b<c<d\in I$
\[ [f(a), f(b), f(c), f(d)] = [a,b,c,d].\]
Then $f$ is fractional linear on $I$. In fact, it is true also if
$a,b,c\in I$ are three (distinct) fixed points, and we assume only
that $f$ preserves cross ratio of $a,b,c,d$ for any $d\in
I\setminus\{a,b,c\}$.
\end{thm}
\noindent{\bf Proof.} Let $a,b,c\in I$ such that $a<b<c$, and $f$
preserves cross ratio of $a,b,c,x$ for any $x\in I\setminus
\{a,b,c\}$. Let $x\in I$. We consider four cases; $x<a$, $a<x<b$,
$b<x<c$, and $c<x$. For each case, the preservation of cross ratio
yields a different equation;

$\hskip 21pt x<a \Rightarrow [f(x), f(a), f(b), f(c)] = [x,a,b,c],$

$ a<x<b          \Rightarrow [f(a), f(x), f(b), f(c)] = [a,x,b,c],$

$b<x<c\hskip  1pt\Rightarrow [f(a), f(b), f(x), f(c)] = [a,b,x,c],$

$ c<x \hskip 22pt\Rightarrow [f(a), f(b), f(c), f(x)] = [a,b,c,x].$

By Remark \ref{Rem_Cross-Ratio-Permut}, each of these equations
implies $[f(a), f(b), f(c), f(x)] = [a,b,c,x]$, and thus by Remark
\ref{Rem_Preserve-Cross-Ratio}, we get $f(x)=\frac{\alpha x +
\beta}{\gamma x + \delta}$ for some $\alpha, \beta, \gamma, \delta$
which depend only on $a, b, c$, $f(a), f(b), f(c)$. Therefore $f$ is
a fractional linear map on $I$. \qed

\subsection{Other representations and properties}\label{Sect_More-Reps}
\subsubsection{Canonical form}
In what follows, we denote by $x=(x_1,\dots,x_n)$ the coordinates of
a point $x$ with respect to the standard basis $\{e_i\}$.
\begin{defn} Let $H^+$ be the half space $\{x_1>1\}$.
The mapping $F_0:H^+\to H^+$ given by
\[F_0(x)=\frac{x}{x_1-1}\]
will be called the {\em canonical}               fractional linear map.
\end{defn}
It is useful to note that the group of fractional linear maps is
generated by its subgroup of affine linear maps, and the above map.
\begin{thm}\label{Thm_Can-Form} Let $F$ be an injective non-affine
fractional linear map with $F(x_0) = y_0$. Then there exist $B, C\in
GL_n$ such that $B(F(Cx+x_0)-y_0)=F_0(x)$.
\end{thm}

\noindent{\bf Proof of Theorem \ref{Thm_Can-Form}.} Define
$G(x):=F(x+x_0)-y_0$, then $G(0)=0$. $G$ is an injective non-affine
fractional linear map, with an inducing matrix of the form:
\[\left(\begin{array}{cccc}
         &   &  &   \\
         & A'&  & b \\
         &   &  &   \\
         & c^T &  & d \\
        \end{array}\right).\] From $0\in Dom(G)$ it follows
$d\neq0$, so (using the multiplicative degree of freedom) we let
$d=-1$. Also, $G(0)=0$ implies $b=0$. Since $G$ is injective, the
inducing matrix is invertible, and by $b=0$ this implies that $A'\in
GL_n$. Non-linearity of $G$ implies $c\neq 0$. Therefore we can
write for some $A'\in GL_n$, $0\neq c\in\R^n$, that
\[G(x) = \frac{A'x}{\iprod{c}{x} - 1}.\] Pick $C \in GL_n$ such that
$C^{t}c = e_1$. We get $\iprod{c}{Cx} = \iprod{e_1}{x} = x_1$.
Therefore \[G(Cx) =\frac{A'Cx}{x_1 -1}.\] Finally, by letting $B =
(A'C)^{-1}$, we get $(B\circ G\circ C)(x) = \frac{x}{x_1-1}$, and so
\[ B(F(Cx+x_0)-y_0) =\frac{x}{x_1-1},\] as required. \qed

\begin{rem} For simplicity, assume below $x_0=y_0=0$. The
representation in Theorem \ref{Thm_Can-Form} is clearly not unique,
as $C$ can be chosen in any way satisfying just one linear
condition, and $B$ depends on $C$. Another form which can be given
is:
\[C^{-1}A'^{-1}FC = \frac{x}{x_1-1},\] where $A'$ is uniquely
determined, and $C$ as before. Yet a third way to view this
representation is: \[ F(x) = \frac{A'x}{\iprod{c}{x} - 1},\] as was
shown in the proof. This form has the advantage of emphasizing the
degrees of freedom of a fractional linear map, since both the point
$c$ and the matrix $A'$ are determined uniquely.
\end{rem}

\subsubsection{Geometric structure}\label{Sect_Geom_Desc}
The mapping $F_0(x) = \frac{x}{x_1-1}$ is defined on
$H^+=\{x_1>1\}$, and satisfies $F_0(H^+)=H^+$. It is an involution
on $H^+$ (and on $H^-=\{x_1<1\}$ as well). Denote the boundary of
$H^+$ by $H$.

For every affine hyperplane parallel to $H$, namely $H_t=\{x:
x_1=t\}$ (for $t\neq 1$), we have $F_0(H_t)=H_{f(t)}$, where
$f(t)=\frac{t}{t-1}$. The restriction $F_0: H_t\to H_{f(t)}$,
thought of as a map on $\R^{n-1}$, is a linear map - in fact, it is
simply a scalar map; $x\mapsto \frac{1}{t-1} x$. In particular we
see that in this family of parallel hyperplanes (shifts of $H$),
parallel hyperplanes are mapped to parallel hyperplanes. This
behavior is unique to shifts of $H$. Indeed, take $v\in\R^n$, then
($F_0^{-1}=F_0$):
\begin{align*}
F_0(\{x: \iprod{x}{v}=c\}) &= \{F_0(x): \iprod{x}{v}=c\}=\{ x:
\iprod{F_0(x)}{v}=c\}\\&=\{x: \iprod{x}{v}=c(x_1-1)\}=\{x:
\iprod{x}{v}=c\iprod{x}{e_1}-c\}\\&=\{x: \iprod{x}{v-ce_1}=-c\}.
\end{align*}
And so we see that if $v\neq\lambda e_1$, hyperplanes parallel to
$v^\perp$ are mapped to hyperplanes which are {\em not} parallel;
$(v-ce_1)^\perp\neq (v-c'e_1)^\perp$ for $c\neq c'$.

These considerations, by Theorem \ref{Thm_Can-Form}, may be applied
to a general fractional linear mapping $F$. There are two
hyperplanes, the first of which, say $H_1$, is the boundary of the
maximal domain of $F$, and the second, $H_2$, is the boundary of the
image of $F$, such that any translate of $H_1$ (which is in the
domain) is mapped to a translate of $H_2$, and moreover, the map $F$
restricted to each translate of $H_1$ is linear. In any other
direction, however, two parallel hyperplanes are mapped to two
hyperplanes which are not parallel.

As for a linear subspace $V$ of $\R^n$ of dimension $0\le k\le n$,
we have $F_0(V)=V$ (by this we mean $F_0(V\cap H^-)=V\cap H^-$ and
$F_0(V\cap H^+)=V\cap H^+$, since $F_0$ is not defined on the
intersection with $H$). For $n-1$ dimensional subspaces, we have
seen it in the formula given above for the image of hyperplanes
under $F_0$; substituting $c=0$ yields $F_0(v^\perp)=v^\perp$ for
every $v\in\R^n$. But in fact it is true trivially for subspaces of
any dimension; simply note that $F_0(x)$ is in the direction of $x$.
In fact, this is a particular case of the more general phenomenon;
lines (more precisely: their intersection with the domain) through a
fixed point in the domain of $F$, $x_0 \in dom(F)$, are mapped into
lines through $F(x_0)$. This is due to interval preservation of $F$.
Since $F$ is smooth, this mapping of lines (but not of points along
the lines) is the linear map given by the differential of $F$,
$dF(x_0)$.

We can say even more about the geometric structure of $F$. For a
point $y_0$ on the boundary of the maximal domain of $F$, the family
of all the rays emanating from the  point $y_0$ (into the domain) is
mapped to the family of all half lines in the image of $F$ which are
parallel to some vector $y_0'$, and vice versa. Again, by Theorem
\ref{Thm_Can-Form} it is enough to show this for the specific map
$F_0(x) = \frac{x}{x_1 - 1}$. Consider a point $\hat{y}=(1,y)$ on
$H$; a ray emanating from $\hat{y}$ {\em into the domain} can be
written, for some $(1,u)\in H$ as
\[ R = \{(1,y) + t(1,u): t\in \R^+\}. \]
It is mapped to the half line\[l' = \{ F( (1,y) + t(1,u)): t\in
\R^+\} = \{{(1,u)}+\frac{1}{t}(1,y): t\in \R^+\} = \{(1,u)+s(1,y) :
s\in \R^+\}.\] So we have seen that for $a$ and $b$ on $H$, the ray
$a+b\R^+$ is mapped under $F_0$ to $b + a \R^+$ and vice versa. For
example all rays emanating from the point $e_1\in H$ are mapped to
all lines perpendicular to $H$. Note that the part of $l$ which is
close to the point $\hat{y}$ (small $t$) is mapped to the part of
$l'$ which is far from the hyperplane $H$ (large $s$). In a sense,
the point $\hat{y}$ is mapped to ``infinity'' in direction opposite
to $H$.

This also shows that fractional linear maps act as a lens on
straight lines intersecting the defining hyperplane. Indeed, a cone
of rays with base $B$, emanating from the point $a$ in $H$, is
mapped to a half infinite cylinder with base $B$, in the direction
$a$. If $a\in B$, the corresponding line is the only one in the cone
which is mapped to itself. When considering a general non-affine
fractional linear map, we get that an infinite cone with base $B$ is
mapped to a half infinite cylinder with base $T(B)$ for some linear
$T$, and vice versa. Of course, if the fractional linear map is
affine it also does this, but by mapping cones to themselves and
cylinders to themselves.
\begin{figure}[h]
\begin{center}
\scalebox{0.24}{\includegraphics{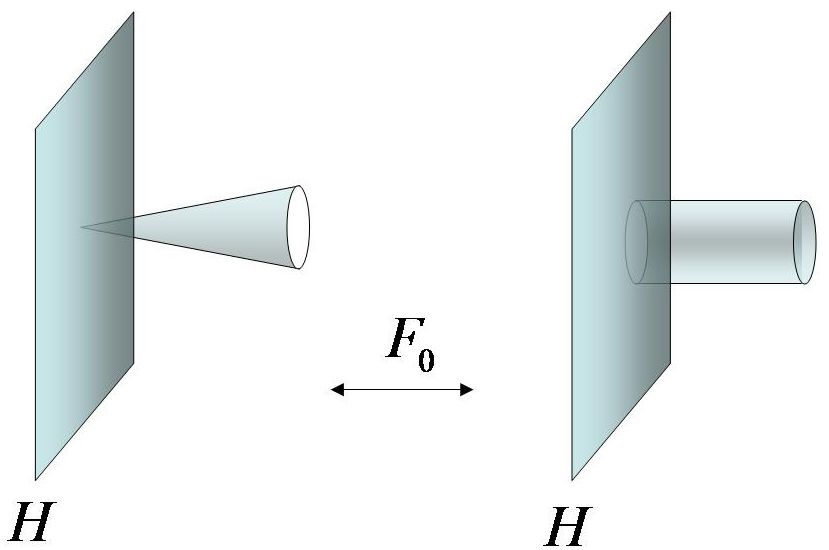}}
\end{center}
\end{figure}

\subsection{Additional results}
\subsubsection{Fractional linear maps and polarity}
For a closed convex set $T$ containing $0$, denote its polar set as
before by $T^\circ$. We claim that in a sense, the ``root'' of a
fractional linear map is the polar map. The following theorem states
that the so called ``distortion'' of fractional linear maps
corresponds to two actions of polarity, each with respect to a
different point of origin.

\begin{thm}\label{Thm_Polara-By-FL}
Let $0\in K\subseteq\{x_1<1\}\subseteq\R^n$ be a closed convex set.
Then for the canonical form of a fractional linear map,
$F_0(x)=\frac{x}{x_1-1}$, the following holds: \[F_0(K) =
(e_1-K^\circ)^\circ.\]
\end{thm}

In \cite{BS}, \cite{Slo} the authors prove uniqueness theorems for
order isomorphisms 
on various families of convex sets. Here we see new such maps, on
the family of closed convex bodies which are contained in a
half space. Uniqueness of these maps in some weak sense (among point
maps) follows immediately from  the uniqueness theorem
\ref{Thm_FL-Unique-nD}.  Applying techniques from those papers one
can get uniqueness of these maps among all order isomorphisms on
this class of convex bodies.

\noindent{\bf Proof of Theorem \ref{Thm_Polara-By-FL}.} Let $T$ be a
closed convex set. Clearly
\[[0,e_1]\subseteq T \quad\Leftrightarrow\quad [-e_1,0]\subseteq T-e_1
\quad\Leftrightarrow\quad (T-e_1)^\circ\subseteq\{x_1>-1\},\] and
therefore under our assumptions for every $x\in(T-e_1)^\circ$ we
have $0<1+x_1$.
 We define $G(-x)=-F_0(x)$, or explicitly $G(x)=\frac{-x}{x_1+1}$. Note that
$F_0$ is an involution on $\{x_1\neq1\}$, and hence $G$ is an
involution on $\{x_1\neq-1\}$. Compute
\begin{align*}
(T-e_1)^\circ&=\left\{x\in\R^n : \iprod{x}{y-e_1}\le 1 \quad \forall
y\in T \right\}\\&=\left\{x\in\R^n : \iprod{x}{y}\le 1+x_1
\quad\forall y\in T \right\}\\&=\left\{x\in\R^n :\left\langle
\frac{-x}{1+x_1},-y\right\rangle\le 1 \quad\forall y\in T\right\}\\&
=\left\{x\in\R^n : \iprod{G(x)}{-y}\le 1 \quad\forall y\in T
\right\}\\&=\left\{G^{-1}(x)\in\R^n : \iprod{x}{y}\le 1\quad \forall
y\in (-T) \right\}\\&=G^{-1}\left( \left\{x\in\R^n : \iprod{x}{y}\le
1\quad \forall y\in (-T) \right\} \right)\\&=
G^{-1}((-T)^\circ)=G(-T^\circ)=-F_0(T^\circ),
\end{align*}
which in turn implies
\[F_0(T^\circ) = (e_1-T)^\circ,\] for sets $T$ which contain
the interval $[0,e_1]$, or conversely, such that $T^\circ\subseteq\{
x_1<1\}$. Therefore we can formulate it in the following way, for a
closed convex $K\subseteq\{x_1<1\}$ such that $0\in K$ we have
\[F_0(K) = (e_1-K^\circ)^\circ.\] \qed

\begin{rem}
Recall that $\{x_1=1\}$ is the defining hyperplane of $F_0$, so we
cannot hope to get that result for $K$ which intersects this
hyperplane. In the other side of this hyperplane, however, we do not
have $0$, and again cannot work with $K^\circ$.
\end{rem}

\begin{rem} By Theorem \ref{Thm_Can-Form}, once we understand the
action of $F_0$ on convex bodies, we understand the action of all
(non-affine) fractional linear maps on convex bodies, and the only
difference is in some linear maps and translations.
\end{rem}
%
%

\subsubsection{Sets that can be preserved}
The fractional linear maps clearly have a non-linear ``distortion''
of the image. As we saw above, when approaching the  defining
hyperplane, the map diverges. However, fractional linear maps
preserve some structure, for example, they preserve combinatorial
structure of polytopes (number of vertices, faces of every
dimension, intersection between faces, etc). We will investigate
which sets can be preserved by fractional linear maps. 

We present some examples of simple convex sets $K$ for which there
exist fractional linear maps $F$ with $F(K)=K$. This will also shed
some light on the question: ``given sets $K_1$, $K_2$, does there
exist a fractional linear map $F$ such that $F(K_1)=K_2$?''. This
question will have consequences in the next section, where we deal
with classes of functions supported on convex sets (``windows''),
and see that the existence of any order isomorphism between two such
classes depends on the existence of a fractional linear map between
the corresponding windows (more precisely; between the corresponding
{\em cylinders}, either $K_i\times\R^+$ or $K_i\times\R$).

Let us start with an explicit two dimensional example: A non-affine
fractional linear map which preserves the Euclidean disk.

\begin{exm}\label{Exm_2-Dim-Ball}{\bf Euclidean ball, 2 dimensions.}
Define $T:D\to\R^2$, where $D=\{(x,y)\in\R^2 : x<2\}$, in the
following way: \[ \begin{pmatrix} x\\ y \end{pmatrix} \mapsto
\begin{pmatrix} T_1(x)\\ T_2(x,y) \end{pmatrix} =
\begin{pmatrix} \frac{2x-1}{2-x}\\ \frac{\sqrt{3}y}{2-x}
\end{pmatrix}.\] Note that $x^2+y^2=1$ implies $T_1(x)^2 + T_2(x,y)^2
= 1$, that is, $S^1$ is mapped to itself by $T$. It is easy to check
that $T$ maps $S^1$ onto itself. By the interval preservation
property of $T$, this implies that the unit ball is mapped to
itself. Note that $T(0)\neq 0$, with correspondence to Theorem
\ref{Thm_Pres-Symm-Plus-0}.
\end{exm}

\begin{exm}\label{Exm_n-Dim-Ball}{\bf Ellipsoids in $n$
dimensions.} The above explicit example can be extended easily to
the Euclidean ball in $\R^n$.  However, let us discuss this case, or
more generally, the case of an ellipsoid in $\R^n$, in a slightly
more abstract way. Note that a conic section is always mapped by a
fractional linear map to a conic section. Indeed, a conic section in
$\R^n$ is given as a section of the cone $C = \{x_{n+1}^2 =
\sum_{i=1}^{n} x_{i}^2\}$ by a hyperplane (identified with $\R^n$).
Equivalently, we may take the section of a linear image of the cone,
$A(C)$ (for $A\in GL_{n+1}(\R)$) by the hyperplane $\{ x_{n+1} = 1\}
\subseteq \R^{n+1}$. Viewing fractional linear maps as traces of
linear maps on $\R^{n+1}$ (say, given by a matrix $B$), we
immediately get that the image of the conic section corresponding to
$A(C)$ is the conic section corresponding to $BA(C)$. Next, letting
$\cal E$ be some closed ellipsoid in the domain of a non-affine
fractional linear map (so, it is bounded away from the defining
hyperplane), it is mapped to a conic section, but since $F$ is
continuous, this must be a compact conic section, and in particular
a {\em bounded} one. Thus, $F(\cal E)$ is an ellipsoid $\cal E'$.
Finally, since any two ellipsoids can be mapped to one another via
an affine linear map, we can find an invertible affine
transformation $A$ such that $AF(\cal E) = {\cal E}$, and $AF$ is a
non-affine fractional linear map.
\end{exm}

Before moving on to the next convex set, we mention that for
Euclidean balls (and hence ellipsoids) we also have a transitivity
result, in the flavor of Lemma \ref{Lem_Transitivity-n+2} for
simplices. It is given in the following proposition.

\begin{prop} Let $B_n$ denote the open unit ball in $\R^n$,
and $\cE$ be some open ellipsoid, with $p\in\cE$. Then there exists
a bijective fractional linear map $F:\cE \to B_n$ with $F(p)=0$.
\end{prop}

\noindent{\bf Proof.} There exists an {\em affine linear} map that
maps $\cE$ to $B_n$ and $p$ to $p'$, and an orthogonal
transformation which maps $p'$ to $\lambda e_1$ for $0\le\lambda<1$.
If $\lambda=0$ we are done, with $F$ being an affine map.\\ Assume
otherwise; then by invertibility of f.l. maps, our task is to find a
bijective fractional linear map $G:B_n\to B_n$ such that
$G(0)=\lambda e_1$, for a given $0<\lambda<1$. Let $a:=1/\lambda$,
$c:=\sqrt{a^2-1}$ (so, $1<a, 0<c$). One possible choice of $G$ is
induced by the $(n+1) \times (n+1)$ matrix
\[ A_{G} = \left(\begin{array}{cccc}
                                      &     &  &    \\
                                      &  D  &  & e_1\\
                                      &     &  &    \\
                                      & e_1^T &  &  a \\
                                     \end{array}\right),\]
where $D$ is diagonal with eigenvalues $\{a,c,\dots,c\}$. The direct
formula corresponding to that choice of $G$ is:
\[G
\begin{pmatrix} x_1\\ \vdots\\x_n \end{pmatrix} = \frac{1}{x_1+a}
\begin{pmatrix} ax_1+1\\ cx_2\\ \vdots\\ cx_n \end{pmatrix}.\]
\qed

We turn to the second example which is again in $\R^2$, a trapezoid.
\begin{exm}{\bf Trapezoid.} Let $\alpha>0$, and $D=\{(x,y)\in\R^2 :
x<1+\alpha^{-1}\}$. Define $T:D\to\R^2$ and $A:\R^2\to\R^2$ in the
following way: \[ T \begin{pmatrix} x\\ y \end{pmatrix} =
\begin{pmatrix} T_1(x)\\ T_2(x,y) \end{pmatrix}=
\begin{pmatrix} \frac{x}{1+\alpha-\alpha x}\\
\frac{(1+\alpha)y}{1+\alpha-\alpha x} \end{pmatrix},\qquad A
\begin{pmatrix} x\\ y \end{pmatrix} =
\begin{pmatrix} 1-y\\ x \end{pmatrix}.\]

The affine linear map $A$ is the $\pi/2$ rotation around $(1/2,
1/2)$, and so it maps the four points $(0,0), (1,0), (0,1), (1,1)$
to themselves in a cyclic manner i.e. to $(1,0), (0,1), (1,1),
(0,0)$ respectively. The fractional linear map $T$ fixes the three
points $(0,0), (1,0), (0,1)$, and maps $(1,1)$ to $(1,1+\alpha)$.

Denote by $K$ the trapezoid with vertices $(0,0), (1,0), (0,1)$,
$(1,1+\alpha)$, and consider $F:K\to K$ defined by $F:=T\circ A\circ
T^{-1}$. It is obvious that $F$ is not affine, and that it maps the
four vertices of $K$ to themselves cyclically, thus by interval
preservation, $F(K)=K$. These two facts can also be verified from
the direct formula of $F$:\[
\begin{pmatrix} x\\ y \end{pmatrix}\mapsto \begin{pmatrix}
\frac{\alpha x - y + 1}{\alpha x + \alpha y + 1} \\
\frac{(\alpha + 1)^2 x}{\alpha x + \alpha y + 1} \end{pmatrix}.\]
\end{exm}

Note, had we chosen $\alpha=0$, our trapezoid $K$ would be a square,
and we would get that $T$, therefore $F$, are both affine maps, and
thus we see that at least with this construction, we did not get a
non-affine fractional linear map that preserves the cube in $\R^2$.
This is, in fact, a general result in $\R^n$.

We denote by $Q^n$ the unit ball of the $l_\infty$ norm in $\R^n$,
and by $B_1^n$ the unit ball of the $l_1$ norm in $\R^n$:
\[Q^n:=\{x\in\R^n : -1\le x_i\le 1, \quad i=1,...,n\},\]
\[B_1^n:=\{x\in\R^n : \sum_{i=1}^n |x_i|\le 1, \quad i=1,...,n\}.\]
\begin{thm}\label{Thm_Cube} Any bijective fractional linear map
$F: Q^n \to Q^n$ is affine.
\end{thm}

\begin{thm}\label{Thm_Cross-Polytope} Any bijective fractional
linear map $F: B_1^n \to B_1^n$ is affine.
\end{thm}

We use the following lemma:
\begin{lem} Let $K\subset\R^n$ be a non-degenerate closed polytope,
and $f:K\to \R^n$ a fractional linear map. If two pairs of opposite
and parallel facets are mapped to such pairs, the map must be affine.
\end{lem}
\noindent{\bf Proof.} By Section \ref{Sect_Geom_Desc}, if $f$ is not
affine, there is only one direction in which $f$ maps parallel
hyperplanes to parallel hyperplanes. Therefore, if two $n-1$
dimensional subsets are parallel (but are not contained in the same
hyperplane), and mapped to parallel sets, they must lie on a
translate of the defining hyperplane of $f$. Assume that $F_1, F_2$
are two parallel facets of $K$, and likewise $F_3, F_4$. There is no
hyperplane whose shifts contain all four facets, since $K$ is a
polytope of full dimension (there are no more than two parallel
facets). Therefore, the fact that the pair $F_1, F_2$ is mapped to a
similar pair, and likewise $F_3, F_4$, implies that $f$ is affine.
\qed

\noindent{\bf Proof of Theorems \ref{Thm_Cube},
\ref{Thm_Cross-Polytope}.} Both the facets of $Q^n$ and of $B_1^n$
have the property that every two non-opposite facets intersect.
Therefore, every pair of opposite facets is mapped to such a pair. In
particular, we have two such pairs, and by the previous lemma this
implies that $f$ is affine. \qed

Next, we prove that if $K$ is a centrally symmetric convex body, the
only fractional linear maps which may preserve both $K$ and $\{0\}$
are affine.

\begin{thm}\label{Thm_Pres-Symm-Plus-0} Let $K\subseteq\R^n$ be a
closed, convex, centrally symmetric body, and let $F:K\to K$ be a
bijective fractional linear map. If $F(0)=0$, then $F$ is linear.
\end{thm}

\noindent{\bf Proof of Theorem \ref{Thm_Pres-Symm-Plus-0}.} As
usual, since $F(0)=0$ we assume that the inducing matrix of $F$ has
the form: \[ F\sim\begin{pmatrix} A &0\\ v^T & -1 \end{pmatrix}, \]
where $A\in GL_n$, and $0, v\in\R^n$. Therefore $F(x) =
\frac{Ax}{\iprod{v}{x}-1}$.\\ We need to show that $v=0$. Otherwise,
let $x\in\partial K$ be such that $\iprod{v}{x}\neq0$ (for example,
take $x$ in the direction of $v$). The interval $[x,-x]$ is mapped
by $F$ to the interval $[F(x), F(-x)]$. Since $F$ is surjective,
$F(x)$ and $F(-x)$ are also on the boundary of $K$, and by the
formula they are in opposite direction, which means that
$F(-x)=-F(x)$, by symmetry of $K$. By $\|{F(x)}\| = \|{F(-x)}\|$ we
get $|{\iprod{v}{x}+1}| = |{\iprod{v}{x}-1}|$, meaning
$\iprod{v}{x}=0$, in contradiction to our choice of $x$. Thus we
conclude $v=0$, which means that $F$ is linear. \qed

\begin{rem} The theorem remains correct also when the condition
``closed" is omitted. If the closure of $K$ is contained in the
maximal domain of $F$ (the half space parallel to the defining
hyperplane), then by continuity of $F$ we get that the same
conditions hold for the closure of $K$, apply the theorem, and
conclude that $F$ is linear. In the other case, i.e. when the
closure of $K$ intersects the defining hyperplane, one must be more
careful, and we omit the details completing the proof.
\end{rem}

\begin{rem} The condition $F(0) = 0$ cannot be omitted. Indeed, we
have seen examples of symmetric bodies preserved by non-affine
fractional linear maps, for instance in Example
\ref{Exm_2-Dim-Ball}.
\end{rem}

\begin{thm}\label{Thm_Pres-Simplex-Plus-Center} Let
$\Delta\subseteq\R^n$ be a closed, non-degenerate simplex, and
$p\in\Delta$ its center of mass. If $F:\Delta\to \Delta$ is a
bijective fractional linear map with $F(p)=p$, then $F$ is affine
linear.
\end{thm}

\noindent{\bf Proof of Theorem \ref{Thm_Pres-Simplex-Plus-Center}.}
Denote by $x_0,...,x_n$ the vertices of $\Delta$. Let
$A:\Delta\to\Delta$ be the affine map defined by the conditions
$A(x_i)=F(x_i)$, $i=0,...,n$. Such a map obviously exists, moreover
it is unique, and it is invertible. Note that $A(\Delta)=\Delta$
implies $A(p)=p$, since the center of mass is a linear invariant. By
Lemma \ref{Lem_Mizdahim-n+2}, this implies $F=A$, meaning that $F$
is affine linear. \qed

\begin{rem} As in the case of symmetric bodies, the condition
$F(p)=p$ cannot be omitted. In fact we have seen in lemma
\ref{Lem_Transitivity-n+2} a transitivity result, stating that
fractional linear maps can map any simplex to itself, with an
arbitrary permutation on the vertices, and in addition map a given
point inside - say, the center of mass - to an arbitrary point
inside. In the last theorem we have seen that among these maps, the
affine maps are the only ones which map the center of mass to
itself.

However, the choice of a different point inside will not give the
same result. Meaning, for any point $p'$ in the interior of $\Delta$
which is not the center of mass, there exists a non-affine
fractional linear map $F$ such that $F(\Delta)=\Delta$ and
$F(p')=p'$. The construction is quite simple - find a linear map $A:
\Delta \to \Delta$ which permutes the vertices and does {\em not}
fix the point $p'$ (such a map is easily seen to exist), and then
compose it with a fractional linear which fixes the vertices but
``restores'' $A(p')$ to $p'$ (that map will be non-affine, since the
only affine map which fixes all the vertices is the identity map).
This composition is the wanted map.
\end{rem}

\section{Background on order isomorphisms}
\label{Sect_Background} Our main interest in what follows is order
preserving and order reversing transforms on convex functions, when
the functions are restricted to being defined on a convex body in
$\R^n$ rather than the whole space. It turns out that this
restriction changes the picture entirely, and a new family of
transformations appears. These transformations are based on
fractional linear maps, which we studied in detail in Section
\ref{Sect_I.P.M}.

\subsection{General order isomorphisms}\label{Sect_Gen-Ord-Iso}
\begin{defn} If $\S_1, \S_2$ are partially ordered sets, and
$\T:\S_1\to\S_2$ is a bijective transform, such that for every $f, g
\in \S_1$: $f \le g \quad \Leftrightarrow\quad \T f \le \T g$, we
say that $\T$ is an {\em order preserving isomorphism}.
\end{defn}

\begin{defn} If $\S_1, \S_2$ are partially ordered sets, and
$\T:\S_1\to\S_2$ is a bijective transform, such that for every $f, g
\in \S_1$: $f \le g \quad \Leftrightarrow\quad \T f \ge \T g$, we
say that $\T$ is an {\em order reversing isomorphism}.
\end{defn}

\begin{defn} A partially ordered set $\S$ is said to be {\em closed
under supremum}, if for every $\{f_\alpha\}\subseteq \S$, there
exists a unique element in $\S$, denoted $\sup \{f_\alpha\}$, with
the following two properties:

1. For every $\alpha$, $f_\alpha \le \sup \{f_\alpha\}$ (bounding
from above).

2. If $g\in \S$ also bounds $\{f_\alpha\}$ from above, then $\sup
\{f_\alpha\}\le g$ (minimality).
\end{defn}

\begin{defn} A partially ordered set $\S$ is said to be {\em closed
under infimum}, if for every $\{f_\alpha\}\subseteq \S$, there
exists a unique element $f\in \S$, with the following two
properties:

1. For every $\alpha$, $f \le f_\alpha$ (bounding from below).

2. If $g\in \S$ also bounds $\{f_\alpha\}$ from below, then $g\le f$
(maximality).
\end{defn}

Consider the case where $\S$ is a partially ordered set which
contains a minimal element, and is closed under supremum. When $\S$
is one of the classes of convex functions we deal with, $\sup
\{f_\alpha\}$ may be given by the pointwise supremum. However, the
corresponding pointwise $\inf \{f_\alpha\}$ operation may not give a
convex function. To obtain an infimum operation (denoted
$\hat{\inf}$), we use the supremum operation in the following way:
\[\hat{\inf_{\alpha\in A}} \{f_\alpha\}:=
\sup\{g\in\S:\forall\alpha\in A\quad g\le f_\alpha\}.\] That is,
$\hat{\inf} \{f_\alpha\}$ is the largest element which is below the
family $\{f_\alpha\}$. Using $\hat{\inf}$, we see that these
classes are also closed under infimum; the first property is due to
the minimality of $\sup$, and the second holds since $\sup$ is a
bound from above. Dealing with convex functions, we have:
\begin{enumerate}
\item $\hat{\inf}\{f_\alpha\}\leq \inf\{f_\alpha\}$.

\item When $\inf\{f_\alpha\}$ is already a convex function,
$\inf\{f_\alpha\}=\hat{\inf}\{f_\alpha\}$.\\ For example, if $f$ is
a convex function, then $f = \inf\{\delta_{x, f(x)}\}$ (recall that
$\delta_{x, c}(y)=+\infty$ for $y\neq x$, and $\delta_{x, c}(x)=c$).
Thus $\inf\{\delta_{x, f(x)}\} = \hat{\inf}\{\delta_{x, f(x)}\} =
f$.
\end{enumerate}

Next we follow Proposition 2.2 from \cite{AM3}, which states that an
order preserving isomorphism $\T$ must satisfy $\T(\sup\{f_\alpha\})
=\sup \{\T f_\alpha\}$ and $\T(\hat{\inf} \{f_\alpha\}) = \hat{\inf}
\{\T f_\alpha\}$, that is, $\sup$ and $\hat{\inf}$ are {\em
preserved} by $\T$. Similarly, an order reversing isomorphism
satisfies $\T(\sup \{f_\alpha\})=\hat{\inf} \{\T f_\alpha\}$ and
$\T(\hat{\inf} \{f_\alpha\}) = \sup \{\T f_\alpha\}$, that is,
$\sup$ and $\hat{\inf}$ are {\em interchanged} by $\T$. We will
prove this lemma for the case of order isomorphisms and order
reversing isomorphisms between two {\em possibly different}
partially ordered sets.

\begin{prop}\label{Prop_Sup-To-Sup} Let $\S_1, \S_2$ be partially
ordered sets closed under supremum and infimum, and let $\T:\S_1 \to
\S_2$ be an order preserving isomorphism. Then for any family
$f_\alpha \in \S_1$ we have
\[\T ( \hat{\inf} \{f_{\alpha}\} ) = \hat{\inf} \{ \T f_\alpha \},\]
\[\T ( \sup       \{f_{\alpha}\} ) = \sup       \{ \T f_\alpha \}.\]
\end{prop}

\begin{prop}\label{Prop_Sup-To-Inf} Let $\S_1, \S_2$ be partially
ordered sets closed under supremum and infimum, and let $\T:\S_1 \to
\S_2$ be an order reversing isomorphism. Then for any family
$f_\alpha \in \S_1$ we have
\[\T ( \hat{\inf} \{f_{\alpha}\} ) = \sup       \{ \T f_\alpha \},\]
\[\T ( \sup       \{f_{\alpha}\} ) = \hat{\inf} \{ \T f_\alpha \}.\]
\end{prop}

Both proofs are almost identical to the proof of Proposition 2.2 in
\cite{AM3}, but we cannot apply it directly, since here the domain
and image of $\T$ may be different sets. Therefore we prove below
only Proposition \ref{Prop_Sup-To-Sup} (the proof of Proposition
\ref{Prop_Sup-To-Inf} follows the exact same lines).

\noindent{\bf Proof of Proposition \ref{Prop_Sup-To-Sup}.} Let
$\{f_\alpha\}_ {\alpha\in A}\subseteq\S_1$. Denote $f=\sup
\{f_\alpha\}$, and $g$ such that $\T g = \sup \{\T f_\alpha\}$ -
such $g$ exists due to surjectivity of $\T$. We wish to show that
$\T f = \T g$, i.e. $f = g$. Since $f \ge f_\alpha$ for all
$\alpha$, we get $\T f \ge \T f_\alpha$ for all $\alpha$, thus $\T f
\ge \sup \{\T f_\alpha\} = \T g$, which implies $f \ge g$. On the
other hand, since $\T g \ge \T f_\alpha$ for all $\alpha$, we have
$g \ge f_\alpha$ for all $\alpha$, thus $g \ge \sup \{f_\alpha\} =
f$. We have seen $f \ge g {\rm ~and~} g \ge f$, therefore $f = g$.\\
For $\hat{\inf}$, denote $f=\hat{\inf} \{f_\alpha\}$, and $g$ such
that $\T g = \hat{\inf} \{\T f_\alpha\}$. We wish to show that $\T
f=\T g$, i.e. $f = g$. Since $f\le f_\alpha$ for all $\alpha$, we
get $\T f \le \T f_\alpha$ for all $\alpha$, thus $\T f \le \hat
{\inf} \{\T f_\alpha\} = \T g$, which implies $f \le g$. On the
other hand, since $\T g \le \T f_\alpha$ for all $\alpha$, we get $g
\le f_\alpha$ for all $\alpha$, thus $g \le \hat{\inf} \{f_\alpha\}=
f$. We have seen $f \ge g {\rm ~and~} g \ge f$, therefore $f = g$.
\qed

\subsection{Order isomorphisms of convex functions}
In a recent series of papers, the first and third named authors have
crystallized the concept of duality and investigated order reversing
isomorphisms (called there ``abstract duality'') for various classes
of objects and functions, see \cite{AM}, \cite{AM2}. The main
theorem in \cite{AM} can be stated in two equivalent forms which we
quote here for future reference.

Recall the Legendre transform $\L$ for a function $\phi: \R^n \to \R
\cup \{\infty\}$; one first fixes a scalar product
$\iprod{\cdot}{\cdot}$ on $\R^n$ (that is, a pairing between the
space and the dual space). The Legendre transform $\L$ is then
defined by \begin{equation} (\L\phi)(x)=\sup_{y} \{ \iprod{x}{y}
-\phi(y)\}.
\end{equation} It is an involution on the class of all
lower-semi-continuous convex functions on $\R^n$, denoted
$Cvx(\R^n)$. More precisely, $Cvx(\R^n)$ consists of all convex
l.s.c. functions $f: \R^n \to \R \cup \{ +\infty\}$, together with
the constant $-\infty$ function.

\begin{thm} Let $\T :Cvx(\R^n)\to Cvx(\R^n)$ be an order reversing
involution. Then there exist $C_0\in\R$, $v_0\in\R^n$ and a
symmetric transformation $B\in GL_n$, such that
\[ (\T \phi) (x) = (\L \phi)(Bx +v_0) + \iprod{x}{v_0}+C_0.\]
\end{thm}
We call these two properties ``abstract duality'', and so we say
that on the class $Cvx(\R^n)$ there is, up to linear terms, only one
duality transform, $\L$. More generally we have:
\begin{thm} Let $\T:Cvx(\R^n)\to Cvx(\R^n)$ be an order reversing
isomorphism. Then, there exist $C_0\in \R, C_1\in \R^+, v_0, v_1\in
\R^n$ and $B\in GL_n$, such that
\[ (\T \phi) (x) =
C_0 + \iprod{v_1}{x} + C_1 (\L \phi)(B(x+v_0)).\]
\end{thm}

As usual, this is equivalent to the following

\begin{thm} Let $\T:Cvx(\R^n)\to Cvx(\R^n)$ be an order preserving
isomorphism. Then there exist $C_0\in \R, C_1\in \R^+, v_0, v_1 \in
\R^n$ and $B \in GL_n$, such that
\[ (\T\phi)(x) = C_1\phi(Bx+v_0) + \iprod{v_1}{x} + C_0.\]
\end{thm}

\subsection{Order isomorphisms of geometric convex
functions}\label{Sect_Cvx0-Old} The subclass of $Cvx(\R^n)$
consisting of non negative functions with $f(0) = 0$ is denoted by
$Cvx_0(\R^n)$. Next we follow \cite{hidden-structures} to define two
transforms $\J$ and $\A$ on this class. Consider the following
transform, defined on $Cvx_0(\R^n)$:
\begin{equation}
({\A}f)(x)=\left\{
  \begin{array}{ll}
    \sup_{\{y\in\R^n: f(y)>0\}}\frac{\iprod{x}{y}-1}{f(y)} &
    \mbox{if}~~ x \in \{f^{-1}(0)\}^{\circ}\\
    {+\infty} & \mbox{if}~~ x \not\in \{f^{-1}(0)\}^{\circ}
  \end{array} \right\}.\end{equation}
(with the convention $\sup \emptyset = 0$). One may check that it is
order reversing. This transform (with its counterpart $\J$ defined
below) first appeared in the classical monograph \cite{Rock}, but
remained practically unnoticed until recently. For details, a
geometric description, and more, see \cite{hidden-structures}. Next
define: \[ \J = \L \A = \A \L.
\] Clearly, as a composition of two order reversing isomorphisms, it
is an order preserving isomorphism. The formula for $\J$ can be
computed (again, see \cite{hidden-structures} for details), and has
the form:
\[ (\J f)(x) = \inf \{ r>0 ~:~ f(x/r)\le 1/r\},\]
with the convention $\inf \emptyset = +\infty$. It turns out that,
apart from the identity transform, up to linear variants, this is
the only order preserving transform on the class $Cvx_0(\R^n)$. It
was shown in \cite{hidden-structures} that the following uniqueness
theorems for $\J$ hold.

\begin{thm}\label{Thm_Cvx0-n=1-Old} If $\T :Cvx_0(\R^+) \to
Cvx_0(\R^+)$ is an order isomorphism, then there exist two constants
$\alpha>0$ and $\beta >0$ such that either (a-la-$\I$) for every
$\phi \in Cvx_0(\R^+)$,
\[ (\T \phi)(x) = \beta \phi (x/\alpha),\]
or (a-la-$\J$), for every $\phi \in Cvx_0(\R^+)$,
\[ (\T \phi)(x) = \beta (\J \phi) (x/\alpha).\]
\end{thm}

In higher dimensions, it was shown that
\begin{thm}\label{Thm_Cvx0-n>1-Old} Let $n \ge 2$. Any order
isomorphism $\T: Cvx_0(\R^n)\to Cvx_0(\R^n)$ is either of the form
$\T f = C_0 f\circ B$ or of the form $\T f = C_0(\J f) \circ B$ for
some $B\in GL_n$ and $C_0>0$.
\end{thm}

It is interesting to notice, and will be quite important in the
sequel, that the map (on functions) $\J$ is actually induced by a
point map on the epi-graphs of those functions. Indeed, one can
check that for every $f\in Cvx_0(\R^n)$, the bijective map
$F:\R^n\times\R^+\to\R^n\times\R^+$ given by
\[ F(x,y) = \left( \frac{x}{y}, \frac{1}{y} \right),  \]
satisfies
\[ epi (\J f) = F(epi (f)),\] where
\[ epi (f) = \{ (x,y)\in \R^n \times \R^+ : f(x) < y\}.\]
See \cite{hidden-structures} for details. Moreover, we see that $F$
is actually a fractional linear map. We will get back to this issue
frequently in the next two sections.

Clearly, if we have a point map which preserves the set ``epi-graphs
of (a certain subset of) convex functions'' then it induces an order
preserving transform on this subset. It is not clear that, in some
cases, any order preserving transform is induced by such a point
map. However, this turns out to be the case both in the theorems
described above, and in all theorems in the next two sections. Let
us emphasize that this is also, usually, the idea behind the proof.
First one shows that the transform must be induced by some point
map, and moreover, one which preserves intervals. Next one uses some
theorem which classifies all interval preserving maps (for example,
the fundamental theorem of affine geometry, or Theorem
\ref{Thm_FL-Unique-nD}), and finally one checks which of these maps
really induces a transform on the right class, by this getting a
full classification of order preserving transforms.

\subsection{Order reversing isomorphisms}
Considering order reversing transforms, the situation is slightly
different, since there are two different cases. The first case is
when one is given a set on which there is a known order reversing
transform, such as $\L$ on $Cvx(\R^n)$ or on $Cvx_0(\R^n)$, for
example. In that case the classification of order reversing
transforms is completely equivalent to the classification of order
preserving ones, by composing each of them with the known transform.
For example, the theorems above give the following:
\begin{thm} Let $n \ge 2$. If $\T: Cvx(\R^n)\to Cvx(\R^n)$ is an
order reversing involution, then $\T$ is of the form $\T f = (\L f)
\circ B + C_0$, for some symmetric $B\in GL_n$ and $C_0\in\R$.
\end{thm}
\begin{thm} Let $n \ge 2$. If $\T: Cvx_0(\R^n)\to Cvx_0(\R^n)$ is an
order reversing involution, then $\T$ is either of the form $\T f =
(\L f) \circ B$, or of the form $\T f = C_0(\A f) \circ B$, for some
symmetric $B\in GL_n$ and $C_0>0$.
\end{thm}

However, there exists a second case in which there is no order
reversing transform and this requires a different treatment, since
one cannot use the above mentioned strategy, and is forced to find
the real obstruction for the existence of such a transform (see
\cite{AM3} for examples). In Section \ref{Sect_Order-Rev-Cvx} we
deal with order reversing isomorphisms on $Cvx(K)$, and show that
when $K\neq\R^n$, there
are no such transforms. 

\section{The cone of convex functions on a window}\label{Sect_Cvx}
\subsection{Introduction}
We investigate the question of characterizing order isomorphisms on
convex functions, when the domain of the functions is not the whole
of $\R^n$ but a convex subset. One such example which has
already been studied (see \cite{hidden-structures}) is the case of
geometric convex functions on $\R^+$. Since this example is central
also for our setting, we describe it in detail below. First, let us
recall the following definition:

\begin{defn} The class of all lower-semi-continuous convex functions
$f: K \to \R\cup \{ \infty \}$ together with the constant $-\infty$
function on $K$ will be denoted $Cvx(K)$. It can be naturally
embedded into $Cvx(\R^n)$ by assigning to $f$ the value $+\infty$
outside $K$.
\end{defn}

We often call $K$ a window, on which we observe the functions of
$Cvx(\R^n)$. Our first results regard a description of order
isomorphisms on the class of convex functions defined on a window.
We state two versions, one of which does not assume surjectivity,
but in which the order preservation condition is replaced by a
slightly stronger condition of preservation of supremum and
generalized infimum.

\begin{thm}\label{Thm_Cvx-Bij} Let $n\ge1$, and let $K_1, K_2
\subseteq\R^n$ be convex sets with non empty interior. If $\T:
Cvx(K_1)\to Cvx(K_2)$ is an order preserving isomorphism, then there
exists a bijective fractional linear map $F: K_1\times \R \to
K_2\times \R$, such that $\T$ is given by \[ epi(\T f) =
F(epi(f)).\] In particular, $K_2$ is a fractional linear image of
$K_1$.
\end{thm}

\begin{thm}\label{Thm_Cvx-Inj} Let $n\ge1$, and let $K_1, K_2
\subseteq\R^n$ be convex sets with non empty interior. If $\T:
Cvx(K_1)\to Cvx(K_2)$ is an injective transform satisfying:
\begin{enumerate}
\item $\T(      \sup_\alpha f_\alpha)=      \sup_\alpha\T f_\alpha$.
\item $\T(\hat{\inf}_\alpha f_\alpha)=\hat{\inf}_\alpha\T f_\alpha$.
\end{enumerate} for any family $\{f_{\alpha}\} \subseteq Cvx(K_1)$,
then there exist $K_2'\subseteq K_2$, and a bijective fractional
linear map $F: K_1\times\R \to K_2'\times\R$, such that $\T$ is
given by \[epi(\T f)= F(epi(f)).\] Note that for $x\not\in K_2'$ we
get $(\T f)(x)= +\infty$.
\end{thm}

Note that by Proposition \ref{Prop_Sup-To-Sup}, an order isomorphism
respects the actions of $\sup$ and $\hat{\inf}$. Therefore Theorem
\ref{Thm_Cvx-Inj} is stronger, and implies Theorem
\ref{Thm_Cvx-Bij}. However, in the bijective case some of the
reasoning is much simpler, and therefore below we prove both
theorems independently, for clarity.

\begin{rem} Let us elaborate on the meaning of the equation $epi(\T f) =
F(epi(f))$. When $F$ induces a transform on $Cvx(K)$, it is shown in
Section \ref{Sect_Which-FL-Work-Cvx} that up to some affine linear
functional $L_1$, $F$ is of the form \[F(x,y)=\left(
\frac{Ax+u}{\iprod{v}{x}+d}, \frac{y}{\iprod{v}{x}+d} \right),\]
where $A\in L_n(\R)$, $u,v\in\R^n$, and $d\in\R$. Denoting $L_0=
\iprod{v}{\cdot}+d$ for the affine linear functional in
the denominator, and $F_b(x)= \frac{Ax+u}{\iprod{v}{x}+d}$ for the
base-map (projection of $F$ to the first $n$ coordinates), we
conclude that
\begin{equation}\label{eq-changeofvars}(\T f) = \left(\frac{f}{L_0}
\right) \circ F_b^{-1} + L_1,\end{equation} where $L_1$ is some
affine linear functional and $F_b^{-1}:K_2\to K_1$ is bijective.
Note that $L_0$ and $F_b$ are not independent, since $L_0$ must
vanish on the defining hyperplane of $F_b$ (where $F_b$ is not
defined). Moreover, note that for a general $f$, the function
$\frac{f}{L_0}$ may not be convex, but the composition with
$F_b^{-1}$ exactly compensates this problem, and the result is again
a convex function. In the special case of $A=I, u=0, L_0(x)=x_1+1,
L_1(x)\equiv0$ we get $F(x,y)=(\frac{x}{x_1+1},\frac{y}{x_1+1})$,
and $(\T f)(x)=(1-x_1)f(\frac{x}{1-x_1})$. This simpler form of the
transform is not general, but if one allows linear actions on the
epi-graphs, before and after $F$ acts on them, it suffices to
consider this form. There is another important, {\em different,}
instance of the equation $epi(\T f) = F(epi(f))$, which may occur
when the transform is defined on the subset of $Cvx(K)$ consisting
of non-negative functions vanishing at the origin. We state it now
for comparison and elaborate below (Theorem \ref{Thm_Cvx0-n>1-New}).
A transform of this second, essentially different, type (a-la-$\J$,
see \cite{hidden-structures}), corresponds to the inducing
fractional linear map: \[F_\J(x,y)=(\frac{x}{y}, \frac{1}{y}),\] and
to the explicit formula:
\[(\J f)(x)=\inf\{r>0 : rf(\frac{x}{r})\le1 \}.\]
\end{rem}

\subsection{The bijective case}
\noindent{\bf Proof of Theorem \ref{Thm_Cvx-Bij}.} The proof is
composed of several steps.\\
\noindent{\sl Extremality of delta functions.} As in
\cite{hidden-structures}, we define the following family $P$ of
extremal functions: $f\in P$ if every two functions above $f$ are
comparable, that is: \[f\le g,h \quad\Rightarrow \quad g\le h \quad
\mbox{or}\quad h\le g.\] This implies that the support of $f$ (the
set on which $f$ is finite) consists of only one point. We call
these functions {\bf delta functions}, and denote by $\delta_{x,c}$
the function which equals $c$ at the point $x$, and $+\infty$
elsewhere.

$\T$ is a bijection between the family $P$ in $Cvx(K_1)$ and the
family $P$ in $Cvx(K_2)$, since this property is defined only using
the ``$\le$'' relation, which $\T$ preserves in both directions.
Thus $\T (\delta_{x, c}) = \delta_{y, d}$, and this map between
delta functions is bijective. This allows us to define a bijection
$F: K_1\times \R \to K_2\times \R$; $F(x,c)=(y(x,c),d(x,c))$, such
that $\T(\delta_{x,c}) = \delta_{F(x,c)}$. In fact, we get that $y = y(x)$ and $d = d(x,c)$
because two functions $\delta_{x, c}$ and $\delta_{x, c'}$ are
comparable, and so must be mapped to comparable functions. Note that
also $y(x)$ is bijective. Indeed, it is injective since the images
of two functions are comparable if, and {\em only} if, the original
functions are comparable, and it is surjective since all delta
functions are in the image of $\T$.

\noindent{\sl Preservation of intervals.} The ``projection'' of $F$
to the first $n$ coordinates, i.e. the mapping $x\mapsto y(x)$, is a
bijective interval preserving map. Indeed, assume $y(x_1)=y_1$,
$y(x_2)=y_2$, and $x_3\in [x_1, x_2]$. Since $\delta_{x_3,
0}\ge\hat{\inf}\{\delta_{x_1, 0}, \delta_{x_2, 0}\}$, the function
$\delta_{x_3, 0}$ must be mapped to a function $\delta_{y_3,d_3}$
which is above $\hat{\inf}\{\delta_{y_1, d_1}, \delta_{y_2, d_2}\}$.
Since $\hat{\inf}\{\delta_{y_1, d_1}, \delta_{y_2, d_2}\}$ is
$+\infty$ outside $[y_1, y_2]$, this implies $y_3\in [y_1, y_2]$.
For $n\ge 2$, it implies that $y(x)$ is fractional linear, by
Theorem \ref{Thm_FL-Unique-nD}. In fact this is true also when
$n=1$, but for $n=1$ it follows from interval preservation of $F$
itself. To see that $F$ is interval preserving, consider $(x_3,
c_3)$ on the interval between $(x_1, c_1)$ and $(x_2, c_2)$. We know
it is mapped to $(y_3, d_3)$ with $y_3\in [y_1, y_2]$ and moreover,
letting $y_3 = \lambda y_1 + (1-\lambda)y_2$, we know $d_3 \ge
\lambda d_1 + (1-\lambda)d_2$. Using surjectivity, we deduce that
$F(x_3, c_3) = \delta_{y_3, \lambda d_1 + (1-\lambda)d_2}$, since
$\delta_{y_3, \lambda d_1 + (1-\lambda)d_2}$ is above the function
$\hat{\inf}\{\delta_{y_1, d_1}, \delta_{y_2, d_2}\}$ and for all
$c<c_3$, $\delta_{x_3, c}$ is {\em not} above the function
$\hat{\inf}\{\delta_{x_1, c_1}, \delta_{x_2, c_2}\}$.

Since $F$ is an injective interval preserving map, we may apply
Theorem \ref{Thm_FL-Unique-nD}, to conclude that $F$ is a fractional
linear map.

To complete the proof of Theorem \ref{Thm_Cvx-Bij}, let $f \in
Cvx(K_1)$, and write it as
\[ f = \hat{\inf} \{ \delta_{x,y} : (x,y)\in epi(f) \}. \]
\begin{align*}
\Rightarrow\T f&=\hat{\inf} \{\T(\delta_{x,y}):(x,y)\in   epi(f) \}
\\             &=\hat{\inf} \{ \delta_{F(x,y)}:(x,y)\in   epi(f) \}
\\             &=\hat{\inf} \{  \delta_{x,y}  :(x,y)\in F(epi(f))\}.
\end{align*}
On the other hand:
\[\T f = \hat{\inf} \{  \delta_{x,y}: (x,y)\in epi(\T f) \}.\]
Therefore we get \[ epi (\T f) = F(epi(f)),\] as desired. This
completes the proof. \qed

Of course, there are restrictions on the structure of $F$ for it to
induce such a transform. This is elaborated in Section
\ref{Sect_Which-FL-Work-Cvx}.

\subsection{The injective case}
We next move to the case of injective transforms. Let us first
remark why in Theorem \ref{Thm_Cvx-Inj} we had to change the
conditions from mere order preservation to preservation of $\sup$
and $\hat{\inf}$.

\begin{rem}\label{Rem_Lattice-Is-Necessary-For-Cvx-Inj} In the
bijective case, order preservation (in both directions) is
equivalent to preservation of $\sup$ and $\hat{\inf}$. One direction
is given in Proposition \ref{Prop_Sup-To-Sup}, and the other is
given here:
\[f\le g\quad\Rightarrow\quad \T(g)=\T(\sup\{f,g\})=\sup\{\T(f),\T(g)\}
\quad \Rightarrow\quad \T(f)\le \T(g),\]
\[\T(f)\le \T(g)\Rightarrow \T(g)=\sup\{\T(f),\T(g)\}=\T(\sup\{f,g\})
\Rightarrow g = \sup\{f,g\} \Rightarrow f\le g.\] This direction is
true also in the injective case (preservation of $\sup$ and
$\hat{\inf}$ implies order preservation), but the opposite (order
preservation in both directions implies preservation of $\sup$ and
$\hat{\inf}$) is not, as shown in the following example. The
following $\T : Cvx(\R^n) \to Cvx(\R^n)$ is injective and $f \le g$
if and only if $\T f \le \T g$: \[ (\T f)(x) = f (x) + x_1^2.\] But
$\T$ does {\em not} map $\hat{\inf}$ to $\hat{\inf}$. The reason
behind this fact is that $\T$ is not surjective. Moreover, there
exist $f, g$, such that $\hat{\inf} \{\T(f),\T(g)\}$ is not in the
image of $\T$, and in particular it is not equal to
$\T(\hat{\inf}\{f,g\})$; for example take $f(x)=x_1, g(x)=-x_1$.
\end{rem}

For the proof of the more general Theorem \ref{Thm_Cvx-Inj}, we need
the following known geometric lemma. The dimension of a set $K$
denotes the minimal dimension of an affine subspace which contains
the set.

\begin{lem}\label{Lem_GeomLemma-Classic} In an $m$-dimensional
affine space, let $M$ be a closed convex set. Let $\cal F$ be a
family of $m$-dimensional closed convex sets such that $K\neq M$ for
all $K\in\cal F$, and $K_1\cap K_2 = M$ whenever $K_1\neq K_2$ and
$K_1, K_2\in\cal F$. Then $\cal F$ is at most countable.
\end{lem}

We reformulate it, to better suit our need:
\begin{lem}\label{Lem_GeomLemma-Rephrase} Let $M\subseteq\R^n$ be a
fixed closed convex set of dimension $m$. Let $\cal F$ be an
uncountable family of closed convex sets such that $K\neq M$ for all
$K\in\cal F$, and $K_1\cap K_2 = M$ whenever $K_1\neq K_2$ and $K_1,
K_2\in\cal F$. Then for at least one set $K\in\cal F$, $dim(K)\ge
m+1$. In particular, $m\le n-1$.
\end{lem}

Lemma \ref{Lem_GeomLemma-Rephrase} follows from Lemma
\ref{Lem_GeomLemma-Classic}, where the minimal subspace which
contains $M$ is taken to be the $m$-dimensional affine space of
Lemma \ref{Lem_GeomLemma-Classic}. Our application of this lemma
requires a little more, so we prove:
\begin{lem}\label{Lem_GeomLemma-New} Let $M\subseteq\R^n$ be a
fixed closed convex set of dimension $m$. Let $\cal F$ be an
uncountable family of closed convex sets such that $K\neq M$ for all
$K\in\cal F$, and $K_1\cap K_2 = M$ whenever $K_1\neq K_2$ and $K_1,
K_2\in\cal F$. Then for at least one set $K\in\cal F$, $dim(K)\ge
m+1$. Moreover, $m\le n-2$.
\end{lem}

\noindent{\bf Proof.} We wish to prove that $m\neq n-1$; the rest
follows from Lemma \ref{Lem_GeomLemma-Rephrase}. Assume otherwise,
then let $H=\{\iprod{x}{u}=c\}$ be the affine subspace of dimension
$n-1$ which contains $M$. Our assumption is that the relative
interior of $M$ in $H$ is not empty. The set $\{K\in\cal F :
K\subseteq H\}$ is at most countable, by Lemma
\ref{Lem_GeomLemma-Classic}. Since $\cal F$ is not countable, there
are at least three sets which are not contained in $H$, and
therefore (without loss of generality) we have $A,B\in\cal F$ such
that $A\cap H^+\neq\emptyset$, $B\cap H^+\neq\emptyset$, where
$H^+:=\{\iprod{x}{u}>c\}$. Let $a\in A, b\in B$ such that $a,b\in
H^+$, and let $x\in M$ be a point in the relative interior of $M$.
Since $conv\{M,a\}\subseteq A$, we conclude that there is some open
half ball of the form $B_{(x,r)} \cap H^+$ contained in $A$, and
likewise for $B$. The two half balls have non empty intersection, in
contradiction to $A\cap B=M$. \qed

We will use this lemma for epi-graphs of functions. Noting that
\[ epi (\max\{f,g\}) = epi(f) \cap epi(g),\] we get the following
lemma for convex functions:
\begin{lem}\label{Lem_GeomLemma-Functions} Let $M:\R^n\to\R$ be a
fixed convex function, such that $epi(M)\subseteq\R^{n+1}$ is of
dimension $m$. Let $\cal F$ be an uncountable family of convex
functions such that $f<M$ for all $f\in\cal F$, and $max\{f_1, f_2\}
= M$ whenever $f_1, f_2\in\cal F$ and $f_1\neq f_2$. Then for at
least one function $f\in\cal F$, $dim(epi(f))\ge m+1$. Moreover,
$m\le n-1$.
\end{lem}

\noindent{\bf Proof of Theorem \ref{Thm_Cvx-Inj}.} We start by
checking where the constant function $+\infty$ is mapped to. Let us
call its image $f_\infty$. Consider the family $\{\delta_x\}_{x\in
K_1}$, and its image $\{ \T \delta_x\}_{x \in K_1}$. It is
uncountable, and every two functions in the second family satisfy
$\max \{g_1, g_2\} = f_\infty$.

This means, by Lemma \ref{Lem_GeomLemma-Functions}, that there
exists $x_1\in K_1$ such that the dimension of the epi-graph of $\T
\delta_{x_1}$ must be higher by at least $1$ than the dimension of
the epi-graph of $f_\infty$. Similarly, for $x_1$ we construct an
uncountable family of functions $\{\delta_{[x_1,y]}\}_{y \in K_1}$
such that the maximum of every two is $\delta_{x_1}$, and by
applying Lemma \ref{Lem_GeomLemma-Functions} again we get that there
exists at least one such function, the image of which has an
epi-graph with dimension higher by at least 1 than the dimension of
the epi-graph of $\T \delta_{x_1}$. After repeating this
construction an overall of $n-1$ times, we conclude that there exist
$x_1,\dots,x_{n-1}\in K_1$ such that the epi-graph of the function
$\T\delta_{conv\{x_1,\dots,x_{n-1}\}}$ is of dimension higher by at
least $n-1$ than the dimension of the epi-graph of $f_\infty$.
Applying Lemma \ref{Lem_GeomLemma-Functions} one last time, we get
that the dimension of the epi-graph of
$\T\delta_{conv\{x_1,\dots,x_{n-1}\}}$ is at most $(n+1)-2=n-1$.
This means that the epi-graph of $f_\infty$ is of dimension $0$,
that is, $f_\infty=+\infty$.

This also shows that $\T (\delta_{x,c}) = \delta_{y,d}$. Indeed,
since the only epi-graph with dimension $0$ has already been
designated to $f_\infty$, the dimension of the epi-graph of $\T
(\delta_{x,c})$ is at least $1$; but we may construct a chain as
above which implies that it is also at most $1$. We define the
injective map $F: K_1\times\R\to K_2\times\R$ by the relation
$\T(\delta_{x,c})= \delta_{F(x,c)}$, and denote
$F(x,c)=(y(x,c),d(x,c))$.

In fact, we get that $y = y(x)$ and $d = d(x,c)$ because the two
functions $\delta_{x,c}$ and $\delta_{x,c'}$ are comparable, and so
must be mapped to comparable functions (by Remark
\ref{Rem_Lattice-Is-Necessary-For-Cvx-Inj}). Note that $y(x)$ is
injective because the images of two functions are comparable if and
{\em only if} the original functions are comparable. In addition,
$y(x)$ is interval preserving. Indeed, assume $y(x_1)=y_1$,
$y(x_2)=y_2$, and $x_3\in [x_1, x_2]$. Since
$\delta_{x_3}\ge\hat{\inf}\{\delta_{x_1}, \delta_{x_2}\}$, the
function $\delta_{x_3}$ must be mapped to a function
$\delta_{y_3,c}$ which is above $\hat{\inf}\{\delta_{y_1,c_1},
\delta_{y_2,c_2}\}$, which implies $y_3\in [y_1, y_2]$. For $n\ge
2$, the fact that $y(x)$ is an injective interval preserving map
implies that it is fractional linear, by Theorem
\ref{Thm_FL-Unique-nD}. Actually this is true also for $n=1$, but it
only follows from the fact that $(x,c) \mapsto (y,d)$ is also
interval preserving, which we will next show.

\noindent{\bf Remark.} We note that until this point in the proof
(for $n\ge 2$) we only use the max/min condition, and not the
stronger assumed condition for sup/inf; we already get that the map
$F$ is very restricted: it is a fractional linear map on the base,
and some one dimensional map $d_x(c)$ on each fiber, and all these
maps $d_x$ must join together to preserve convexity of epi-graphs.
This seems to restrict $d(x,c)$ enough to determine its form, but we
chose to continue using a different argument, which works also for
$n=1$, but requires the preservation of sup/inf.

To see that $F$ is interval preserving consider the function
$\hat{\min}\{\delta_{x_1,c_1}, \delta_{x_2,c_2}\}$, which is
$+\infty$ outside the interval $[x_1, x_2]$ and linear in it, with
$f(x_1) = c_1$ and $f(x_2) = c_2$. By assumption, it is mapped to
$\hat{\min}\{\delta_{y_1,d_1}, \delta_{y_2,d_2}\}$. Taking $(x_3,
c_3) \in [(x_1, c_1), (x_2, c_2)]$ we have that $\delta_{x_3,c_3}
\ge \hat{\min}\{\delta_{y_1,d_1}, \delta_{y_2,d_2} \}$ and so the
point $(y_3, d_3)$ lies above or on the segment $[(y_1, d_1), (y_2,
d_2)]$.

On the other hand, look at $x_3 = \lambda x_1 + (1-\lambda)x_2$ and
$c_3' < \lambda c_1 + (1-\lambda) c_2$. That is, we take a point
$(x_3, c_3')$ which is under the segment $[(x_1, c_1), (x_2, c_2)]$.
From the ``only if'' condition, we have that $\T (\delta_{x_3,c_3'})
\not \ge \hat{\min}\{\delta_{x_1,c_1}, \delta_{x_2,c_2}\}$. So
$(y_3, d_3')$ is under the segment $[(y_1, d_1), (y_2, d_2)]$, since
$y_3 \in [y_1, y_2]$ and it cannot be above or on it. Since
$\delta_{x_3,c_3} = \sup_{c_3 ' < c_3} \{\delta_{x_3,c_3'}\}$, we
may use the condition of supremum to get $d_3 = \sup \{d_3'\}$, and
thus $(y_3, d_3)$ is below or on the segment $[(y_1, d_1), (y_2,
d_2)]$. Together with what we saw before, this implies $(y_3,
d_3)\in[(y_1, d_1), (y_2, d_2)]$.

So, we have shown that $F: K_1\times \R \to K_2 \times \R$ is an
injective interval preserving map, and we may apply Theorem
\ref{Thm_FL-Unique-nD} to conclude that it is fractional linear.

To complete the proof of Theorem \ref{Thm_Cvx-Inj}, we proceed in
exactly the same way as in the proof of Theorem \ref{Thm_Cvx-Bij},
to conclude that
\begin{align*}
\T f&=\hat{\inf} \{  \delta_{x,y}  :(x,y)\in F(epi(f))\}
\\  &=\hat{\inf} \{  \delta_{x,y}: (x,y)\in epi(\T f) \},
\end{align*}
and thus \[ epi (\T f) = F(epi(f)),\] which completes the proof.\qed

Both proofs generalize without any complication to various other
settings in which one considers different classes, such as the class
of all non negative functions in $Cvx(\R^n)$, or in $Cvx(K)$, or
more generally:
\[S_{f_0} = Cvx(\R^n) \cap \{ f: f_0\le f\},\] for
some fixed $f_0 \in Cvx(\R^n)$. We get:

\begin{thm}\label{Thm_Cvx-Bij-Above-F0} Let $n\ge1$, and let $f_1,
f_2\in Cvx(\R^n)$ be convex functions with support of full
dimension. If $\T:S_{f_1}\to S_{f_2}$ is an order isomorphism, then
there exists a bijective fractional linear map $F: epi(f_1) \to
epi(f_2)$, such that $\T$ is given by \[ epi(\T f) = F(epi(f)).\]
\end{thm}

\begin{thm}\label{Thm_Cvx-Inj-Above-F0} Let $n\ge1$, and let $f_1,
f_2\in Cvx(\R^n)$ be convex functions with support of full
dimension. If $\T:S_{f_1}\to S_{f_2}$ is an injective transform
satisfying:
\begin{enumerate}
\item $\T(      \sup_\alpha f_\alpha)=      \sup_\alpha\T f_\alpha$.
\item $\T(\hat{\inf}_\alpha f_\alpha)=\hat{\inf}_\alpha\T f_\alpha$.
\end{enumerate} for any family $\{f_{\alpha}\} \subseteq S_{f_1}$,
then there exist $f_2'\in S_{f_2}$, and a bijective fractional
linear map $F: epi(f_1) \to epi(f_2')$, such that $\T$ is given by
\[ epi(\T f) = F(epi(f)).\]
\end{thm}

It is tempting to consider Theorems \ref{Thm_Cvx-Bij} and
\ref{Thm_Cvx-Inj} as manifestations of Theorems
\ref{Thm_Cvx-Bij-Above-F0} and \ref{Thm_Cvx-Inj-Above-F0}, where
$f_i$ is the function which attains only the values $-\infty$ on
$K_i$ and $+\infty$ outside $K_i$. The only problem is that these
functions are not elements of $Cvx(\R^n)$, but in fact Theorems
\ref{Thm_Cvx-Bij-Above-F0} and \ref{Thm_Cvx-Inj-Above-F0} can be
further generalized without any effort. Instead of considering only
classes of the form $S_{f_0}=\{f\in Cvx(\R^n): epi(f)\subseteq
epi(f_0)\}$, consider also $\{f\in Cvx(\R^n): epi(f)\subseteq K\}$,
where $K$ is some convex set (in the case of Theorems
\ref{Thm_Cvx-Bij} and \ref{Thm_Cvx-Inj}, $K$ is the infinite
cylinder $K_i\times\R$).

\subsection{Classification of admissible fractional linear
maps}\label{Sect_Which-FL-Work-Cvx} Since fractional linear maps
send intervals to intervals, it is clear (a-posteriori, once we know
the transform is induced by a fractional linear map) that a delta
function $\delta_{x,c}$ is mapped to a delta function
$\delta_{y,d}$; since these are the only functions with epi-graphs
that are half-lines. Moreover, by order preservation, we see that
$y$ is a function only of $x$. Observations of this kind allow us to
classify the type of fractional linear maps that induce
transforms as in Theorem \ref{Thm_Cvx-Bij}.

Let the inducing matrix $A_F\in GL_{n+2}$ be given by
\[A_F = \left(\begin{array}{ccccc}
               &        &      &  u'_1  &   u_1  \\
               &    A   &      & \vdots & \vdots \\
               &        &      &  u'_n  &   u_n  \\
          v'_1 & \cdots & v'_n &    a   &    b   \\
          v_1  & \cdots & v_n  &    c   &    d   \\
      \end{array} \right),\] where $A$ is an $n\times n$ matrix, $v,v',
u,u'\in\R^n$, and $a,b,c,d\in \R$.

The infinite cylinder $K_1\times\R$ is contained in the domain of
$F$, so it must not intersect the defining hyperplane
$H=\{\iprod{v}{x}+cy = -d\}$, which implies $c = 0$. In particular,
$K_1 \subseteq \{\iprod{v}{x} > -d\}$ (the sign of the denominator
is constant on $dom(F)$, and we choose it to be positive; we may do
so due to the multiplicative degree of freedom in the choice of
$A_F$).

Since $F$ must map fibers $\{(x, y): y\in \R\}$ to fibers, we see
that for $i=1, \ldots, n$, $F((x,y))_i=
\left(\frac{Ax+yu'+u}{\iprod{v}{x}+d}\right)_i$ does not depend on
$y$, which implies $u' = 0$.

Let $f\in Cvx(K_1)$. The image of $epi(f)$ must be the epi-graph of
some $g\in Cvx(K_2)$. Since we have chosen a positive sign for the
denominator, this simply means that $a>0$, and we choose $a=1$, thus
exhausting the multiplicative degree of freedom in the choice of
$A_F$.

Finally, let $F'$ be the map corresponding to the following
$(n+1)\times (n+1)$ matrix, having removed the next to last row and
column from $A_F$: \[ A_{F'} = \left(\begin{array}{cccc}
                                      &     &  &   \\
                                      &  A  &  & u \\
                                      &     &  &   \\
                                      & v^T &  & d \\
                                     \end{array}\right).\] The map
$F':K_1\to K_2$ is fractional linear, and corresponds to the action
of $F$ on fibers (the ``projection'' of $F$ to $\R^n$). Thus
$A_{F'}$ must be invertible. We note that this condition always
holds; we have $A_F\in GL_{n+2}$, and since the $(n+1)^{th}$ column
of $A_F$ is $e_{n+1}$, $det(A_{F'})=\pm det(A_F)\neq0$.

We claim that these restrictions are not only necessary but also
sufficient:
\begin{prop} Let $K_1\subseteq\R^n$ be a convex set with interior,
for $n\ge1$. Let $A$ be an $n\times n$ matrix, $u,v,v'\in\R^n$,
$b,d\in\R$, and let $F,F'$ be the fractional linear maps defined by
the following matrices:
\[A_F = \left(\begin{array}{ccccc}
         &        &      &    0   &        \\
         &    A   &      & \vdots &    u   \\
         &        &      &    0   &        \\
         &  v'^T  &      &    1   &    b   \\
         &   v^T  &      &    0   &    d   \\
  \end{array}\right),\qquad A_{F'}=\left(\begin{array}{cccc}
     &     &  &   \\
     &  A  &  & u \\
     &     &  &   \\
     & v^T &  & d \\
  \end{array}\right).\] If the following two conditions are
satisfied:
\begin{enumerate}
\item $K_1 \subseteq \left\{ \iprod{v}{x} > -d\right\}$.
\item $A_{F'}\in GL_{n+1}$, or equivalently $A_F\in GL_{n+2}$.
\end{enumerate}
then $F$ induces an order isomorphism from $Cvx(K_1)$ to $Cvx(K_2)$
by its action on epi-graphs, where $K_2 = F'(K_1)$.
\end{prop}
\noindent {\bf Proof.} The following four conditions must be
checked: that epi-graphs are mapped to epi-graphs, that convexity of
the functions is preserved under the transform, that it is
bijective, and that it is order preserving. Bijectivity and
convexity preservation follow easily by the bijectivity and interval
preservation properties of fractional linear maps, and order
preservation is immediate for transforms induced by a point map. The
fact that epi-graphs are mapped to epi-graphs follows from the zeros
in the ${(n+1)}^{th}$ (next to last) column of $A_F$. \qed

Denote the map from the fiber above $x_1$ to the fiber above
$F'(x_1)=x_2$ by $F_{x_1}:\R\to\R$. It is an affine linear map,
given by\[F_{x_1}(y)=\frac{\iprod{v'}{x_1}+y+b}{\iprod{v}{x_1}+d}.\]

\begin{rem} Letting $x_2 = F'(x_1)$ we get
\[(\T f)(x_2) = F_{x_1}(f(x_1)).\]\\
Note that there is a sort of coupling between the ``projected'' map
$F'$, which determines the $x\in\R^n$ dependency, and $F_{x_1}$,
which determines the $y$ dependency. More precisely: given $F'$, the
transform induced by $F$ is determined, up to multiplication by a
positive scalar, and addition of an affine linear function. We next
show that the linear part is determined by $v'$ and $b$. Consider a
transform $\T$ induced by a map $F$, where \[A_F = \left(
\begin{array}{ccccc}
   &        &        &    0   &       \\
   &    A   &        & \vdots &   u   \\
   &        &        &    0   &       \\
 0 & \cdots &    0   &    1   &   0   \\
   &   v^T  &        &    0   &   d   \\
  \end{array} \right).\] Next, consider the transform:
$(\tilde{\T}f)(x) = (\T f)(x) + \iprod{x}{w} + e$, induced by a map
$\tilde{F}$, where $w\in\R^n$ and $e\in\R$. As before, denote
\[A_{\tilde{F}} = \left(\begin{array}{ccccc}
      &               &      &    0   &           \\
      &   \tilde{A}   &      & \vdots & \tilde{u} \\
      &               &      &    0   &           \\
      &  \tilde{v}'^T &      &    1   & \tilde{b} \\
      &  \tilde{v}^T  &      &    0   & \tilde{d} \\
  \end{array}\right).\] Then $A=\tilde{A}$, $u=\tilde{u}$,
$v=\tilde{v}$, and $d=\tilde{d}$. The only difference is in the next
to last row, namely $v'$ and $b$, and a simple calculation shows
that \[\left(\begin{array}{c}
                        \\
              \tilde{v} \\
                        \\
              \tilde{b} \\
             \end{array}\right)=
\left(\begin{array}{cccc}
       &     &   &   \\
       & A^T &   & v \\
       &     &   &   \\
       & u^T &   & d \\
      \end{array}\right)
\left(\begin{array}{c}
         \\
       w \\
         \\
       e \\
      \end{array}\right).\]
The matrix appearing above is exactly $A_{F'}^T$, so it is
invertible, and therefore, the set of all $v, b$ corresponds exactly
to the set of all affine linear additions to $\T$ (clearly these
affine additions do not harm the properties of order preservation,
bijectivity, etc.).
\end{rem}

\subsection{Order reversing isomorphisms}\label{Sect_Order-Rev-Cvx}
The Legendre transform $\L: Cvx(\R^n)\to Cvx(\R^n)$, is the unique
order reversing isomorphism on $ Cvx(\R^n)$. The corresponding
question for windows is, given $K_1,K_2\subseteq\R^n$, what are all
the possible order reversing isomorphisms between $Cvx(K_1)$ and
$Cvx(K_2)$? It turns out that there are no such order reversing
isomorphisms, except in the aforementioned case where
$K_1=K_2=\R^n$. This is due to the fact that the delta functions
``have nowhere to be mapped to''. We formulate this simple
observation in the following Proposition \ref{Prop_Cvx-Rev}. To this
end we use the following two definitions.

\begin{defn}
Let $P_K\subset Cvx(K)$ denote the following subset of extremal
functions:
\[P_K:=\{f\in Cvx(K):g,h\ge f \Rightarrow g,h \mbox{ are comparable}\}.\]
\end{defn}

\begin{defn}
Let $Q_K\subset Cvx(K)$ denote the following subset of extremal
functions (dual to $P$):
\[Q_K:=\{f\in Cvx(K):g,h\le f \Rightarrow g,h \mbox{ are comparable}\}.\]
\end{defn}

Recall that in this new notation, for any closed convex $K$
(actually, for any $K\subseteq\R^n$), $P_K$ consists exactly of the
delta functions. In $Cvx(\R^n)$, it is clear that $Q_{\R^n}$
consists of linear functions; it follows from the fact that the only
functions below $f=\iprod{c}{x}+ d$ are of the form $g(x)=
\iprod{c}{x}+d'$, for $d'<d$. In the next lemma we see that when
$K\neq\R^n$ is a convex set with non empty interior,
$Q_K=\emptyset$.

\begin{lem}\label{lem:noQ} If $K\subsetneq\R^n$ is a convex set with
non empty interior, then $Q_K=\emptyset$.
\end{lem}
\noindent{\bf Proof.} Clearly, if $f$ is a non linear convex
function, $f\not\in Q$ (take two hyperplanes supporting $epi(f)$ in
different directions). For a linear function $f$, one may easily
construct two non-parallel linear functions below it, which are not
comparable (they will satisfy $g(x), h(x)\le f(x)$ for every $x\in
K$, not for every $x\in\R^n$). Note that the fact that $K$ has
non empty interior is essential, otherwise there is no guarantee
that the functions will differ on $K$, as demonstrated by the
example of $K$ being a subspace.\qed

We have shown in the proof of Theorem \ref{Thm_Cvx-Bij} that an
order preserving isomorphism $\T: Cvx(K_1)\to Cvx(K_2)$ defines a
bijection from $P_{K_1}$ to $P_{K_2}$. Similarly, an order reversing
isomorphism defines a bijection from $P_{K_1}$ to $Q_{K_2}$ (and
from $Q_{K_1}$ to $P_{K_2}$, of course), which is why we say $Q$ is
``dual'' to $P$.

\begin{prop}\label{Prop_Cvx-Rev} Let $n\ge1$, and let $K_1, K_2
\subseteq\R^n$ be convex sets with non empty interior, such that
either $K_1\neq\R^n$ or $K_2\neq\R^n$. Then there does not exist any
order reversing isomorphism $\T:Cvx(K_1)\to Cvx(K_2)$.
\end{prop}
\noindent {\bf Proof of Proposition \ref{Prop_Cvx-Rev}.} Without
loss of generality, assume $K_2\neq\R^n$ (otherwise consider
$\T^{-1}$). Let $x\in K_1$, then $\delta_{x,0}\in P_{K_1}$.
Therefore $\T(\delta_{x,0})\in Q_{K_2}$, which contradicts the
conclusion of Lemma \ref{lem:noQ}. \qed

\section{Geometric convex functions on a window}\label{Sect_Cvx0}

Recall the definition of geometric convex functions on a window:
\begin{defn} For a convex set $K\subseteq\R^n$ with $0\in K$, the
subclass of $Cvx(K)$ containing non negative functions
satisfying $f(0) = 0$ is called the class of {\em geometric convex
functions}, and denoted by $Cvx_0(K)$, i.e. \[Cvx_0(K) = \{f\in
Cvx(K): f\geq 0, f(0)=0\}.\] It is naturally embedded in
$Cvx_0(\R^n)$ by assigning to $f$ the value $+\infty$ outside $K$.
Therefore an equivalent definition is \[Cvx_0(K)= \{f\in Cvx(\R^n):
1_K \le f\le 1_{\{0\}}\}\] where $1_K$ denotes the {\em convex}
indicator function of $K$, which is zero on $K$ and $+\infty$
elsewhere, and similarly $1_{\{0\}}$. Note that these functions are
usually denoted by $1_K^\infty$, however, we never use in this paper
the standard characteristic functions, so this notation can not lead
to a misunderstanding.
\end{defn}

In this section we deal with order isomorphisms from $Cvx_0(K_1)$ to
$Cvx_0(K_2)$, where $K_i$ are convex sets (containing $0$, of
course), and some generalizations of these classes.

As the example of $\J$ in $Cvx_0(\R^n)$ (which was discussed in
Section \ref{Sect_Cvx0-Old}) shows us, the case of
$Cvx_0(K)$ is more involved than $Cvx(K)$, and a transform can be
more complicated than a mere fractional linear change in the domain
with the corresponding change in the fiber. Indeed, here we know
already of an example where an indicator function is not mapped to
such.

However, for the cases of $K = \R^+$ and $K = \R^n$ we do have
theorems of the sort, see Theorem \ref{Thm_Cvx0-n=1-Old} and Theorem
\ref{Thm_Cvx0-n>1-Old}. There, the transform {\em is} given by a
fractional linear point map on the epi-graphs. In each of these
cases we observe two different types of behavior; one where fibers
{\em are} mapped to fibers (a-la-$\I$), and one when they are not
(a-la-$\J$).

In this section we generalize these theorems to apply to an order
isomorphism $\T: Cvx_0(K_1)\to Cvx_0(K_2)$, for convex domains
$K_1$, $K_2$.

\begin{thm}\label{Thm_Cvx0-n>1-New} Let $n\ge2$, and let $K_1, K_2
\subseteq\R^n$ be convex sets with non empty interior. If
$\T:Cvx_0(K_1)\to Cvx_0(K_2)$ is an order preserving isomorphism,
then there exists a bijective fractional linear map $F: K_1\times
\R^+ \to K_2\times\R^+$, such that $\T$ is given by \[ epi(\T f) =
F(epi(f)).\]
\end{thm}

The case $n=1$ is slightly different since the two domains $\R^+$
and $\R^-$ do not interact.  Other than that, the result is the
same, for example see Theorem \ref{Thm_Cvx0-n=1-New}.

\begin{rem} Of course, it is not true that every fractional
linear map on $K_1 \times \R^+$ induces such a transform. A
discussion of which fractional linear maps do induce such a
transform (similar to that in Section \ref{Sect_Which-FL-Work-Cvx})
is given in Section \ref{Sect_Which-FL-Work-Cvx0}.
\end{rem}
\begin{rem}
In Section \ref{Sect_Which-FL-Work-Cvx0} we will also see that there
is a difference between the cases $0\in\partial K$ and $0\in int(K)$
, where in the former a ``$\J$-type'' transform does exist, and in
the latter it does not (except in the case $K_1=K_2=\R^n$).
\end{rem}

First, we will prove the one-dimensional theorem. We will do this in
two ways. The first (in Section \ref{Sect_Cvx0-1D}) is by using the
known uniqueness Theorem \ref{Thm_Cvx0-n=1-Old} for $\J$ and $\I$.
The second is a direct proof, which we postpone to Section \ref
{Sect_Cvx0-1D-ind}. We add this second proof for two purposes; to
make the paper self contained, and also to clarify the case of a
transform $\T: Cvx_0([0, x_1])\to Cvx_0([0, x_2])$, that is when
the domain of all functions is bounded.

Second, we will prove the multi-dimension theorem, in the following
stages: we show that the transform must act ``ray-wise''. Then, on
each ray, we {\em could} already apply the one-dimensional
conclusion, but in fact we need much less - thus we continue
directly and show that two extremal families of functions, namely
linear functions and indicator functions, determine the
full shape of $\T$. The extremality property forces the transform to
act bijectively on these two families, and in a monotone way. Here,
we do not need to discover the exact rule of this monotone mapping
(even though we have it, since we've solved the one dimensional
case). Instead, we prove that there is some point map on the
epi-graphs, controlling the rule of the transform for a third
family, namely triangle functions. We show that this point map is
interval preserving, and then apply Theorem \ref{Thm_FL-Unique-nD}
to show that it is fractional linear. Finally, we show that the rule
of the transform for triangles determines the whole transform, thus
completing the proof. This plan follows the proof from
\cite{hidden-structures} of the case $K_1=K_2=\R^n$.

\subsection{Dimension one}\label{Sect_Cvx0-1D}
In \cite{hidden-structures}, the first and third named authors
showed that essentially, any order isomorphism $\T :Cvx_0(\R^+) \to
Cvx_0(\R^+)$ is either $\I$ or $\J$, see Theorem \ref{Thm_Cvx0-n=1-Old}.
We note that in this case, indeed, for each of these two families of
transforms, the transform is induced by a point map on the
epi-graphs which is fractional linear. The first family of
transforms (a-la-$\I$) is given by \[ (\T \phi)(x) = \beta \phi
(x/\alpha),\] for positive $\alpha$ and $\beta$, and the inducing
maps are $F_{\alpha,\beta}^\I(x,y) = (\alpha x,\beta y)$. The second
family of transforms (a-la-$\J$) is given by \[ (\T \phi)(x) = \beta
(\J \phi) (x/\alpha),\] for positive $\alpha$ and $\beta$, and the
inducing maps are $F_{\alpha,\beta}^\J(x,y) = \left(\frac{\alpha
x}{y}, \frac{\beta}{y}\right)$.

We introduce a third transform, with a parameter $z>0$, to be able
to switch between the bounded and non bounded cases;

\begin{defn} Let $z>0$, and $F_z:[0, z)\times\R^+ \to\R^+\times\R^+$
be the bijective fractional linear map defined by $F_z(x,y) =
\left(\frac{x}{z-x}, \frac{y}{z-x}\right)$.
\end{defn}

\begin{lem} $F_z$ induces an order isomorphism $\T_z:Cvx_0([0, z))
\to Cvx_0(\R^+)$ by its action on epi-graphs, that is
\[epi(\T_z(f)) = F_z(epi(f)).\]
\end{lem}

\noindent {\bf Proof.} To see that a transform defined using a point
map on the epi-graphs, is an order isomorphism, three things need to
be checked; that it is well defined, that it is bijective, and that
it preserves order in both directions. For $\T_z$ to be well
defined, $F_z$ must map epi-graphs of geometric convex functions to
epi-graphs of geometric convex functions. Since $F_z$ is fractional
linear, it is interval preserving, thus a convex epi-graph is mapped
to some convex set. Among all convex sets, epi-graphs of geometric
convex functions are characterized by two inclusions; \[\{(0, y) :
y>0\}=epi(1_{\{0\}})\subseteq epi(f)\subseteq epi(1_K)
=\{(x, y):x\in K, y>0\}.\] Note that $F_z$ maps the half line $\{(0,
y) : y>0\}$ onto itself, and the entire domain $[0, z)\times\R^+$
onto the image $\R^+\times\R^+$. Therefore also $F_z(epi(f))$ is
between these two sets, which means it is the epi-graph of some
geometric convex function. Bijectivity of $F_z$ implies bijectivity
of $\T_z$. Since $f\le g\Leftrightarrow epi(g)\subseteq epi(f)$, a
transform induced by a bijective point map on the epi-graphs,
automatically preserves order in both directions. \qed

We are ready to prove the one dimensional theorem, dealing with
$I_1, I_2\subseteq\R$ which may be either bounded intervals or half
lines.

\begin{thm}\label{Thm_Cvx0-n=1-New} Let $I_1\subseteq\R$ be either of
the form $I_1=[0,x_1)$ for some positive $x_1$, or $I_1=[0,\infty)$,
and likewise $I_2$. If $\T : Cvx_0(I_1) \to Cvx_0(I_2)$ is an order
isomorphism, then there exists a bijective fractional linear map $F:
I_1\times \R^+ \to I_2\times \R^+$, such that $\T$ is given by \[
epi(\T f) = F(epi(f)).\]
\end{thm}
\noindent {\bf Proof of Theorem \ref{Thm_Cvx0-n=1-New}.} Define
$\tilde{\T}:Cvx_0(\R^+)\to Cvx_0(\R^+)$ in the following way:\\
If $I_1=[0, x_1   )$ and $I_2=[0, \infty)$, then: $\tilde{\T}:=               \T \circ \T_{x_1}^{-1}$\\
If $I_1=[0, x_1   )$ and $I_2=[0, x_2   )$, then: $\tilde{\T}:=\T_{x_2} \circ \T \circ \T_{x_1}^{-1}$\\
If $I_1=[0, \infty)$ and $I_2=[0, \infty)$, then: $\tilde{\T}:=               \T                    $\\
If $I_1=[0, \infty)$ and $I_2=[0, x_2   )$, then: $\tilde{\T}:=\T_{x_2} \circ \T                    $\\
$\tilde{\T}$ is clearly an order isomorphism. Next, by simply
applying Theorem \ref{Thm_Cvx0-n=1-Old}, we get that our original
$\T$ is some composition of the transforms $\I$, $\J$, $\T_z$, and
$\T_z^{-1}$, which are all induced by fractional linear point maps
on the epi-graphs. Thus we conclude that $\T$ is also induced by
such a map. \qed

\begin{rem}\label{Rem_Cvx0-For-[0,z]} For transforms on (or to)
$Cvx_0([0,z])$ simply note that all elements of $Cvx_0([0,z))$ are
non decreasing and lower-semi-continuous functions, and thus have a
unique extension to $[0,z]$, which preserves order in both
directions. Therefore, by embedding $Cvx_0([0,z))=Cvx_0([0,z])$
(where $f$ is mapped to its unique extension) we get an order
isomorphism of the form described in Theorem \ref{Thm_Cvx0-n=1-New},
and thus have the same result for closed intervals $[0,z]$, where
epi-graphs are taken {\em without} the point $z$. In particular we
see that there exist order isomorphisms between $Cvx_0([0,z])$ and
$Cvx_0(\R^+)$.
\end{rem}

\subsubsection{Table of one dimension transforms}
Straightforward computation of the transform in each of the cases
gives, in each of the four scenarios, two types of transforms;
a-la-identity and a-la-$\J$. We list them here, indicated by the
fractional linear maps which induce them, namely $F_{a,b}:
I_1\times\R^+\to I_2\times\R^+$. Each family is two-parametric, for
convenience we choose the parameters $a,b$ such that $a, b> 0$ gives
exactly all the functions in the family:
\begin{center}
    \begin{tabular}{ | l | l | l | p{5cm} |} \hline
    $\quad I_1$ & $\quad I_2$ &
    $            \mbox{a-la-}{\I};\quad F_{a,b}(x,y)$ &
    $\qquad\quad \mbox{a-la-}{\J};\quad F_{a,b}(x,y)$\\ \hline

    $[0, x_1)$ & $[0, x_2)$ & $\frac{x_2}{x(1-a)+x_1a}\cdot
    \left(\begin{array}{c}x\\by\\ \end{array}\right)$ &
    $\qquad\frac{bx_2}{bx+y}\cdot
    \left(\begin{array}{c}x\\a(x_1-x)\\ \end{array}\right)$\\ \hline

    $[0, x_1)$ & $[0, \infty)$ & $\qquad\frac{a}{x_1-x}\cdot
    \left(\begin{array}{c}x\\by\\ \end{array}\right)$ &
    $\qquad\quad\frac{b}{y}\cdot
    \left(\begin{array}{c}x\\a(x_1-x)\\ \end{array}\right)$\\ \hline

    $[0, \infty)$ & $[0, x_2)$ & $\qquad\frac{ax_2}{ax+1}\cdot
    \left(\begin{array}{c}x\\by\\ \end{array}\right)$ &
    $\qquad\quad\frac{bx_2}{bx+y}\cdot
    \left(\begin{array}{c}x\\a\\ \end{array}\right)$\\ \hline

    $[0, \infty)$ & $[0, \infty)$ & $\qquad\quad a\cdot
    \left(\begin{array}{c}x\\by\\ \end{array}\right)$ &
    $\qquad\qquad\frac{b}{y}\cdot
    \left(\begin{array}{c}x\\a\\ \end{array}\right)$\\ \hline

    \end{tabular}
\end{center}

There is an essential difference between the $\I$-type and $\J$-type
transforms; they handle differently the extremal elements of
$Cvx_0(I)$, which are indicators and linear functions (see Section
\ref{Sect_Extrm} for exact definitions). The $\I$-type transforms
map indicators to themselves (bijectively), and likewise linear
functions. The $\J$-type transforms, however, interchange between
the two sub-families, mapping indicators to linear functions
(bijectively) and vice versa. In the inducing maps, we also have a
natural distinction between the $\I$-type and $\J$-type maps. In
both cases the determinant of the Jacobian of the inducing map never
vanishes; it is positive for $\I$-type maps, and negative for
$\J$-type maps.

\subsection{Multi dimension}
\subsubsection{Acting on rays}
We next prove that in the $n$-dimensional case, one merely deals
with many copies of the one dimensional problem (in fact, the case
of functions on $\R^+$).

The next lemma states that an order isomorphism basically works in
the following way: first, there is a permutation on the rays, and
then on each ray, the transform acts independently of the functions'
values on other rays.

There are two nuances here; first, if $K\neq\R^n$, then in some
directions it does not contain a {\em full} ray. Since this does not
affect the argumentation in any way, we don't distinguish between a
full ray ($\R^+z$) and a restricted ray ($\R^+z\cap K$), which may
be a bounded interval, and use ``ray'' to describe both. Second, if
$0\in int(K)$, then the set of all relevant rays can be described by
$S^{n-1}$, but if $0\in\partial K$, then there are less relevant
rays (in some directions $z$, $\R^+z\cap K = \{0\}$). Therefore we
are again forced to add another definition, for the set of all
relevant rays - $\S(K)\subseteq S^{n-1}$. $\S(K):=\{z\in S^{n-1} :
\R^+z\cap K\neq \{0\}\}$. In what follows, the support of a function
is defined to be (the closure of) the set on which it is finite;
$\overline{\{x:f(x)<\infty\}}$

\begin{lem}\label{Lem_Raywise} Let $n\ge2$, and let $K_1, K_2
\subseteq\R^n$ be convex sets with non empty interior. If
$\T:Cvx_0(K_1)\to Cvx_0(K_2)$ is an order preserving isomorphism,
then there exists a bijection $\Phi:\S(K_1)\to\S(K_2)$, such that
any function supported on $\R^+y$ is mapped to a function supported
on $\R^+z$, for $z = \Phi(y)$. Moreover, $\T$ acts ray-wise, namely
$(\T f)|_{\R^+z}$ depends only on $f|_{\R^+y}$, for $z = \Phi(y)$.
\end{lem}

We remark that if we were to prove the theorem directly for order
reversing transformations then we would not encounter this ray-wise
behavior, and get a transform $\A$ (or $\L$) which, miraculously,
when combined with $\L$ acts ray-wise. Later on, it will follow that
$\Phi$ must be induced by a linear map.

The proof uses the following simple observation: if $x,y\in\S(K_1)$
are two different points, and $f_x, f_y\in Cvx_0(K_1)$ are two
functions supported on $\R^+x$ and $\R^+y$ respectively, then
$\max\{   f_x,    f_y\} = 1_{\{0\}}$, and thus also $\max\{\T f_x,
\T f_y\} = 1_{\{0\}}$, which means that $\T f_x$ and $\T f_y$ are
supported on different sets.

\noindent {\bf Proof of Lemma \ref{Lem_Raywise}.} For two functions
$f,g$ to have $\max \{f,g\} = 1_{\{0\}}$ they must be supported on
two sets whose intersection equals $\{0\}$. A function with support
in a line cannot be mapped to one whose support includes two
positively-linearly-independent points because then $\T^{-1}$ would
map two functions whose support intersects at $\{0\}$ only, to
functions supported on the same ray - impossible. Thus functions
supported on a given ray are all mapped to functions supported on
another fixed ray. By invertibility, we get that this defines a
mapping $\Phi:\S(K_1)\to\S(K_2)$ which is bijective.

As for the ray-wise action of $\T$, the values of $\T f$ on $\R^+z$
are the same as the values of $\max \{\T f, R_z\}$, where $R_z$
denotes the function which is $0$ on $\R^+z\cap K_2$ and $+\infty$
elsewhere. This maximum is the image of the function
$\max\{f,R_y\}$, because $\T R_y = R_z$ (each being the smallest
function supported on the corresponding ray). Since $\max \{f,
R_y\}$ does not depend on the values $f$ attains outside $\R^+y$,
our claim follows. \qed

\subsubsection{Extremal elements and monotonicity}\label{Sect_Extrm}
Restricted to a ray $I$, we consider two families of extremal
functions in $Cvx_0(I)$; indicator functions, and linear functions.

a) $1_{[0,z]}$ which equals to $0$ on $[0,z]$ and $+\infty$
elsewhere (indicator).

b) $l_c(t) = \max\{ct, 1_I(t)\}$ (linear).

\noindent Formally, the function $l_c$ is defined on the whole of
$\R^n$, therefore it is not really linear, but we will use this name
in short. All the $\J$-type transforms switch (a) and (b) -
bijectively, and all the $\I$-type transforms fix (a) and fix (b) -
again, bijectively. We will show that this is no coincidence - a
general order isomorphism $\T$ must act in one of these two ways. We
derive this from two properties of these families - the extremality
property, and the non-comparability relation between these two
families.

\begin{defn} A function $f \in Cvx_0(I)$ is called {\em extremal} if there
exist no two functions $g,h \in Cvx_0(I)$ such that $g\not\ge f$ and
$h \not \ge f$ but $\max\{g,h\}\ge f$.
\end{defn}
In the language of epi-graphs, this means that for $epi(f)$ to
contain $A\cap B$, it must contain either $A$ or $B$ - whenever $A,
B$ are also epi-graphs of geometric convex functions.

We claim that extremality {\em characterizes} indicator and linear
functions in $Cvx_0(I)$:
\begin{lem}\label{Lem_Only-Extremal} The only extremal functions in
$Cvx_0(I)$ are either of the form $1_{[0,z]}$ for some
$z\in I$ or of the form $l_c$ for some $c\in\R^+$.
\end{lem}

\noindent {\bf Proof of Lemma \ref{Lem_Only-Extremal}.} It is easy
to check that both families are extremal. To show that any extremal
function $f \in Cvx_0(I)$, must be of one of the two forms, we first
show that if it assumes some value $0<c\neq \infty$, it must be
linear. Indeed, let $f(x) = c$. Without loss of generality we may
assume $x\in int(I)$, since $f$ is lower-semi-continuous. Consider
the function $1_{[0,x]}$ assuming $0$ in the interval $[0,x]$ and
$+\infty$ elsewhere; $f \not\le 1_{[0,x]}$, since $1_{[0,x]}(x) = 0
< f(x)$. Consider the function $L_x(y) = \frac{c}{x}y$. By convexity
of $f$, on the interval $[0,x]$, $f \le L_x$. Since outside $[0,x]$
we have $f \le 1_{[0,x]}$, this implies $f\le \max\{1_{[0,x]},
L_x\}$, and so by extremality it must be that $f \le L_x$. Since $x$
is in the interior of $I$, this means that $f=L_x$, and therefore
$f$ is linear. The only other option is that $f$ assumes only the
values $0$ and $+\infty$, which implies it is an indicator function,
by convexity. \qed

\begin{lem}\label{Lem_Dichotomy} If $\T :Cvx_0(I_1) \to Cvx_0(I_2)$
is an order isomorphism then either:

$\T$ is a bijection from linear functions to indicators, and a
bijection from indicators to linear functions, or:

$\T$ is a bijection from linear functions to themselves, and a
bijection from indicator functions to themselves.
\end{lem}

\noindent {\bf Proof.} Extremality is preserved under $\T$. Indeed,
if there exist two functions $g,h \in Cvx_0(I_2)$ such that
$g\not\ge \T f$ and $h \not \ge \T f$ but $\max\{g,h\}\ge \T f$,
then the functions $\T^{-1}g$ and $\T^{-1}h$ contradict extremality
for $f$. So, we see that the family of all extremal functions is
mapped to itself, and by Lemma \ref{Lem_Only-Extremal} this family
is exactly the union of linear and indicator functions. Since
$\T^{-1}$ shares the same properties as $\T$, we see that the map is
surjective.

Secondly, all linear functions are comparable to one another and all
indicator functions are comparable to one another (by $f$ and $g$
comparable we mean that either $f \le g$ or $g \le f$). However, no
indicator function is comparable to a linear function - except for
the trivial examples of $1_{\{0\}}$ and $0$, whose behavior is
obvious - since in $Cvx_0(I)$, these are the maximal and minimal
elements (they are also the only mutual elements in both families).
Hence, once we know that one linear function is mapped to a linear
function then all of them must be, and then all indicator functions
are mapped to indicators. The alternative is of course that all
linear functions are mapped to indicators, and then all indicators
are mapped to linear functions. \qed

In this last Lemma, a dichotomy, not apparent at first sight,
appears. We have two very different possibilities, one corresponding
to $\I$, the identity transform (which clearly maps linear functions
to themselves, likewise for indicator functions), and the other
possibility corresponds to the transform $\J$, which - as can be
checked - maps linear functions to indicator functions and
vice-versa. Despite this dichotomy, in the statement of the next
lemma we do not need to separate the two cases.

Next we claim that $\T$ is a {\em monotone} bijection on each of the
extremal families. Monotonicity has a meaning here since both
families are fully ordered subsets of $Cvx_0(I)$ - ``chains'' -
bounded together by the minimal and maximal elements $f_0\equiv0$
and $f_\infty=1_{\{0\}}$.

If $\T$ maps linear functions to themselves (and likewise indicator
functions), we define $S:I_1\to I_2$ to be the function for which
$\T 1_{[0,x]} = 1_{[0,S(x)]}$, and $A:\R^+\to\R^+$, for which
$\T(l_c) = l_{A(c)}$. If $\T$ interchanges between the two families,
we define $S:I_1\to\R^+$ to be the function for which $\T 1_{[0,x]}=
l_{S(x)}$, and $A:\R^+\to I_2$, for which $\T(l_c) = 1_{[0,A(c)]}$.
In this next simple lemma we formulate the monotonicity property:

\begin{lem}\label{Lem_Monotone-Bij} Assume $\T :Cvx_0(I_1) \to
Cvx_0(I_2)$ is an order isomorphism.

If $\T$ maps linear functions to themselves, then $S$ and $A$ are
increasing bijections.

If $\T$ interchanges between the two families, then $S$ and $A$ are
decreasing bijections.
\end{lem}
\noindent{\bf Proof.} $S$ and $A$ are bijections, since $\T$ is a
bijection. Note that $1_{[0,x]} \le 1_{[0,y]}  \Leftrightarrow x\ge
y$ and $\l_c \le l_d  \Leftrightarrow c\le d$. Therefore, if $\T$
fixes each of the families, $S$ and $A$ are increasing, and if $\T$
switches between the families, $S$ and $A$ are decreasing. \qed

\subsubsection{Triangles functions - completing the proof}
Next, we handle another family of functions, ``triangle'' functions.
We show it is preserved under $\T$, and that the rule of the
transform for it is monotone. We show that when leaving the
one-dimensional perspective, the rule of the transform for triangles
is controlled by an interval preserving bijection; and thus we apply
our uniqueness theorem for such maps, Theorem
\ref{Thm_FL-Unique-nD}. Finally we show that the transform is
determined by its behavior on triangles, which proves Theorem
\ref{Thm_Cvx0-n>1-New}.

For $z\in K$ and $c\in\R^+$, we introduce the ``triangle'' functions,
denoted  $\lhd_{z,c}\in Cvx_0(K)$:
\[ \lhd_{z,c}(x) =
\left\{
  \begin{array}{ll}
     c|x|, & \hbox{if~~} x \in [0,z] \\
    +\infty, & \hbox{otherwise}.
  \end{array}\right. \]

Note that they are one-dimensional (i.e. supported on a ray), so
they can be thought of as elements of $Cvx_0(I)$ where $I$ is a ray,
and then $\lhd_{z,c} = \max\{1_{[0,z]}, l_c\}$.

\begin{lem}\label{Lem_Triangle-To-Triangle} If $\T :Cvx_0(I_1) \to
Cvx_0(I_2)$ is an order isomorphism then a triangle function
$\lhd_{z,c}$ is mapped under $\T$ to a triangle function
$\lhd_{z',c'}$, where $(z',c')$ is a function of $(z,c)$.
\end{lem}

\noindent {\bf Proof of Lemma \ref{Lem_Triangle-To-Triangle}.} A
triangle is the maximum of an indicator and a linear function. By
Proposition \ref{Prop_Sup-To-Sup} $\T$ respects $\sup$ and
$\hat{\inf}$, and thus in both cases of Lemma \ref{Lem_Dichotomy}, a
triangle is mapped to the maximum of an indicator and a linear
function; that is, to a triangle. \qed

\noindent{\bf Remark.} Since in Lemma \ref{Lem_Monotone-Bij} we
showed that $\T$ maps indicator and linear functions in a monotone
way, it is obvious that this is the case also for triangles, meaning
either $\T(\lhd_{z, c})=\lhd_{S(z), A(c)}$, or $\T(\lhd_{z,
c})=\lhd_{A(c), S(z)}$, and in both cases, fixing any of the
parameters $z$ or $c$ and changing the other monotonously, changes
also the triangle in the image monotonously. Since we already know
the exact shape of 1D transforms, we could have concluded this
immediately. However, in what follows, we only use the fact that
$\T(\lhd_{x,c})=\lhd_{y,d}$, and that this map is monotone, meaning
that on a fixed ray, either $y=y(x)$, $d=d(c)$, and both functions
are bijective and increasing, or $y=y(c)$, $d=d(x)$, and both
functions are bijective and decreasing.

We return to the $n$-dimensional picture, the first time since we
reduced the discussion to ray-wise action. We wish to see how the
different mappings of triangles on different rays all fit together.
To this end, we replace the ``parametrization'' of triangles, from
the point $z$ (indicating the support of the function) and the slope
$c$, to the point $z$ and the {\em value of the function at that
point} $h=c|z|$. To avoid abuse of notation, for $h=c|z|$ we will
denote $\lhd_{z,c}$ by $\lhd^{z,h}$. With this notation, we denote
by $F: (K_1\setminus\{0\})\times\R^+\to
(K_2\setminus\{0\})\times\R^+$ the bijective map for
which $\T\lhd^{z,h} = \lhd^{F(z,h)}$. 

\begin{prop}\label{Prop_Triangle-Rule-Is-FL}  Let $n\ge2$,  $K_1, K_2
\subseteq\R^n$ convex sets with non empty interior, and
$\T:Cvx_0(K_1)\to Cvx_0(K_2)$  an order preserving isomorphism.
Assume $F: (K_1\setminus\{0\})\times\R^+\to
(K_2\setminus\{0\})\times\R^+$ is the bijection satisfying
$\T(\lhd^{x,h}) = \lhd^{F(x,h)}$ for every $(x,h)\in
(K_1\setminus\{0\})\times\R^+$. Then $F$ is a fractional linear map.
\end{prop}

\noindent {\bf Proof of Proposition \ref{Prop_Triangle-Rule-Is-FL}.}
First we show that the restriction of $F$ to any domain for which
$(0,0)$ is an extreme point, is fractional linear. Let $(x_1, h_1),
(x_2, h_2)\in K_1\times\R^+$ such that $0\not\in [x_1, x_2]$. This
merely means that our argument does not hold if $x_1$ and $x_2$ are
on opposite rays. Letting $(x_3, h_3)\in [(x_1, h_1), (x_2, h_2)]$,
and denoting $F(x_i, h_i)= (y_i, l_i)$, we need to prove that
$(y_3,l_3)\in[(y_1, l_1), (y_2, l_2)]$. If $x_i$ are on the same
ray, then it follows from the one dimensional case, handled in
Section \ref{Sect_Cvx0-1D-ind}, that the restriction of $F$ to this
line is fractional linear, and in particular it maps intervals to
intervals, that is $F([(x_1, h_1),(x_2, h_2)])=[(y_1, l_1),
(y_2, l_2)]$. Assume otherwise, that $x_i$ are linearly independent.
Note that $\lhd^{x_3, h_3}\ge
\hat{\inf}\{\lhd^{x_1, h_1}, \lhd^{x_2, h_2}\}$, and that in this
inequality $x_3$ is maximal, and $h_3$ is minimal. Therefore
$\lhd^{y_3, l_3}\ge \hat{\inf}\{\lhd^{y_1, l_1}, \lhd^{y_2, l_2}\}$,
and in {\em this} inequality - due to the monotonicity of $\T$ on
triangles - again $y_3$ is maximal, and $l_3$ is minimal (recall
that in Lemma \ref{Lem_Monotone-Bij} we saw that if indicators and
linear functions are exchanged, $S$ and $A$ are decreasing, and if
they are preserved, $S$ and $A$ are increasing - thus in any case
maximality of $x_3$ and minimality of $h_3$ coincides with
maximality of $y_3$ and minimality of $l_3$). Therefore $y_3$, which
lies on a different ray than those of $y_1, y_2$ ($\Phi$ is
bijective), is in the triangle with vertices $0, y_1, y_2$, and due
to its maximality - $y_3\in[y_1, y_2]$. Moreover, the point $(y_3,
l_3)$ is above or on the interval $[(y_1, l_1), (y_2, l_2)]$, and
due to its minimality, it is {\em on} this line. Therefore $(y_3,
l_3)\in [(y_1, l_1), (y_2, l_2)]$, which means that $F$ preserves
intervals which do not intersect the positive $h$-axis;
$\{(0,h):h\ge0\}$. In other words, the restriction of $F$ to any
domain for which $(0,0)$ is an extreme point, is interval
preserving. By applying Theorem \ref{Thm_FL-Unique-nD}, we conclude
that $F$ is fractional linear on each such domain, and thus, since
$n\ge2$, we may use Corollary \ref{Lem_Mizdahim-Interior} to
conclude that $F$ is a fractional linear map on the whole of
$(K_1\setminus\{0\})\times\R^+$. \qed

\noindent{\bf Remark.} The proof of Proposition
\ref{Prop_Triangle-Rule-Is-FL} does {\em not} work in one dimension,
since the only two rays; $\R^+,\R^-$ cannot interact - they have $0$
in their convex hull, and therefore a direct proof is needed in this
case, to show that the transform is given by a fractional linear map
on the epi-graphs. In fact, while it is true for transforms on a
ray, it is indeed {\em not} the case for transforms on $Cvx_0(\R)$,
or on $Cvx_0(I)$ where $I$ is an interval containing $0$ in the
interior.

\noindent{\bf Remark.} The function $F$ which is defined formally
only for $(x,h)\in (K\setminus\{0\})\times\R^+$, can in fact be
extended to $K\times\R^+$, since the defining hyperplane of $F$ does
not intersect $epi(1_{\{0\}})= \{(0,h):h>0\}$. Indeed, it is obvious
that if it intersects this ray in one point it must contain the
whole ray. In such a case, it follows from the properties of
fractional linear maps, that rays emanating from a point in the
hyperplane are mapped to parallel rays emanating from the
hyperplane. Such a point map does not induce a transform on
$Cvx_0(K)$. Therefore $F$ can be defined on the whole of
$K\times\R^+$. Moreover, using the fact that the supremum of all
triangles is $1_{\{0\}}$, we get that $F(epi(1_{\{0\}})) =
epi(1_{\{0\}})$. 

Finally, knowing that the transform
rule for triangle functions is controlled by a fractional linear map
$F$, we turn to see that this is also the case for the {\em
epi-graph} of any function. We use the following simple equality
$epi(f)= \{(x,h)\in (K\setminus\{0\})\times\R^+: \lhd^{x,h}>f\}\cup
epi(1_{\{0\}}) $ which holds for every $f\in Cvx_0(K)$. 

\noindent {\bf Proof of Theorem \ref{Thm_Cvx0-n>1-New}.} By the
previous proposition, there exists a bijective fractional linear map
$F:(K_1\setminus\{0\})\times\R^+ \to (K_2\setminus\{0\})\times\R^+$,
and we need to show that $F(epi(f))=epi(\T f)$. \begin{align*}
  epi(\T                     f                       )  &=\hskip 12pt
  \{(y,l)\in (K_2\setminus\{0\})\times\R^+:   \lhd^{y,l}   >\T f\} \hskip 21pt\cup   epi(1_{\{0\}})  \\&=
F(\{(x,h)\in (K_1\setminus\{0\})\times\R^+:   \lhd^{F(x,h)}>\T f\})           \cup F(epi(1_{\{0\}})) \\&=
F(\{(x,h)\in (K_1\setminus\{0\})\times\R^+: \T\lhd^{x,h}   >\T f\})\hskip 05pt\cup F(epi(1_{\{0\}})) \\&=
F(\{(x,h)\in (K_1\setminus\{0\})\times\R^+:   \lhd^{x,h}   >   f\})\hskip 21pt\cup F(epi(1_{\{0\}})) \\&=
F(epi(                    f                         ))
\end{align*} \qed

\subsection{Additional results}
\subsubsection{Direct uniqueness proof in the one dimensional
bounded case}\label{Sect_Cvx0-1D-ind} We focus on the possible
transforms in the case where linear functions are mapped to
themselves, likewise indicator functions. Clearly, the function $S:
I_1 \to I_2$ for which we have that $\T 1_{[0,x]} = 1_{[0,S(x)]}$ is
bijective and increasing (so it is continuous as well). Similarly
$A: \R^+\to \R^+$, for which $\T(l_c) = l_{A(c)}$, is bijective,
increasing, and continuous. Note that we deal now only with $I_1$
and $I_2$ which are bounded, which means that $S$ maps an interval
to an interval, and $A$ maps a full ray to a full ray.

\begin{lem}\label{Lem_Rule-Of-I} Let $I_1=[0, x_1), I_2=[0, x_2)$,
where $x_i\in\R$ are two positive numbers. Let $\T :Cvx_0(I_1)\to
Cvx_0(I_2)$ be an order preserving isomorphism. Assume further, that
for some increasing bijective function $S:I_1\to I_2$ we have $\T
1_{[0,x]} = 1_{[0,S(x)]}$, and for another
increasing bijective function $A: \R^+\to \R^+$, we have that $\T
l_c = l_{A(c)}$. Then there exist two constants $\alpha>0$ and $d<1$
such that $A(c) = \alpha c$ and $S(x) =
\frac{x_2}{x_1}\cdot\frac{x}{d(x/x_1-1)+1}$.
\end{lem}

\noindent {\bf Proof of Lemma \ref{Lem_Rule-Of-I}.} Denote as before
$\lhd_{x,c} = \max\{1_{[0,x]}, l_c\}$, and similarly $g_{x,c} =
\hat{\inf}\{1_{[0,x]}, l_c\}$. We get (on $I_2$ replace $x_1$ by
$x_2$):
\[g_{x,c}(z) =
\left\{
  \begin{array}{ll}
                0            & ;\hbox{if~~} z \in [0,x] \\
     c(z-x)\frac{x_1}{x_1-x} & ;\hbox{if~~} z \in [x,x_1] \\
              +\infty        & ;\hbox{otherwise}
  \end{array}
\right.,\,\,  \lhd_{x,c}(z) = \left\{
  \begin{array}{ll}
       cz   & ;\hbox{if~~} z \in [0,x] \\
    +\infty & ;\hbox{otherwise}
  \end{array}
\right.. \] By Proposition \ref{Prop_Sup-To-Sup} we get
$\T(\lhd_{x,c})=\lhd_{S(x),A(c)}$, $\T(g_{x,c})=g_{S(x),A(c)}$. Let
$0<t<1$ and consider $g=g_{tx,
(\frac{c}{1-t})(\frac{x_1-tx}{x_1})}$. It can be easily checked that
$g \le \lhd_{x,c}$, and $g(x)=\lhd_{x,c}=cx$, so that $g \not\le
\lhd_{x,c'}$ for any $c'<c$, and $g \not\le\lhd_{x',c}$ for any
$x'>x$. In fact, when a $g$-type function and a $\lhd$-type function
behave that way ($g\le\lhd_{x,c}$ with maximal $x$ and minimal $c$)
it must be that they are equal at the ``breaking point of the
triangle'', i.e. at the point $x$. Since $\T$ preserves order in
both directions, $\T(g)$ and $\T(\lhd_{x,c})$ behave in the same
way, and therefore:
\[g_{S(tx), A((\frac{c}{1-t})(\frac{x_1-tx}{x_1}))}\le\lhd_{S(x), A(c)}\]
with equality between the two functions at the point $S(x)$,
meaning:
\[A\left(c\frac{x_1-tx}{x_1-tx_1}\right) \cdot (S(x)-S(tx))
 \cdot \left(\frac{x_2}{x_2-S(tx)}\right)=  A(c)S(x)\]
for every $0<t<1$, every $0<x<x_1$, and every $0<c$. By defining
$u=\frac{x_1-tx}{x_1-tx_1}$ and rearranging the equation, we get:
\begin{equation}\label{Eq_Separate-A-And-S}
\frac{A(cu)}{A(c)}=\left(\frac{S(x)}{S(x)-S(tx)}\right)\cdot
\left(\frac{x_2-S(tx)}{x_2}\right). \end{equation}

In particular, the ratio $\frac{A(cu)}{A(c)}$ does not depend on $c$
- thus it is equal to $\frac{A(u)}{A(1)}$, and we may write
\begin{equation}\label{Eq_Just-A-Of-c} A(cu) =
\frac{A(c)A(u)}{A(1)},\end{equation} which holds for all $0<c$ and
$1<u$ (see the definition of $u$). For $u=1$ it is true trivially.
For $0<u<1$ we denote $u':=1/u>1$. Noticing the symmetry between $u$
and $c$, we interchange their roles to see that $A(1)=
\frac{A(u)A(u')}{A(1)}$, and write

\[\frac{A(cu)}{A(c)} = \frac{1}{\frac{A(cu\cdot u')}{A(cu)}} =
 \frac{1}{\frac{A(u')}{A(1)}} = \frac{1}{\frac{A(1)}{A(u)}} =
 \frac{A(u)}{A(1)}.\]
%

Equation (\ref{Eq_Just-A-Of-c}), valid for all $c>0, u>0$, together
with the continuity of $A$, implies that $A$ is of the form \[A(c) =
\alpha c^\gamma\] for some fixed $\alpha>0$ and $\gamma$.

Therefore, $\frac{A(cu)}{A(u)}=u^\gamma$. Returning to equation
\eqref{Eq_Separate-A-And-S} with this new information, and
substituting $u=\frac{x_1-tx}{x_1-tx_1}$, we get
\begin{equation}\label{Eq_Just-S-Of-x}
\left(\frac{x_1-tx}{x_1-tx_1}\right)^\gamma =
\left(\frac{S(x)}{S(x)-S(tx)}\right)
\cdot\left(\frac{x_2-S(tx)}{x_2}\right).
\end{equation}

This can be written also as \[S(tx)=S(x)\cdot\left(\frac
{x_2\left(\frac{x_1-tx}{x_1-tx_1}\right)^\gamma-x_2}
{x_2\left(\frac{x_1-tx}{x_1-tx_1}\right)^\gamma-S(x)}\right),\] to
show that for a given $0<x<x_1$, $f(t):=S(tx)$ is differentiable as
a function of $t$, for all $0<t<1$. This means $S$ is differentiable
in $(0,x_1)$ (the interior of $I_1$).
%
%

Denote $D_{a,b}=\frac{S(b)-S(a)}{b-a}$ for $a,b\in [0,x_1]$, and
similarly $D_{a,a}=S'(a)$ for $a\in (0,x_1)$, so that $D_{a,b}\to
D_{a,a}$ when $b\to a$. Note that for $a\not=b$, $0<D_{a,b}<\infty$.
Rearranging equation \eqref{Eq_Just-S-Of-x} yields:
\[\left(\frac{x_1-tx}{x_1-tx_1}\right)^{\gamma-1} =
\frac{D_{0,x}\cdot D_{tx,x_1}}{D_{0,x_1}\cdot D_{tx,x}}.\]

Choose $x<x_1$ such that $S'(x)\not=0$ and let $t\to1^-$, then the
right hand side of the equation tends to a finite, strictly positive
number, and since $\frac{x_1-tx}{x_1-tx_1}\to\infty$ when $t\to 1$,
this implies $\gamma=1$. Therefore for every $0<t<1, 0<x<x_1$ we
have:
\[D_{0,x_1}\cdot D_{tx,x} = D_{0,x}\cdot D_{tx,x_1}\]
or alternatively: \[[0,tx,x,x_1]=[S(0),S(tx),S(x),S(x_1)]\] which by
Theorem \ref{Thm_FL-Unique-1D} implies that $S$ is fractional
linear. Combined with $S(0)=0$, $S(x_1)=x_2$, $S'(x)>0$, and
$\gamma=1$, this implies that $S$ and $A$ each belongs to a
one-parametric family of maps of the form \[S(x) =
x_2\cdot\frac{x/x_1}{d(x/x_1-1)+1}\qquad A(c) = \alpha c\] where
$d<1$ and $\alpha<0$. \qed

\subsubsection{Classification of admissible fractional linear
maps}\label{Sect_Which-FL-Work-Cvx0} We wish to fully classify the
type of fractional linear maps that induce transforms as in Theorem
\ref{Thm_Cvx0-n>1-New}. (The one dimensional case was fully
described in Section \ref{Sect_Cvx0-1D}). Denote by
$A_\infty=\{(0,y) : y>0\}$ the epi-graph of
$\delta_{0,0}=1_{\{0\}}$; the maximal function in $Cvx_0(K)$, and by
$A_0^1=\{(x,y): x\in K_1, y>0\}$ the epi-graph of $1_{K_1}$; the
minimal function in $Cvx_0(K_1)$ (similarly $A_0^2=\{(x,y): x\in
K_2, y>0\}$ for $Cvx_0(K_2)$). Since $Cvx_0(K) = \{f\in Cvx(\R^n) :
1_K\le f \le 1_{\{0\}}\}$, it turns out that a necessary and
sufficient condition for a bijection $F: K_1\times\R^+ \to
K_2\times\R^+$ to induce an order isomorphism is that it maps the
minimal and maximal elements in $Cvx_0(K_1)$ to the minimal and
maximal elements in $Cvx_0(K_2)$, namely:
\begin{equation}\label{Eq_Cvx0-Condition-Max}
F(A_\infty)=A_\infty,\end{equation}
\begin{equation}\label{Eq_Cvx0-Condition-Min}
F(A_0^1)=A_0^2.\end{equation} Indeed, since $F$ is a bijection
from the cylinder $K_1\times\R^+$ to the cylinder $K_2\times\R^+$,
we see that the transform is bijective. Order preservation (in both
directions) is automatic for point-map-induced transforms. One must
check that $F(epi(f))$ is an epi-graph of some convex function,
which follows from it being a convex set containing the fiber
$A_\infty$. Since $A_\infty\subseteq epi(f) \subseteq A_0^1$, we get
$A_\infty\subseteq F(epi(f)) \subseteq A_0^2$, meaning that
$F(epi(f))$ is an epi-graph of a function in $Cvx_0(K_2)$.
Therefore, we give the description of a general fractional linear
map $F$ which satisfies (\ref{Eq_Cvx0-Condition-Max}) and
(\ref{Eq_Cvx0-Condition-Min}). Let the matrix $A_F$ be given by
\[A_F = \left(
  \begin{array}{ccccc}
         &        &      &  v'_1  &  u'_1  \\
         &    A   &      & \vdots & \vdots \\
         &        &      &  v'_n  &  u'_n  \\
    v_1  & \cdots & v_n  &    a   &    b   \\
    u_1  & \cdots & u_n  &    c   &    d   \\
  \end{array} \right),\] for $A\in L_n(\R)$, $v,v',u,u'\in\R^n$, and
$a,b,c,d\in \R$. Thus $F$ is given by: \[
\left(\begin{array}{c} x\\ y\\ \end{array}\right) \mapsto
\left(\begin{array}{c}
\frac{Ax+yv'+u'}        {\iprod{u}{x}+cy+d}\\ \\
\frac{\iprod{v}{x}+ay+b}{\iprod{u}{x}+cy+d}\\ \end{array}\right)\]

Condition (\ref{Eq_Cvx0-Condition-Max}) means that
$\left(\begin{array}{c} 0\\   y \\ \end{array}\right)$ is mapped to
$\left(\begin{array}{c} 0\\ g(y)\\ \end{array}\right)$, where
$g:\R^+\to\R^+$ is some bijection, and therefore
\[\frac{yv'+u'}{cy+d}=0 \qquad \mbox{for all } y>0,\] which
implies $v'=u'=0$. For $g(y)=\frac{ay+b}{cy+d}$ to be a bijection
there exist only two options, corresponding to the two types of
transforms on $Cvx_0(K)$: either $g$ is increasing, and then $g(y)
=\frac{y}{d}$ for some $d>0$, which is associated with the $\I$-type
transforms, or $g$ is decreasing, and then $g(y)=\frac{b}{y}$ for
some $b>0$, which is associated with the $\J$-type transforms. We
denote these two different cases by $F^\I$ and $F^\J$, and (using
the multiplicative degree of freedom in $A_F$) get:
\[A_{F^\I} = \left(
  \begin{array}{ccccc}
         &        &      &    0   &    0   \\
         &    A   &      & \vdots & \vdots \\
         &        &      &    0   &    0   \\
    v_1  & \cdots & v_n  &    1   &    0   \\
    u_1  & \cdots & u_n  &    0   &    d   \\
  \end{array} \right),\qquad
  A_{F^\J} = \left(
  \begin{array}{ccccc}
         &        &      &    0   &    0   \\
         &    A   &      & \vdots & \vdots \\
         &        &      &    0   &    0   \\
    v_1  & \cdots & v_n  &    0   &    b   \\
    u_1  & \cdots & u_n  &    1   &    0   \\
  \end{array} \right)\]
Note that in both cases, $A_F\in GL_{n+2}\Leftrightarrow A\in GL_n$.
Turning to condition (\ref{Eq_Cvx0-Condition-Min}), we separate the
two cases, dealing first with the $\I$-type.

This case is very similar to the $Cvx(K)$ case (see discussion in
Section \ref{Sect_Which-FL-Work-Cvx}), where the preservation of
infinite cylinders is replaced with preservation of a part of those
cylinders. Since \[ F^\I\left(\begin{array}{c} x\\ y\\
\end{array}\right) = \left(\begin{array}{c}
\frac{Ax}            {\iprod{u}{x}+d}\\ \\
\frac{\iprod{v}{x}+y}{\iprod{u}{x}+d}\\ \end{array}\right)\] we have
that  $\frac{\iprod{v}{x}+y} {\iprod{u}{x}+d}>0$ for all $x\in K_1$
and $y>0$, which implies $K_1\subseteq\{\iprod{u}{x}+d>0\}$. Since
$\frac{\iprod{v}{x}+y}{\iprod{u}{x}+d}$ maps $\R^+$ to $\R^+$ (as a
function of $y$), $\iprod{v}{x}=0$ for all $x\in K_1$, which implies
$v=0$ (recall that $K_1$ has interior). The general form of an
$\I$-type inducing map, is thus given, for $A\in GL_n$, $u\in\R^n$,
and $d>0$, such that $K_1\subseteq\{\iprod{u}{x}+d>0\}$ by
\[
F^\I \left(\begin{array}{c} x\\ y\\
\end{array}\right) = \left(\begin{array}{c}
\frac{Ax}{\iprod{u}{x}+d}\\ \\
\frac{ y}{\iprod{u}{x}+d}\\ \end{array}\right).\]

For the $\J$-type case, we know that
\[ F^\J\left(\begin{array}{c} x\\ y\\
\end{array}\right) = \left(\begin{array}{c}
\frac{Ax}            {\iprod{u}{x}+y}\\ \\
\frac{\iprod{v}{x}+b}{\iprod{u}{x}+y}\\ \end{array}\right).\]
Therefore $\frac{\iprod{v}{x}+b} {\iprod{u}{x}+y}>0$ for all $x\in
K_1$ and $y>0$, which implies $K_1\subseteq\{\iprod{u}{x}\ge0\}$,
and also $K_1\subseteq\{\iprod{-v}{x}\le b\}$.

In the image, we know that each fiber $\{(x_2, y)\}$ above a point
$x_2\in K_2$ must contain all positive $y$. The fiber above
$\frac{Ax_0}{\iprod{u}{x_0}+y_0}$ is given by
$\frac{t\iprod{v}{x_0}+b}{t(\iprod{u}{x_0}+y_0)}$, and is the image
of the ray $(tx_0, ty_0)$ in $K_1\times\R^+$, which may be bounded
or not. If it is bounded, say of the form $[(0,0), (x_0, y_0)]$, we
must have $\iprod{v}{x_0}=-b$. If it is not bounded, we must have
$\iprod{v}{x_0}=0$. Therefore we handle the following cases
separately:

\noindent{\bf A cone $K_1$:} In this case, all rays $(tx_0, ty_0)$
in $K_1\times\R^+$ are not bounded, therefore all directions $x_0$
in $K_1$ satisfy $\iprod{v}{x_0}=0$, and therefore $v=0$, since
$K_1$ has interior.

\noindent{\bf Bounded $K_1$:} In this case, all rays $(tx_0, ty_0)$
in $K_1\times\R^+$ are bounded, therefore all rays in $K_1$
emanating from the origin have end points in the hyperplane
$\{\iprod{v}{\cdot} = -b\}$ (in particular, $v\neq0$). This means
that $K_1$ is a truncated cone, i.e. $K_1=\K_1\cap S_1$, where
$\K_1$ is the minimal cone containing $K_1$ and $S_1$ is the slab
$\{0\le \iprod{-v}{\cdot} \le b\}$.

\noindent{\bf General $K_1$:} In this case, some rays $(tx_0, ty_0)$
in $K_1\times\R^+$ are bounded, which implies $v\neq0$. All non
bounded directions $x_0$ must satisfy as before $\iprod{v}{x_0}=0$,
which implies that $K_1$ is bounded in directions $x_0\not\in
v^\perp$, and $K_1\cap v^\perp$ is a (degenerate) cone. As in the
bounded case, we get $K_1=\K_1\cap S_1$.

We can sum up the above three options as follows. The set $K_1$ is
the intersection of some cone, with the (possibly degenerate; if
$v=0$) slab $\{0\le \iprod{-v}{\cdot}\le b\}$.

Similarly, since in every direction, $K_2$ is given by $\{tAx:\quad
0\le t\le \iprod{u}{x}^{-1}\}$ for some $x\in K_1$ (if $\iprod{u}{x}
=0$ we let $\iprod{u}{x}^{-1}=\infty$), it contains full rays in all
directions $A(u^\perp)$, and in other directions it contains
intervals with end points $x_0$ which satisfy $\iprod{A^{-T}u}{x_0}
=1$. This implies, as before, that $K_2$ is some cone, intersected
with the (possibly degenerate) slab $\{0\le \iprod{A^{-T}u}{\cdot}
\le 1\}$.

One can further investigate the possible restrictions on $v, u, A$ 
with respect to the bodies $K_i$, but it involves considering
different cases for $K_1$ and $K_2$. We do not go into this in
detail but instead give a few examples.

\begin{rem}\label{Exm_0-Int-K1} Under the condition $0\in int(K_1)$,
the only $\J$-type order isomorphism $\T: Cvx_0(K_1)\to Cvx_0(K_2)$
is possible when $K_1=K_2=\R^n$. Indeed, in that case $K_1\subseteq
\{\iprod{u}{x}\ge0\}$ implies $u=0$, which means that the projection
of $F^\J$ to the first $n$ coordinates is $ P_n  F^\J  \left(
\begin{array}{c} x\\ y\\ \end{array}\right) =\frac{1}{y}Ax$.
Clearly, since $A$ is invertible, and $K_1$ contains $0$ in the
interior, this means $K_2=\R^n$, and by the exact same argument also
$K_1=\R^n$. In addition, $K_1\subseteq \{\iprod{v}{x}+b\ge0\}$
implies $v=0$, thus the general form of $F^\J$ which induces a
transform on $Cvx_0(\R^n)$ is \[F^\J\left(\begin{array}{c} x\\ y\\
\end{array}\right) = \left(\begin{array}{c} \frac{Ax}{y}\\ \\
\frac{b}{y}\\ \end{array}\right)\] for $A\in GL_n$ and $b>0$, as
stated in Theorem \ref{Thm_Cvx0-n>1-Old}.
\end{rem}

\begin{exm}\label{Exm_u-Not-Zero} When $u\neq0$, the defining
hyperplane of $F^\J$ intersects the cylinder $K_1\times\R$, and the
defining hyperplane of the image intersects the cylinder $K_2\times
\R$. This is a restriction on the bodies $K_i$; since $K_2$ is the
intersection of a cylinder (which has its base on the defining
hyperplane of the image) with a half space ($y>0$), we have (see
Section \ref{Sect_Geom_Desc}) that $K_1\times\R^+$ is the
intersection of a cone (emanating from the origin) with a half
space. This is true also for $K_2\times\R^+$, and so both our bodies
are simultaneously the intersection of a half space with a cylinder
and the intersection of a half space with a cone. For example, let
$K_1=(\R^+)^n$ and $K_2=\mbox{conv}\left\{0,e_1,\dots,e_n\right\}$,
and let $F^\J$ be given by \[A_{F^\J} = \left(
  \begin{array}{ccccc}
        &        &     &    0   &    0   \\
        &   I_n  &     & \vdots & \vdots \\
        &        &     &    0   &    0   \\
    0   & \cdots &  0  &    0   &    b   \\
    1   & \cdots &  1  &    1   &    0   \\
  \end{array} \right).\]
\end{exm}

\begin{exm}\label{Exm_v-Not-Zero} Let $K=K_1=K_2$ be the slab
$\{0\le x_1\le 1\}$. The following matrix induces a mapping $F^\J:
K\times\R^+\to K\times\R^+$:
\[A_{F^\J} = \left( \begin{array}{ccccc}
 &        & &    0   &    0   \\
 &   I_n  & & \vdots & \vdots \\
 &        & &    0   &    0   \\
 & -e_1^T & &    0   &    1   \\
 &  e_1^T & &    1   &    0   \\ \end{array} \right),\] which
induces an order isomorphism on $Cvx_0(K)$.
\end{exm}

\subsection{Generalized geometric convex functions}\label{Sect_CvxTK}
\subsubsection{Introduction}
\begin{defn} Let $n\ge2$, and let $T\subset K$ be two closed convex
sets. The subclass of $Cvx(\R^n)$ consisting of functions above
$1_K$ and below $1_T$ will be denoted $Cvx_T(K)$, that is
\[Cvx_T(K) := \{f\in Cvx(\R^n): 1_K\le f\le 1_T\}.\]
\end{defn}
%

\noindent{\bf Remarks

\noindent 1.} In the case $n=1$ this definition would still make
sense, but it does not really generalize the case of $T=\{0\}$.
Indeed, $Cvx_T(K)$ is isomorphic to $Cvx_0([0,1])$ if $K\setminus T$
is connected, and to $Cvx_0([-1,1])$ otherwise.

\noindent{\bf 2.} When $T=\emptyset$, this is the case of convex
functions on a window, $Cvx(K)$.

\noindent{\bf 3.} When $T=\{0\}$, this is the case of geometric
convex functions on a window, $Cvx_0(K)$.

\noindent{\bf 4.} Throughout this section we will assume that $K$ is of
dimension $n$, and that the interior of $K\setminus T$ is connected.

\begin{defn} Let $\T:Cvx_{T_1}(K_1)\to Cvx_{T_2}(K_2)$ be an order
preserving isomorphism, and $F:K_1\times\R^+\to K_2\times\R^+$ a
fractional linear map such that $epi(Tf)=F(epi(f))$ for every $f\in
Cvx_{T_1}(K_1)$. The transform $\T$ and the map $F$ are said to be
of $\I$-type in two cases: the first, if $F$ is linear map, and the
second, if $F$ is a non affine fractional linear map with its
defining hyperplane containing a ray in the $\R^+$ direction.
Otherwise, $\T$ and $F$ are said to be of $\J$-type.

Note that this definition coincides with that of the particular case
$Cvx_0(K)$, given in the previous section, \ref{Sect_Which-FL-Work-Cvx0}.
\end{defn}

\begin{defn} Let $\T:Cvx_{T_1}(K_1)\to Cvx_{T_2}(K_2)$ be an order
reversing isomorphism. We say that $\T$ is of $\A$-type if the
composition $\T\circ\A$ is an order preserving isomorphism of
$\I$-type, otherwise we say $\T$ is of $\L$-type.
\end{defn}

We deal with order isomorphisms from $Cvx_{T_1}(K_1)$ to
$Cvx_{T_2}(K_2)$. We show that order preserving isomorphisms are
induced by fractional linear point maps on $K_1\times\R^+$, which
are always of $\I$-type. We show that up to a composition with such
transforms, the only order reversing isomorphism is the geometric
duality $\A$. It may be formulated for order preserving or for order
reversing transforms:

\begin{thm}\label{Thm_CvxTK-Preserving}
Let $n\ge2$, and let $T_1\subset K_1\subset\R^n, T_2\subset K_2
\subset\R^n$ be four non empty, convex, compact sets, and assume
that $int(K_i)\neq\emptyset$, and that $int(K_1\setminus T_1)$ is
connected. If $\T:Cvx_{T_1}(K_1)\to Cvx_{T_2}(K_2)$ is an order
preserving isomorphism, then there exists a fractional linear map
$F:K_1\times\R^+\to K_2\times\R^+$ such that for every $f\in
Cvx_{T_1}(K_1)$, we have \[epi(\T f)=F(epi(f)).\] Moreover, $F$ is
of $\I$-type, and in particular $T_2$ is a fractional linear image
of $T_1$, and $K_2$ is a fractional linear image of $K_1$.
\end{thm}

\begin{thm}\label{Thm_CvxTK-Reversing}
Let $n\ge2$, let $K\subset T\subseteq\R^n$ be two convex sets such
that $0\in int(K)$, and assume that $K$ does not contain a full
line, and that $int(T\setminus K)$ is connected. Let $T'\subset K'
\subset\R^n$ be two non empty, convex, compact sets, and assume that
$int(K')\neq\emptyset$. If $\T: Cvx_K(T)\to Cvx_{T'}(K')$ is an
order reversing isomorphism, then 
$\T$ is of $\A$-type. In particular, $T',K'$ are fractional linear
images of $T^\circ,K^\circ$ respectively.
\end{thm}

\noindent{\bf Proof of Theorem \ref{Thm_CvxTK-Reversing}.} The
composition $\tilde{\T}:=\T\circ\A: Cvx_{T^\circ}(K^\circ)\to
Cvx_{T'}(K')$ is an order preserving isomorphism, and the
assumptions on $T$ and $K$ imply that $T, K, T', K'$ satisfy the
conditions of Theorem \ref{Thm_CvxTK-Preserving}. Indeed, $T^\circ,
K^\circ$ are non empty convex sets, and $0\in int(K)$ implies that
they are compact. Since $K$ does not contain a full line, $K^\circ$
is not contained in any hyperplane, thus it has non empty interior.
It is easy to check that for two convex sets $A\subset B$,
$int(B\setminus A)$ is connected if and only if
$int(A^\circ\setminus B^\circ)$ is connected, thus we may apply
Theorem \ref{Thm_CvxTK-Preserving}. We get that $\tilde{\T}$ is
induced by some fractional linear map $F: T^\circ\times\R^+\to
T'\times\R^+$, which is of $\I$-type, thus $\T$ is of $\A$-type, as
desired. \qed

For the proof of Theorem \ref{Thm_CvxTK-Preserving}, we first need to
define and characterize extremal elements in the class $Cvx_T(K)$.
Then we show that extremal elements are mapped to such, which will
imply that the transform induces a point map on a subset of
$\R^{n+1}$. We show that this point map is 
interval preserving for a sufficiently large set of intervals, in
order to use a theorem of Shiffman \cite{Shiffman}, which states that the map is
fractional linear. Finally we show that under our assumptions, the
transform is of $\I$-type, thus completing the proof of Theorem
\ref{Thm_CvxTK-Preserving}. We will need the following notations
throughout this section.

\begin{itemize}
\item Let $n\geq 2$, and let $A\subset B\subseteq\R^n$ be two closed
convex sets. We denote \[\kn(A,B)=\{K\subseteq\R^n: K\mbox{ is
closed, convex, and } A\subseteq K\subseteq B\}.\] For $\kn(
\emptyset,\R^n)$ we simply write $\K^n$. Note that if $T\neq
\emptyset$, any element in $\K^{n+1}(epi(1_T),epi(1_K))$ is an
epi-graph of some function $f\in Cvx_T(K)$.

\item For the convex hull of two sets $A$ and $B$ we write \[A\vee B=
\bigcap_{K\in\K^n, (A\cup B)\subseteq K} K.\]
\end{itemize}

\subsubsection{Extremal elements}\label{Sect_Extrm2}
\begin{defn} A set $K\in \kn(A,B)$ is called {\em extremal } if $
\quad\forall T,P \in \kn(A,B)$: \[K=T\vee P\qquad\Longrightarrow
\qquad K=T, \quad\text{ or }\quad K=P.\]
\end{defn}
\begin{defn} A function $f\in Cvx_T(K)$ is called {\em extremal } if
$\quad\forall g,h \in Cvx_T(K)$: \[f=\hat{\inf}\{h,g\}\qquad
\Longrightarrow\qquad f=h, \quad\text{ or }\quad f=g.\] Another
formulation of which is: \[epi(f)=epi(h)\vee
epi(g)\qquad\Longrightarrow\qquad f=h, \quad\text{ or }\quad f=g,\]
which (in the case $T\neq\emptyset$), means that $epi(f)$ is
extremal in $\K^{n+1}(epi(1_T),epi(1_K))$.
\end{defn}

Recall that for bijective transforms, order-preservation in both
directions is equivalent to preservation of the lattice operations 
$\hat{\inf}$ and $\sup$ (see Proposition \ref{Prop_Sup-To-Sup} and
Remark \ref{Rem_Lattice-Is-Necessary-For-Cvx-Inj}). Since the
extremality property is defined by the $\hat{\inf}$ operation, all
extremal elements in the domain are mapped to all extremal elements
in the range.

In the next few lemmas we investigate extremal elements of
$Cvx_T(K)$. We need the following simple observation.

\begin{lem}\label{Lem_Like-Compact}
Let $\varphi:\R^n\to\R$ be an affine linear functional and $K\subset
\R^n$ a closed, convex set that does not contain a ray on which
$\varphi$ is constant. If $\varphi(K)>0$, then there exists some $c\in
\R$ such that $\varphi(K)\ge c>0$.
\end{lem}

\noindent{\bf Proof.} Consider the slab $S=\varphi^{-1}([0,1])$. If
the intersection $K\cap S$ is empty then we may take $c=1$. Assume
otherwise, then $K\cap S$ is a closed convex set, and moreover, it
is bounded. Indeed, the slab $S$ contains only rays on which
$\varphi$ is constant, and $K$ contains no such rays, therefore
$K\cap S$ contains no rays, and one can easily verify that for a
convex set this is equivalent to boundedness. Since $K\cap S$ is
compact and $\varphi$ is continuous, there exists $x_0 \in K$ such
that $\varphi(K)\ge\varphi(x_0) \equiv c>0$. \qed

\begin{lem}\label{Lem_Extremal-Property-Epi-Graphs} Let $n\ge2$, and
let $T\subset K\subset\R^n$ be two non empty, compact, convex sets.
Consider the subsets $A=T\times\R^+$, $B=K\times\R^+$ of $\,\R^n
\times\R=\R^{n+1}$. If $K\in\K^{n+1}(A,B)$ is extremal, then
$K=A \vee \{x\}$, for some $x\in B$.
\end{lem}

\begin{proof}[{\bf Proof of Lemma \ref{Lem_Extremal-Property-Epi-Graphs}.}]
Let $K\in\K^{n+1}(A,B)$ be extremal. By a Krein-Milman type theorem
for non compact sets, see \cite{Klee1}, $K$ is the convex hull of
its extreme points and extreme rays. Since the only rays in $K$ are
translates of $\{0\}\times\R^+$, and any extreme ray must emanate
from an extreme point, $K$ is the convex hull of $A$ and its extreme
points. Finally, since the set of exposed points is dense in the set
of extreme points, see \cite{Straszewicz}, if we denote by $E$ the
set of {\em exposed} extreme points of $K$ which are not in $A$, we
have $K=A\vee E$ (actually in \cite{Straszewicz} this is proved for
compact convex sets, but the non compact case follows as an
immediate consequence, and also appears in a more general setting of
normed spaces in \cite{Klee2}, as Theorem 2.3).

Let $x_1\in E$, and let
$\varphi_1$ be an affine functional such that
$\varphi_1(K\setminus \{x_1\})>0$ and $\varphi_1(x_1)=0$. Note that
$\varphi_1$ cannot be constant on translates of $\{0\}\times\R^+$,
since then it would be constant $0$ on the translate of
$\{0\}\times\R^+$ emanating from $x_1$, contradicting strict
positivity on $K\setminus\{x_1\}$. If $E\subseteq A\vee\{x_1\}$, the
proof is complete. Assume otherwise; that there exists $x_2 \in E
\setminus (A\vee\{x_1\})$. We may separate $x_2$ from the closed set
$A\vee\{x_1\}$ by an affine functional $\varphi_2$ such that
$\varphi_2(A\vee\{x_1\})>0$ and $\varphi_2(x_2)<0$.
Denote by $H_2^-$ the (closed) half space on which $\varphi_2\le0$
and by $H_2^+$ the (closed) half space on which $\varphi_2\ge0$.
Consider the sets $K^+=A\vee (E\cap H_2^+)$, $K^-=A\vee (E\cap H_2^-
)$. Clearly $K^i\in\K^{n+1}(A,B)$ and $K=K^+\vee K^-$, thus by
extremality of $K$ we must have either $K=K^+$ or $K=K^-$. Since
both $A$ and $E\cap H_2^+$ are contained in $H_2^+$, so is $K^+$,
thus $x_2\not\in K^+$. This implies $K\neq K^+$, i.e. $K=K^-$. We
next show that $x_1\not\in K^-$, which leads to the wanted
contradiction. To this end we claim that $\varphi_1(K^-)>0=\varphi_1
(x_1)$. Indeed, $\varphi_1(A)>0$, and the only rays contained in $A$
are translates of $\{0\}\times\R^+$, on which $\varphi_1$ is not
constant. Thus, by Lemma \ref{Lem_Like-Compact}, there exists some
constant $c$ such that $\varphi_1(A)\ge c>0$. Similarly
$\varphi_1(K\cap H_2^-)\ge c'>0$. For the convex hull we get
$\varphi_1(K^-)\ge\min\{c,c'\}>0$, so $x_1\not\in K^-$.
\end{proof}

The following is a simpler version of Lemma
\ref{Lem_Extremal-Property-Epi-Graphs}, which we do not use in this
paper but add it to complete the picture.

\begin{lem}\label{Lem_Extremal-Property-Sets} Let $n\geq 2$, and let
$A\subset B\subset\R^n$ be two compact convex sets. If $K\in
\kn(A,B)$ is extremal, then $K=A\vee \{x\}$, for some $x\in B$.
\end{lem}
We omit the proof, as it is contained in the proof of the previous
lemma (the use of Lemma \ref{Lem_Like-Compact} is replaced by a
straightforward compactness argument).

A reformulation of Lemma \ref{Lem_Extremal-Property-Epi-Graphs} is:

\begin{lem}\label{Lem_Extremal-Property-Functions} Let $n\ge2$, and
let $T\subset K\subset\R^n$ be two non empty, compact, convex sets.
If $f\in Cvx_T(K)$ is extremal, then either:
\begin{itemize}
\item $f=1_T$, or:
\item $f=\hat{\inf}\{1_T, \delta_{k,h}\}$ for some $k\in
K\setminus T$ and $h\ge0$.
\end{itemize}
\end{lem}
\noindent{\bf Proof of Lemma \ref{Lem_Extremal-Property-Functions}.}
By Lemma \ref{Lem_Extremal-Property-Epi-Graphs}, $epi(f)=epi(1_T)
\vee\{x\}$ for some $x\in epi(1_K)$. If $x\in epi(1_T)$, then
$f=1_T$. If $x\not\in epi(1_T)$, then $f=\hat{\inf}\{1_T,
\delta_{k,h}\}$ for some $k,h$ as stated above.
\qed 

\subsubsection{The point map}\label{Sect_The-Point-Map}
So far we have seen that an order isomorphism $\T:Cvx_{T_1}(K_1)\to
Cvx_{T_2}(K_2)$ is in particular a bijection between the extremal
families. Clearly $\T(1_{T_1})=1_{T_2}$. Aside of the maximal
element $1_{T_1}$, each extremal function in $Cvx_{T_1}(K_1)$
corresponds to a point in $\R^{n+1}$, thus $\T$ induces a bijective
point map $F:(K_1\setminus T_1)\times\R^+\to (K_2\setminus T_2)
\times\R^+$.
%
%
Denote by $E_1$ the interior of the set $(K_1\setminus T_1)\times
\R^+$ (by our assumption, it is connected). The sets
$(K_i\setminus T_i)\times\R^+$ inherit the partial order structure
of $Cvx_{T_i}(K_i)$, after restriction to the set of extremal
elements, and the bijective map $F:(K_1\setminus T_1)\times\R^+ \to
(K_2\setminus T_2)\times\R^+$ is an order isomorphism. In the
following lemmas we will use the fact that the injective map
$F|_{E_1}:E_1\to(K_2\setminus T_2)\times\R^+$ is an order
isomorphism on its image, to prove 
that for {\em some} intervals $[a,b]\subset E_1$, $F([a,b])$ is
again an interval (these can be characterized as the ones that,
extended to a full line, do not intersect $epi(1_{T_1})$). Since the
use of the uniqueness Theorem \ref{Thm_FL-Unique-nD} requires the
preservation of {\em all} intervals, we apply a result by Shiffman
from \cite{Shiffman}, which roughly states that if a set of points
is covered by an open set of intervals which are all mapped to
intervals, then the inducing map is fractional linear. More
precisely, denote by $\L(\R^n)$ the set of all lines in $\R^n$, not
necessarily intersecting the origin. It may be seen as a subset of
the Grassmannian $G_{n+1,2}$, therefore it is equipped with the
usual inherited metric topology (for some details see Remark 4
below). Denoting by $\L(U)\subseteq\L(\R^n)$ the set of all such
lines intersecting a given set $U\subseteq\R^n$, we have

\begin{thm}\label{Thm_Shiffman}\cite{Shiffman}
Let $n\ge2$, let $U$ be an open connected set in $\R^n$, and let
$\L_0$ be an open subset of $\L(U)$, which covers $U$, i.e. $U
\subseteq\cup_{l\in\L_0}l$. Assume that $F:U\to\R^n$ is a continuous
injective map, and that $F(l\cap U)$ is contained in a line for all
$l\in\L_0$. Then $F$ is fractional linear.
\end{thm}

\noindent{\bf Remarks

\noindent 1.} Theorem \ref{Thm_Shiffman} is adjusted to the real,
linear, setting (i.e. when $U$ is a subset of $\R^n$, which is
embedded in $\R P^n$), and is a particular case of the more general
statement Shiffman proves in \cite{Shiffman}. The general result
applies for subsets of $\R P^n$ or $\C P^n$, and states that the map
$F$ is projective linear.

\noindent{\bf 2.} In \cite{Shiffman}, Theorem \ref{Thm_Shiffman} is
proved for $\R P^n$ and $\C P^n$ simultaneously. However,
considering only the case of $\R P^n$, one may check (by following
the proof in \cite{Shiffman}), that in this case continuity is
actually not required, and may be replaced by the following weaker
condition; if $I\subset U$ is an interval and $I\subset l\in\L_0$,
then $F(I)$ is again an interval. We will use this stronger version
of Theorem \ref{Thm_Shiffman}.
%

\noindent{\bf 3.} In our setting, we have epi-graphs of functions in
$Cvx_T(K)$, therefore we apply Theorem \ref{Thm_Shiffman} to the
function $F$ defined on the set $U=E_1\subset\R^{n+1}$.

\noindent{\bf 4.} A line in $\L(\R^n)$ is determined by its closest
point to the origin and its direction. That
is, for every $l\in\L(\R^n)$ let $x_l\in l$ be the unique point
satisfying $|x_l|=\min\{|x|: x\in l\}$, and let $u_l\in\{x_l\}
^\perp$ be one of the two points satisfying $l=\{x_l+tu_l,t\in\R\}$,
$|u_l|=1$ (the other being $-u_l$). Note that directions in $\{x_l\}
^\perp$ correspond to $S^{n-2}$ if $x_l\neq0$, and to $S^{n-1}$ if
$x_l=0$. Denoting the line $l$ by the pair $(x_l,u_l)$, we get a
correspondence between $\L(\R^n)\subset G_{n+1,2}$ and
$\left((\R^n\setminus\{0\})\times S^{n-2}\right)\,\bigcup\,
\left(\{0\}\times S^{n-1}\right)$, which is 1-1, modulo the $\pm$
choice in the direction $u$. The metric $d$ on $\L(\R^n)$ is
inherited from that on $G_{n+1,2}$, and it follows that
$d((x,u_1),(x,u_2))=|u_1-u_2|$, and that $d((x_1,u),(x_2,u))=
d(\hat{(x_1,1)},\hat{(x_2,1)})$, where $\hat{x}=\frac{x}{|x|}$.\\
A neighborhood of $(x,u)$ is therefore constructed by perturbing
simultaneously $x$ and $u$.
It can be checked that such a perturbation contains the following
``cylinder'' of lines; fix a point $z\in(x,u)$, let $M>0$, let $a,b
\in(x,u)$ satisfy $|a-z|=|b-z|=M$, and let $A,B$ be open balls of
radius $1/M$ and centers $a,b$ respectively. We take our
``cylinder'' of lines to be $\L_{l,z,M}:=\L(A)\cap\L(B)$. For every
$z\in l$ and every $M>0$, there exists a small perturbation of
$l=(x,u)$ which is contained in $\L_{l,z,M}$. More precisely, there
exists $\varepsilon>0$ such that $G_{l,\varepsilon}=
\{(y,v):|y-x|+|v-u|<\varepsilon\}\subset\L_{l,z,M}$. This fact is
useful in the proof of Lemma \ref{Lem_Lines-Family-Is-Rich}.

Let $\tilde{\L_0}:=\L(E_1)\setminus \L(T_1\times\R^+)$, that is,
$\tilde{\L_0}$ is the set of lines through $E_1$ (the domain of
$F$), which do not intersect the inner half cylinder $T_1\times
\R^+$. In Lemma \ref{Lem_Lines-Family-Is-Rich} we prove that the
interior of $\tilde{\L_0}$, denoted $\L_0$, is an open subset of
$\L(E_1)$ which covers $E_1$, and in Lemma
\ref{Lem_Lines-Family-Is-Good} we prove that $F(l\cap E_1)$ is
contained in a line for all $l\in\L_0$, and that intervals which are
segments of lines in $\L_0$ are mapped to intervals.
%
%

\begin{lem}\label{Lem_Lines-Family-Is-Rich} The open set $\L_0=int
(\tilde{\L_0})$ described above, covers $E_1$. That is, \[E_1
\subseteq\bigcup_{l\in\L_0} l.\]
\end{lem}
\noindent {\bf Proof of Lemma \ref{Lem_Lines-Family-Is-Rich}.} Let
$x\in E_1$. We may separate $x$ from the closed set $T_1\times\R^+$
by a hyperplane. Denote by $H$ the translate of this hyperplane
containing $x$. We claim that if $l\subset H$ is a line containing
$x$, which is not parallel to the ray $\{0\}\times\R^+$, then
$l\in\L_0$. Indeed, it is clear that $l\in\tilde{\L_0}$. Consider
the set of lines $\L_{l,x,M}$, for some $M>0$ (see the last remark).
It is an open neighborhood of $l$, and since $T_1$ is compact and
$l$ is not parallel to $\{0\}\times\R^+$, we have (for large enough
$M$) that $\L_M\subset\tilde{\L_0}$, thus $l\in\L_0$. This implies
$x\in\bigcup_{l\in\L_0} l$, and hence $\L_0$ covers $E_1$. \qed

Next we prove that the set $\L_0$ consists exactly of all the lines
in $\L(E_1)$, with the property that points along these lines are
non comparable.
\begin{lem}\label{Lem_Lines-Family-Is-Non-Comp}
Let $a,b\in E_1$ be two different points, and let $l_{a,b}$ be the
line containing $a$ and $b$. Then $l_{a,b}\not\in\L_0$ if and only
if $a$ and $b$ are comparable.
\end{lem}
\noindent {\bf Proof of Lemma \ref{Lem_Lines-Family-Is-Non-Comp}.}
The point $a$ is ``greater'' than the point $b$, if and only if
$a\in epi(1_{T_1})\vee\{b\}$, therefore $a$ and $b$ are comparable
if and only if $l_{a,b}$ is in the closure of $\L(\{b\})\cap\L(T_1
\times\R^+)\subset\L(E_1)\setminus\tilde{\L_0}$ (in fact, the
closure is only necessary if $a$ and $b$ are on the same translate
of $\{0\}\times\R^+$). This closure does not intersect $\L_0$, the
interior of $\tilde{\L_0}$, therefore we have shown:\[a,b \mbox{ are
comparable}\qquad\Rightarrow\qquad l_{a,b}\not\in\L_0.\]

If $l_{a,b}\not\in\L_0$ there are two cases. First assume $l_{a,b}
\not\in\tilde{\L_0}$ (that is, $l_{a,b}$ intersects $T_1\times
\R^+$). Thus $a$ and $b$ are comparable (one is in the convex hull
of the other and $epi(1_{T_1})$). Otherwise, assume $l_{a,b}\in
\tilde{\L_0}$. Since it is not in the interior $\L_0$, we get
$l_{a,b}\in\partial\tilde{\L_0}$. Since $\L(E_1)\setminus\bigcup_
{x\in K_1\setminus T_1}\{x\}\times\R$ is open, and $\L(T_1\times
\R^+)$ is closed, $\L_0=int(\tilde{\L_0})$ must contain
$\left(\L(E_1)
\setminus \bigcup_{x\in K_1\setminus T_1}\{\{x\}\times\R\}
\right)
\setminus \L(T_1\times\R^+)$, and therefore $\L_0=\tilde{\L_0}
\setminus\bigcup_{x\in K_1\setminus T_1}\{x\}\times\R.$ Thus
$l_{a,b}\in\partial\tilde{\L_0}$ implies that $l_{a,b}$ is parallel
to the ray $\{0\}\times\R$, and that $a$ and $b$ are comparable.
Thus we have shown: \[l_{a,b}\not\in\L_0\qquad\Rightarrow\qquad a,b
\mbox{ are comparable}.\] \qed

\begin{lem}\label{Lem_Lines-Family-Is-Good} If $l\in\L_0$, then
$F(l\cap E_1)$ is contained in a line. Moreover, if $I\subset E_1$
is an interval and $I\subset l\in\L_0$, then $F(I)$ is again an
interval.
\end{lem}
\noindent {\bf Proof of Lemma \ref{Lem_Lines-Family-Is-Good}.} The
intersection of every $l\in\L_0$ with the convex set $K_1\times\R^+$
is either a ray or an interval. Since $l\in\L_0$, it does not
intersect $T_1\times\R^+$, and therefore also $l\cap E_1$ is either
a ray or an interval. Thus, by Lemma \ref{Lem_Lines-Family-Is-Non-Comp},
it is enough to show that for every two non comparable points
$a,b\in E_1$, we have $F([a,b])=[F(a),F(b)]$. Denote for every $x\in
E_1$ by $\delta_x$ the function with epi-graph $(T_1\times\R^+)\vee
\{x\}$. Of all the extremal functions $\delta_x$, only those
corresponding to $x\in[a,b]$ have the following minimality property:
$\delta_x\ge\hat{\inf}\{\delta_a, \delta_b\}$, and for every $y$
with $\delta_y\ge\hat{\inf}\{\delta_a, \delta_b\}$, we have
$\delta_y\not<\delta_x$. This property is preserved by $F$,
therefore the interval $[a,b]$ is mapped to the interval
$[F(a),F(b)]$. \qed

\noindent {\bf Proof of Theorem \ref{Thm_CvxTK-Preserving}.} The set
$E_1$ is open and connected. Therefore, by Lemmas 
\ref{Lem_Lines-Family-Is-Rich} and \ref{Lem_Lines-Family-Is-Good},
we may apply Theorem \ref{Thm_Shiffman} (see Remark 2 after Theorem
\ref{Thm_Shiffman}) to the map $F|_{E_1}$, and conclude it is
fractional linear. To see that $F:(K_1\setminus T_1)\times\R^+\to
(K_2\setminus T_2)\times\R^+$ is fractional linear, note that a
point in the boundary of $(K_1\setminus T_1)\times\R^+$ is the
infimum of all the points below it which are in $E_1$. To see that
$F$ induces the transform $\T:Cvx_{T_1}(K_1)\to Cvx_{T_2}(K_2)$,
note that the epi-graph of a function $f\in Cvx_{T_i}(K_i)$
corresponds to the set of extremal functions above it, and that $f$
is given as the infimum of those extremal functions. Finally we need
to show that $F$ is of $\I$-type, that is, assuming $F$ is a non
affine fractional linear map, we need to show that the defining
hyperplane is parallel to the $\R^+$ direction. If it is not, then
by Section \ref{Sect_Geom_Desc}, the half cylinder $K_1\times\R^+$
is mapped to some cone, which must be $K_2\times\R^+$. But since
$K_2$ is compact, $K_2\times\R^+$ is not a cone. Therefore the map
$F$ is either affine, or it is non affine, but with a defining
hyperplane containing the direction of the epi-graphs (the ray
$\{0\}\times\R^+$). \qed

\end{document}